\documentclass[a4paper,11pt,pdf,reqno]{amsart}
\usepackage{enumerate, amsmath, amsfonts, amssymb, amsthm, thmtools, wasysym, graphics, graphicx, xcolor, frcursive,comment,bbm}
\usepackage[colorlinks=true,citecolor=cyan,backref=none]{hyperref}

\usepackage{lscape}
\usepackage{tikz-cd}

\makeatletter
\newcommand{\thickhline}{%
    \noalign {\ifnum 0=`}\fi \hrule height 1pt
    \futurelet \reserved@a \@xhline
}

\definecolor{darkblue}{rgb}{0.0,0,0.7} 
 
\definecolor{darkred}{rgb}{0.7,0,0} 
\usepackage{hyperref}
\usepackage[all]{xy}
\usepackage{mathtools}
\usepackage[utf8]{inputenc}
\usepackage{utfsym}
\usepackage[T1]{fontenc}
\usepackage{stmaryrd}
\usepackage{adjustbox}

\usepackage{vmargin}            
\setmarginsrb{3.5cm}{2.2cm}{3.5cm}{2.2cm}{0cm}{0.6cm}{0cm}{-0.1cm}

\usepackage{caption,lipsum}
\usepackage{graphicx}                  
\usepackage{pstricks,pst-plot,pst-text,pst-tree,pst-eps,pst-fill,pst-node,pst-math}
\usepackage{setspace}
\usepackage{multicol}

\usepackage{yhmath}
\hypersetup{
    citecolor=magenta,
    colorlinks=true,
    linkcolor=blue,
    filecolor=cyan,      
    urlcolor=magenta,
}

\numberwithin{equation}{section}

\newcommand{\Z}{\mathbb{Z}}

\newcommand{\R}{\mathbb{R}}
\newcommand{\llangle}{\langle\langle}
\newcommand{\rrangle}{\rangle\rangle}
\newcommand{\N}{\mathbb{N}}
\newcommand{\C}{\mathbb{C}}
\newcommand{\B}{\mathcal{B}}
\newcommand{\im}{\text{Im}}

\newcommand{\M}{\mathcal{M}}
\newcommand{\D}{\mathcal{D}}

\newcommand{\into}{\hookrightarrow}

\newcommand{\cube}[3]{\theta(\theta(#1,#2),\theta(#1,#3))}

\newcommand{\xto}{\xrightarrow}

\newcommand{\col}[6]{\left\{\!\!\!\left\{\begin{pmatrix}#1\\#2\end{pmatrix},\begin{pmatrix}#3\\#4\end{pmatrix},\begin{pmatrix}#5\\#6\end{pmatrix}\right\}\!\!\!\right\}}
\theoremstyle{plain}

\newtheorem{theorem}{Theorem}[section]
\newtheorem{lemma}[theorem]{Lemma}
\newtheorem{proposition}[theorem]{Proposition}
\newtheorem{corollary}[theorem]{Corollary}

\theoremstyle{definition}

\newtheorem{definition}[theorem]{Definition}
\newtheorem{example}[theorem]{Example}
\newtheorem{remark}[theorem]{Remark}

\newtheorem{notation}[theorem]{Notation}

\title[Braid groups of $J$-reflection groups and associated Garside structures]{Braid groups of $J$-reflection groups and associated classical and dual Garside structures}

\author{Igor Haladjian}
\address{Institut Denis Poisson, CNRS UMR 7019, Faculté des Sciences et Techniques, Université de Tours, Parc de Grandmont, 
37200 TOURS, France}
\DeclareRobustCommand{\SkipTocEntry}[5]{}
\begin{document}

%\tableofcontents
\thispagestyle{empty}

\begin{abstract}
The family of $J$-reflection groups can be seen as a combinatorial generalisation of irreducible rank two complex reflection groups and was introduced by the author in a previous article. In this article, we define the braid groups associated to $J$-reflection groups, which coincide with the complex braid group when the $J$-reflection group is finite. We show that the isomorphism types of the braid groups only depend on the reflection isomorphism types of the corresponding $J$-reflection groups. Moreover, we show that these braid groups are always abstractly isomorphic to circular groups. At the same time, we show that the center of these braid groups is cyclic and sent onto the center of the corresponding $J$-reflection groups under the natural quotient. Finally, we exhibit two Garside structures for each braid group of $J$-reflection group. These structures generalise the classical and dual Garside structures (when defined) of rank two irreducible complex reflection groups. In particular, the dual Garside structure of $J$-reflection groups provides candidates for dual monoids associated to all irreducible complex reflection groups of rank two, which coincide with their dual monoids in the sense of Bessis and Gobet when it is defined.
\end{abstract}
\maketitle

\tableofcontents
\section{Introduction}
Coxeter groups are groups admitting a presentation of the form 
\begin{equation}\label{CoxDef}
\left\langle\begin{array}{c|c} S\, &\,(st)^{\mathrm{m}_{s,t}}=1, \forall s,t\in S
\end{array}\right\rangle
\end{equation}
with $\mathrm{m}_{s,t}\in \N^*\cup\{\infty\}$, such that  $\mathrm{m}_{s,s}=1$ and $\mathrm{m}_{s,t}=\mathrm{m}_{t,s}\geq 2$ for all $s\neq t\in S$, with the convention that $\mathrm{m}_{s,t}=\infty$ means that there are no relations between $s$ and $t$. For such a presentation $P$, the associated Artin group is the group with presentation
\begin{equation}\label{ArtinDef}
\left\langle\begin{array}{c|c} S\, &\,\underbrace {sts\cdots}_{m_{s,t}\, \text{factors}}=\underbrace{tst\cdots}_{m_{s,t} \, \text{factors}}, \forall s\neq t\in S, m_{s,t}\neq \infty
\end{array}\right\rangle.
\end{equation}
Coxeter groups are linked to finite real reflection groups via the following theorem of Coxeter: 
\begin{theorem}[\cite{CoxeterClassification2} and \cite{CoxeterClassification}]
A group is a finite Coxeter group if and only if it is a finite real reflection group.
\end{theorem}
The most famous example among finite Coxeter groups
is the symmetric group $S_{n}$, which can be realized as a reflection group over $\R$ and admits an associated braid group, the so-called Artin braid group $B_n$. This situation is generalised in the following manner: for any finite real reflection group $W\subset \mathrm{GL}(V)$ with $V$ a finite-dimensional $\R$-vector space, one defines its braid group $\B(W)$ as $\pi_1(V^{\mathrm{reg}}/W)$, where $V^\mathrm{reg}$ is the complement in $\C\otimes_\R V$ of the union of the hyperplanes fixed by at least one reflection of $W$. It turns out that given a finite real reflection group with Presentation \eqref{CoxDef}, the group $\B(W)$ is isomorphic to the associated Artin group with Presentation \eqref{ArtinDef} (see \cite{BrieskornF}). 

The braid group of finite complex reflection groups is defined in the same way as in the real case. In \cite{BMR}, Broué, Malle and Rouquier give presentations of all complex reflection groups, which we refer to as the BMR presentations. In almost all cases, the BMR presentations turns out to have a similar link to the complex braid group as Coxeter presentations of real reflection groups have to their associated Artin groups:

\begin{theorem}[\cite{Bannai},\cite{BessisKP},\cite{BessisMichel},\cite{BrieskornF},\cite{BMR} and \cite{Garnier31}]\label{BraidTorsion}
Let $W$ be a complex reflection group different from $G_{24},G_{27},G_{33}$ or $G_{34}$. The group presentation obtained by removing the torsion relations from its BMR presentation is a group presentation for $\B_\C(W)$, and the obtained generators in $\B_\C(W)$ are braid reflections. Moreover, the center of $\B_\C(W)$ is cyclic and is sent onto the center of $W$ under the natural quotient $\B_\C(W)\to W$. 
\end{theorem}

\begin{remark}\label{RemarkBraidTorsionIntro}
It was shown in \cite{BessisKP} that the same result as Theorem \ref{BraidTorsion} holds for the complex braid groups of $G_{24},G_{27},G_{33}$ and $G_{34}$ with the group presentations given in \cite{BessisMichel}.
\end{remark}

\begin{remark}
The proof of Theorem \ref{BraidTorsion} and Remark \ref{RemarkBraidTorsionIntro} is uniform for real reflection groups, whereas it is a case by case verification for the complex reflection groups that cannot be realised as real reflection groups.
\end{remark}

In \cite{AA}, Achar and Aubert introduced a family of groups called $J$-groups which generalizes the rank two complex reflection groups. The $J$-groups are defined by generators and relations:
\begin{definition}[{\cite[Introduction]{AA}}]
Let $k,n,m\in \N_{\geq 2}$. The group $J(k,n,m)$ is defined by the following presentation:
\begin{equation*}
\left\langle\begin{array}{c|c}
s,t,u\, &\, s^k=t^n=u^m=1; \, stu=tus=ust
\end{array}\right\rangle.
\end{equation*}
More generally, given pairwise coprime integers $k',n',m'\in \N^*$ such that $x'$ divides $x$ for each $x\in \{k,n,m\}$, the group $J\begin{pmatrix}k & n & m\\ k' & n' & m'\end{pmatrix}$ is defined as the normal closure $\llangle s^{k'},t^{n'},u^{m'}\rrangle_{J(k,n,m)}$ of $\{s^{k'},t^{n'},u^{m'}\}$ in $J(k,n,m)$.
\end{definition}
In \cite{AA}, Achar and Aubert completely classify the finite $J$-groups, which turn out to be the rank two complex reflection groups (see Theorem \ref{FiniteJGroup} below). In \cite{Gobet Toric}, Gobet studies a family of (non-necessarily finite) $J$-groups which coincides with the family of groups that he calls toric reflection groups (\cite[Theorem 2.12]{Gobet Toric}). He also defines the notion of reflections for $J$-groups (see \cite[Definition 2.2]{Gobet Toric}) and classifies toric reflection groups up to reflection isomorphisms. Finally, he introduces (see Definition \ref{BraidToric}) an associated braid group for toric reflection groups in the same spirit as Theorem \ref{BraidTorsion}, and proves that the image of its (cyclic) center under the natural quotient is the center of the toric reflection group.\\

In \cite{VCRG}, the author defined the family of $J$-reflection groups as the $J$-groups of the form $W_b^c(k,bn,cm):=J\begin{pmatrix} k & bn & cm \\ 1 & n & m\end{pmatrix}$. The family of $J$-reflection groups includes all rank two complex reflection groups, as well as all toric reflection groups (the latter corresponding to the $J$-reflection groups with parameters $b$ and $c$ equal to 1). A presentation where all generators are reflections is given for all $J$-reflection groups (see \cite[Theorem 2.29]{VCRG}), which coincides with the BMR presentation when the groups are finite (see \cite[Remark 2.44]{VCRG}). Moreover, the center of $J$-reflection groups is shown to be cyclic (see \cite[Theorem 2.52]{VCRG}) and the $J$-reflection groups are classified up to reflection isomorphisms (see \cite[Theorem 3.11] {VCRG}).\\

In the first part of this article we define braid groups associated to $J$-reflection groups by removing the torsion relations of the presentations obtained in \cite[Theorem 2.29]{VCRG}. For the convenience of the reader, we write the presentation here:
\begin{definition}\label{DefBraidIntro}
Let $n,m\in \N^*$ be two coprime integers and let $m=qn+r$ with $0\leq q$ and $0\leq r\leq n-1$. Let $\B_*^*(n,m)$ be the group defined by the following presentation:
\begin{equation*}
\begin{aligned}
&(1) \,\, \mathrm{Generators}\!:\,  \{x_1,\dots,x_n,y,z\};\\
&(2) \,\, \mathrm{Relations}\!: \,\\
  &  x_1\cdots x_nyz=zx_1\cdots x_ny,\\
       & x_{i+1}\cdots x_nyz\delta^{q-1}x_1\cdots x_{i+r}=x_i\cdots x_nyz\delta^{q-1}x_1\cdots x_{i+r-1}, \, \forall 1\leq i \leq n-r,\\
      &  x_{i+1}\cdots x_nyz\delta^qx_1\cdots x_{i+r-n}=x_i\cdots x_nyz\delta^qx_1\cdots x_{i+r-n-1},\, \forall n-r+1\leq i \leq n,
\end{aligned}
\end{equation*}
where $\delta$ denotes $x_1\cdots x_ny$.\\
For $k,b,c\in \N_{\geq 2}$ we define this group to be \textbf{the braid group associated to} $W_b^c(k,bn,cm)$ and when useful, we also refer to it as $\mathcal B(W_b^c(k,bn,cm))$.\\
If $m\geq 2$, the group $\B_*(n,m)$ is defined as $\B_*^*(n,m)/\llangle z\rrangle$.
%which gives the following presentation:
%\begin{equation*}
%\begin{aligned}
%&(1) \,\, \mathrm{Generators}\!:\,  \{x_1,\dots,x_n,y\};\\
%&(2) \,\, \mathrm{Relations}\!: \,\\
%       & x_{i+1}\cdots x_ny\delta^{q-1}x_1\cdots x_{i+r}=x_i\cdots x_ny\delta^{q-1}x_1\cdots x_{i+r-1}, \, \forall 1\leq i \leq n-r,\\
 %     &  x_{i+1}\cdots x_ny\delta^qx_1\cdots x_{i+r-n}=x_i\cdots x_ny\delta^qx_1\cdots x_{i+r-n-1},\, \forall n-r+1\leq i \leq n.
%\end{aligned}
%\end{equation*}
%where $\delta$ denotes $x_1\cdots x_ny$.\\
For $k,b\geq 2$, we define this group to be \textbf{the braid group associated to} $W_b^1(k,bn,m)$ and when useful, we also refer to it as $\B (W_b^1(k,bn,m))$.\\
If $n\geq 2$, the group $\B^*(n,m)$ is defined as $\B^*_*(n,m)/\llangle y \rrangle$.
%which gives the following presentation:
%\begin{equation*}
%\begin{aligned}
%&(1) \,\, \mathrm{Generators}\!:\,  \{x_1,\dots,x_n,z\};\\
%&(2) \,\, \mathrm{Relations}\!: \,\\
%  &  x_1\cdots x_nz=zx_1\cdots x_n,\\
%       & x_{i+1}\cdots x_nz\delta^{q-1}x_1\cdots x_{i+r}=x_i\cdots x_nz\delta^{q-1}x_1\cdots x_{i+r-1}, \, \forall 1\leq i \leq n-r,\\
 %     &  x_{i+1}\cdots x_nz\delta^qx_1\cdots x_{i+r-n}=x_i\cdots x_nz\delta^qx_1\cdots x_{i+r-n-1},\, \forall n-r+1\leq i \leq n.
%\end{aligned}
%\end{equation*}
%where $\delta$ denotes $x_1\cdots x_n$.\\
For $k,c\geq 2$, we define this group to be \textbf{the braid group associated to} $W^c_1(k,n,cm)$ and when useful, we also refer to it as $\B (W^c_1(k,n,cm))$.\\
For $n,m$ two coprime integers larger than or equal to $2$, the group $\B(n,m)$ is defined as $\B_*^*(n,m)/\llangle y,z\rrangle$. For $k,n,m\geq 2$, the group $\B(n,m)$ is defined to be \textbf{the braid group associated to} $W_1^1(k,n,m)$. Note that $\B(n,m)$ is already defined as $\B(W_1^1(k,n,m))$ in \cite[Introduction]{Gobet Toric}.
\end{definition}
\begin{remark}
Combining \cite[Remark 2.44]{VCRG} with Theorem \ref{BraidTorsion}, we immediately get that whenever $W_b^c(k,bn,cm)$ is finite, the given presentation of its associated braid group as a $J$-reflection group coincides with the BMR presentation of its complex braid group. Therefore, the generators of the aforementioned presentation generalize braid reflections in complex braid groups, and we shall refer to them as such.
\end{remark}

We show that if two $J$-reflection groups are isomorphic in reflection, their braid groups are isomorphic. Moreover, we explicitly compute their (abstract) isomorphism type (which depends on the number of conjugacy classes of reflecting hyperplanes, see Definition \ref{DefReflections}) and identify their center:
\begin{theorem}[Theorems \ref{BraidG33} and \ref{braid2nm}]\label{ThmBraidIntro1}
Let $W$ be a $J$-reflection group.\\
(i) If $W$ has one conjugacy class of reflecting hyperplanes, its braid group is isomorphic to a torus knot group.\\
(ii) If $W$ has two conjugacy classes of reflecting hyperplanes, its braid group is isomorphic to a dihedral Artin group.\\
(iii) If $W$ has three conjugacy classes of reflecting hyperplanes, its braid group is isomorphic to the circular group $G(3,3)$.\\
In case (i), the generators of the braid group correspond to meridians in the torus knot groups. In cases (ii) and (iii), even though the groups are isomorphic to torus link groups, the generators of the braid groups do not correspond to meridians.
\end{theorem}
In particular, Theorem \ref{ThmBraidIntro1} implies that the braid group of a $J$-reflection group is always a circular group, generalising what is true for irreducible complex reflection groups of rank two (see \cite[Theorems 1 and 2]{Bannai}). It also implies that the braid group of a $J$-reflection group is always Garside (see \cite[Example 5]{Origines}).
\begin{theorem}[Theorem \ref{BraidIsoType} and  Corollary \ref{CenterBraidtoFINDANAME}]\label{CenterBraidToJref}
Let $W_1$ and $W_2$ be $J$-reflection groups. The following hold:\\
(i) The center of $\B(W_1)$ is cyclic and sent onto the center of $W_1$ under the natural quotient $\B(W_1)\to W_1$.\\
(ii) If $W_1$ and $W_2$ are isomorphic in reflection, the groups $\B(W_1)$ and $\B(W_2)$ are isomorphic.
\end{theorem}

\begin{remark}[Remark \ref{TruncatedDihedral}]
Specialising to the $J$-reflection groups of the form $W(k,2,m)$ with $m$ odd or $W_n(k,n,m)$, we obtain the family of truncated Artin groups of dihedral type (see \cite[Theorem 1.1]{Gobet Toric} and \cite[Corollary 2.31]{VCRG}). In particular, Theorem \ref{CenterBraidToJref} shows that the center of a truncated Artin group of dihedral type is the image of the center of the corresponding dihedral Artin group under the quotient.
\end{remark}

In \cite{Origines}, Dehornoy and Paris treat Artin groups of spherical type as examples of Garside groups. Moreover, generalising the work of Brady for type A Coxeter groups in \cite{BradyKP} and Brady \& Watt for type B and D Coxeter groups in \cite{BradyWattKP}, Bessis constructed in \cite{DualBessis} a dual presentation for all Artin groups of spherical type, realising them as interval Garside groups. In general, the monoid corresponding to the BMR presentation for braid groups of irreducible complex reflection groups is not Garside (see, e.g \cite{Garnier31}). However, in the case of irreducible complex reflection group of rank two, it is Garside. Moreover, the braid groups associated to the groups $G(e,e,r)$ are Garside, as shown in \cite{BessisCorran} and \cite{CorranPicantin}. In \cite{NeaimeInterval}, Neaime realised the braid groups associated to the groups $G(e,e,r)$ as interval groups. In \cite{CLL}, Corran-Lee-Lee then proved that the braid groups associated to the groups $G(de,e,r)$ are quasi-Garside.

In the second part of this article, we show that the presentations of the braid groups of $J$-reflection groups given in Definition \ref{DefBraidIntro} are Garside. More precisely, for $W:=W_b^c(k,bn,cm)$ a $J$-reflection group with $n\leq m$, define $\M(W)$ to be the monoid with the same presentation as the one of $\B(W)$ given in Definition \ref{DefBraidIntro}. The following theorem holds:
\begin{theorem}[Theorem \ref{ClassicalGarside}]\label{ClassicalGarsideIntro}
The monoid $\M(W)$ is Garside. 
\end{theorem}
\begin{remark}
Theorem \ref{ClassicalGarsideIntro} provides new Garside structures for the groups mentioned in Theorem \ref{ThmBraidIntro1}. Moreover, the atoms of $\M(W)$ are braid reflections under the identification $\M(W)\into \B(W)$.
\end{remark}
Finally, in \cite{DualBessis} Bessis proposed a definition of dual braid monoids for some irreducible rank two complex reflection groups, which was extended  to toric reflection groups by Gobet in \cite{Gobet Toric}. 

The last part of this article consists in giving a second Garside presentation for the braid group of every $J$-reflection group. If the $J$-reflection group is finite, this second presentation coincides with the dual presentation given in \cite{DualBessis} and \cite{Gobet Toric} whenever it is defined. Moreover, this presentation is again a reflection presentation (that is, the generators are again braid reflections). For the convenience of the reader, we write the presentations here: 
   \begin{proposition}[Proposition \ref{DualPresentationsBraidGroups}]\label{DualPresentationsBraidGroupsIntro}
Let $n,m\in \N^*$ with $n\leq m$ and assume $n\wedge m=1$. Then $\B_*^*(n,m)\cong \B_*^*(m,n)$ admits the following presentation:
\begin{equation*}
\begin{aligned}
    &(1) \,\, \mathrm{Generators}\!:\,  \{x_1,\dots,x_m,y,z_1,\dots,z_m\};\\
&(2) \,\, \mathrm{Relations}\!: \,\\
  &x_iz_{i+1}=z_ix_i , \, \forall 1\leq i\leq m-1, \, \\ &x_{m}yz_1=z_{m}x_{m}y,\\
  &z_{i+1}x_{i+1}\cdots x_{i+n}=z_ix_i\cdots x_{i+n-1} , \, \forall 1\leq i\leq m-n,\\
       & z_{i+1}x_{i+1}\cdots x_{m}yx_1\cdots x_{i+n-m}=z_ix_i\cdots x_{m}yx_1\cdots x_{i+n-m-1} , \,   \forall m-n+1\leq i\leq m-1,\\
      & yz_1x_1\cdots x_{n}=z_{m}x_{m}yx_1\cdots x_{n-1}.
\end{aligned}
\end{equation*}
If $m\geq 2$, the group $\B_*(n,m)\cong \B^*(m,n)$ admits the following presentation:
\begin{equation*}
\begin{aligned}
    &(1) \,\, \mathrm{Generators}\!:\,  \{x_1,\dots,x_m,z_1,\dots,z_m\};\\
&(2) \,\, \mathrm{Relations}\!: \,\\
  &x_iz_{i+1}=z_ix_i , \, \forall 1\leq i\leq m-1, \,\,\\
  &x_mz_1=z_mx_m,\\
  &z_{i+1}x_{i+1}\cdots x_{i+n}=z_ix_i\cdots x_{i+n-1} , \, \forall 1\leq i\leq m-1,\,\,\,\,\,\,\,\,\,\,\,\,\,\,\,\,\,\,\,\,\,\,\,\,\,\,\,\,\,\,\,\,\,\,\,\,\,\,\,\,\,\,\,\,\,\,\,\,\,\,\,\,\,\,\,\,\,\,\,\,\,\,\,\,\,\,\,\,\,
\end{aligned}
\end{equation*}
where indices are taken modulo $m$.\\\\
If $n\geq 2$, the group $\B^*(n,m)\cong \B_*(m,n)$ admits the following presentation: 
\begin{equation*}
\begin{aligned}
    &(1) \,\, \mathrm{Generators}\!:\,  \{x_1,\dots,x_m,y\};\\
&(2) \,\, \mathrm{Relations}\!: \,\\
  &x_{i+1}\cdots x_{i+n}=x_i\cdots x_{i+n-1} , \, \forall 1\leq i\leq m-n,\\
       & x_{i+1}\cdots x_{m}yx_1\cdots x_{i+n-m}=x_i\cdots x_{m}yx_1\cdots x_{i+n-m-1},  \,   \forall m-n+1\leq i\leq m-1,\,\,\,\,\,\,\,\,\,\,\,\,\\
       & yx_1\cdots x_n=x_myx_1\cdots x_{n-1}.
\end{aligned}
\end{equation*}
\end{proposition}
It turns out that the presentations given in Proposition \ref{DualPresentationsBraidGroupsIntro} are Garside as well. More precisely, for $W:=W_b^c(k,bn,cm)$ a $J$-reflection group with $n\leq m$, define $\mathcal D(W)$ to be the monoid with the same presentation as the one of $\B(W)$ given in Proposition \ref{DualPresentationsBraidGroupsIntro}. The following theorem holds:
\begin{theorem}[Theorem \ref{GarsideDual}]
The monoid $\mathcal D(W)$ is Garside. 
\end{theorem}
We refer to the monoid $\D(W)$ as the dual monoid for $W$, since it is the dual monoid associated to $W$ whenever $W$ is finite and has an associated dual monoid. In particular, the monoids $\D(W)$ provide combinatorial candidates for what could be defined as the dual monoids associated to the irreducible complex reflection groups of rank two which did not have one yet. In the following table, we  explicitly write these new presentations:
\begin{table}[hbt]
\centering
\begin{tabular}{|c|c|}
\hline
Group & Dual presentation \\
\hline
$G(cd,d,2)$& $\left\langle
       \begin{array}{l|cl}
           x_1,\dots,x_{d}     \,        & x_ix_{i+1}=x_jx_{j+1}\, \, \forall 1\leq i<j\leq d-1\\
   \,\,\,\,\,\,\,\,\,\,\,\,   y     \,   &\,    x_{d-1}x_dy=x_dyx_1=yx_1x_2

                          \end{array}
     \right\rangle$
\\
\hline
$G(2cd,2d,2)$&$\left\langle
       \begin{array}{l|cl}
           x_1,\dots,x_{d},y     \,        & x_iz_{i+1}=z_ix_i=z_{i+1}x_{i+1} \, \, \forall i\in [d-1]\\
      z_1,\dots,z_{d}     \,   &\,  x_dyz_1=yz_1x_1=z_dx_dy

                          \end{array}
     \right\rangle$
\\
\hline
$G_{13}$&$\left\langle
       \begin{array}{l|cl}
           x_1,x_2,x_3     \,        & x_1z_2=z_1x_1, \, x_2z_3=z_2x_2,\, x_3z_1=z_3x_3 \\
      z_1,z_2,z_3   \,   &\,   z_1x_1x_2=z_2x_2x_3=z_3x_3x_1

                          \end{array}
     \right\rangle$
\\
\hline
$G_{15}$& $\left\langle
       \begin{array}{l|cl}
           x_1,x_{2},y     \,        & z_1x_1=z_2x_2=x_1z_2\\
      z_1,z_{2}     \,   &\,    x_2yz_1=yz_1x_1=z_2x_2y

                          \end{array}
     \right\rangle$
\\
\hline

\end{tabular}
\caption{New dual monoids for some rank two complex reflection groups.}
\end{table}

This article is organised as follows: In section 2, we recall the necessary definitions and theorems about $J$-groups and complex reflection groups. In section 3, we define the braid groups associated to $J$-reflection groups, study their isomorphism type and show that they are well-defined up to reflection isomorphism. Moreover, we show that the center of braid groups of $J$-reflection groups is sent onto the center of the corresponding $J$-reflection group. Finally, we exhibit explicit reflection isomorphisms between braid groups associated to $J$-groups that are isomorphic in reflection. In section 4, we exhibit a second presentation for braid groups associated to $J$-reflection groups, which coincides with the dual presentation of irreducible rank two complex braid groups whenever the $J$-group is finite and the corresponding braid group already has a dual presentation defined. In section 5, we recall some prerequisites on Garside theory and prove useful lemmas to show that some homogeneous monoids are Garside. In particular, in specific cases the Garside structure is preserved by quotient monoids. Sections 6 and 7 are dedicated to proving that in the case where the two presentations of braid groups of $J$-reflection groups given in Sections 3 and 4 are positive, the corresponding monoids are Garside.
\addtocontents{toc}{\SkipTocEntry}
\section*{Acknowledgments}
This work is part of my PhD thesis, done under the supervision of Thomas Gobet and Cédric Lecouvey at Université de Tours, Institut Denis Poisson. I thank Thomas Gobet for his careful and numerous readings of all redaction stages of this article. I also thank him for his precious advice, suggestions and his answers to my questions. I also thank Cédric Lecouvey for his careful reading at numerous stages of this article and his advice. Moreover, some early stage computations were done using Magma's free online calculator.

\section{\texorpdfstring{$J$-reflection groups}{}}
$J$-reflection groups were defined by the author in \cite{VCRG} as a generalisation of rank two complex reflection groups. These groups are part of a family of groups called $J$-groups which also generalise rank two complex reflection groups, introduced by Achar \& Aubert in \cite{AA}.\\
In this section, we recall the definition of $J$-groups and $J$-reflection groups and state some useful results about them. 
\begin{definition}[{\cite[Introduction]{AA}}]\label{defJParent}
Let $k,n,m\in \N_{\geq 2}$. The group $J(k,n,m)$ is defined by the following presentation:
\begin{equation}\label{PresParent}
\left\langle\begin{array}{c|c}
s,t,u\, &\, s^k=t^n=u^m=1; \, stu=tus=ust
\end{array}\right\rangle.
\end{equation}
The elements $s,t$ and $u$ are called the \textbf{canonical generators} of $J(k,n,m)$.\\
More generally, given pairwise coprime integers $k',n',m'\in \N^*$ such that $x'$ divides $x$ for each $x\in \{k,n,m\}$, the group $H:=J\begin{pmatrix}k & n & m\\ k' & n' & m'\end{pmatrix}$ is defined as the normal closure $\llangle s^{k'},t^{n'},u^{m'}\rrangle_{J(k,n,m)}$ of $\{s^{k'},t^{n'},u^{m'}\}$ in $J(k,n,m)$. Groups of this form are called $J$-\textbf{groups}. Whenever $k',n'$ or $m'$ is equal to 1 we omit it in the notation.\\
Finally, the multiset of columns $C(H)$ is defined as $\col{k}{k'}{n}{n'}{m}{m'}$.
\end{definition}

\begin{definition}[{\cite[Definition 2.13]{VCRG}}]
Let $k,b,n,c,m\in \N^*$ with $k,bn,cm\geq 2$ and $n\wedge m=1$. The group $W_b^c(k,bn,cm)$ is defined as the $J$-group $J\begin{pmatrix} k & bn & cm \\ & n & m\end{pmatrix}$. A group of this form is called a \textbf{$J$-reflection group}.
\end{definition}

\begin{notation}
Whenever $b$ or $c$ is equal to $1$ we omit it in the notation. More precisely, the group $J\begin{pmatrix} k & n & cm \\ & n & m\end{pmatrix}$ is denoted by $W^c(k,n,cm)$, the group $J\begin{pmatrix} k & bn & m \\ & n & m\end{pmatrix}$ is denoted by $W_b(k,bn,m)$ and the group $J\begin{pmatrix} k & n & m \\ & n & m\end{pmatrix}$ is denoted by $W(k,n,m)$.
\end{notation}

\begin{remark}
The notation $W(k,n,m)$ already appears in \cite{Gobet Toric}, where these groups are shown to be isomorphic to so called toric reflection groups. In this paper, we shall call a group of the form $W(k,n,m)$ a \textbf{toric reflection group}, following conventions from \cite{Gobet Toric}.
\end{remark}

A complex reflection group $W$ is a finite subgroup of $\mathrm{GL}_n(\C)$ for some $n\in \N^*$ generated by complex reflections of $\C^n$ (i.e finite order automorphisms whose set of fixed points is a complex hyperplane). If the natural representation $W\into \mathrm{GL}_n(\C)$ is irreducible we call $W$ an irreducible complex reflection group. Moreover, the rank of $W$ is defined as the dimension of the orthogonal complement of the subspace of $\C^n$ fixed pointwise by $W$. As shown by Achar and Aubert in \cite{AA}, both $J$-groups and $J$-reflection groups generalise irreducible complex reflection groups of rank two:

\begin{theorem}[{\cite[Theorem 1.2 and Table 1]{AA}}]\label{FiniteJGroup}
A group is a finite $J$-group if and only if it is a complex reflection group of rank two. Moreover, every such group is a $J$-reflection group.
\end{theorem}

Theorem \ref{FiniteJGroup} suggests $J$-groups and $J$-reflection groups as candidates for analogs of Coxeter groups for complex reflection groups of rank two. We focus our study on $J$-reflection groups, for which group presentations by generators and relations are known.

\begin{theorem}[{\cite[Theorem 2.29]{VCRG}}]\label{GeneralPres}
Let $k,b,n,c,m\in \N^*$ be such that $k,bn,cm\geq 2$ and assume $n\wedge m=1$. Write $m=qn+r$ with $q\geq 0$ and $0\leq r\leq n-1$\footnote{Since $n\wedge m=1$, the case $r=0$ corresponds to $n=1$.}. Then $W_b^c(k,bn,cm)$ admits the following presentation:
\begin{equation}\label{GeneralPresW} 
\begin{aligned}
&(1) \,\, \mathrm{Generators}\!:\,  \{x_1,\dots,x_n,y,z\};\\
&(2) \,\, \mathrm{Relations}\!: \,\\
&x_i^k=y^b=z^c=1\, \forall i=1,\cdots,n, \\
  &  x_1\cdots x_nyz=zx_1\cdots x_ny,\\
       & x_{i+1}\cdots x_nyz\delta^{q-1}x_1\cdots x_{i+r}=x_i\cdots x_nyz\delta^{q-1}x_1\cdots x_{i+r-1}, \, \forall 1\leq i \leq n-r,\\
      &  x_{i+1}\cdots x_nyz\delta^qx_1\cdots x_{i+r-n}=x_i\cdots x_nyz\delta^qx_1\cdots x_{i+r-n-1},\, \forall n-r+1\leq i \leq n.
\end{aligned}
\end{equation}
where $\delta$ denotes $x_1\cdots x_ny$.\\
In terms of the canonical generators of $J(k,bn,cm)$, we have $x_i=t^{i-1}st^{1-i}$ for all $i\in [n]$, $y=t^n$ and $z=u^m$.
\end{theorem}

\begin{example}
In the case of $W_3^2(2,3,4)=J\begin{pmatrix} 2 & 3 & 4 \\ &  & 2 \end{pmatrix}\cong G_{15}$, Presentation \eqref{GeneralPresW} is

\begin{equation*}\left\langle
\begin{array}{l|cl}
& x_1^2=y^3=z^2=1, \\
x_1,y,z\, &\, x_1yz=zx_1y\\
\,& \,x_1yzx_1y=yzx_1yx_1
\end{array}
\right\rangle,\end{equation*}
which is the BMR presentation of $G_{15}$.
\end{example}

\begin{corollary}[{\cite[Corollary 2.38]{VCRG}}]\label{QuotientCircularToric}
Let $b,c\in \N^*$, $k,n,m\geq 2$ and assume $n\wedge m=1$. We have the following commutative square of quotients: 
\begin{equation}\label{CommSquareFINDANAME}\begin{tikzcd}
	{W_b^c(k,bn,cm)} & {W^c(k,n,cm)} \\
	{W_b(k,bn,m)} & {W(k,n,m)}
	\arrow["{y=1}", two heads, from=1-1, to=1-2]
	\arrow["{z=1}", two heads, from=1-2, to=2-2]
	\arrow["{z=1}"', two heads, from=1-1, to=2-1]
	\arrow["{y=1}"', two heads, from=2-1, to=2-2]
\end{tikzcd}\end{equation}

\end{corollary}

In \cite[Tables 1-3]{BMR} Broué, Malle and Rouquier gave presentations by generators and relations for all rank two complex reflection groups, where generators are reflections. We call such presentations the \textbf{BMR presentations} for rank two complex reflection groups.
\begin{remark}[{\cite[Remark 2.44]{VCRG}}]\label{BMRpres1}
Whenever $W_b^c(k,bn,cm)$ is finite, Presentation \eqref{GeneralPresW} agrees with the BMR presentation.
\end{remark}

Generalising what happens for irreducible complex reflection groups of rank two, the center of $J$-reflection groups is cyclic, described as follows: 

\begin{theorem}[{\cite[Theorem 2.52 (2)]{VCRG}}]\label{CenterFINDANAME}
Let $k,b,n,c,m\in \N^*$ with $k,bn,cm\geq 2$ and assume $n\wedge m=1$. With the notation of Theorem \ref{GeneralPres}, let $\Delta$ be the element $(x_1\cdots x_ny)^mz^n\in W_b^c(k,bn,cm)$. The center of $W_b^c(k,bn,cm)$ is cyclic, generated by $\Delta$.
\end{theorem}

\begin{corollary}[{\cite[Corollary 2.53]{VCRG} and \cite[Theorem 3.3 (2)]{Gobet Toric}}]\label{CenterSpecialCases}
Specialising Theorem \ref{CenterFINDANAME} when $b$ or $c$ is equal to $1$ gives the following description of the center:\\
(i) If $m\geq 2$, the center of $W_b(k,bn,m)$ is $\langle (x_1\cdots x_ny)^m\rangle$.\\
(ii) If $n\geq 2$, the center of $W^c(k,n,cm)$ is $\langle (x_1\cdots x_n)^mz^n\rangle$.\\
(iii) If $n,m\geq 2$, the center of $W(k,n,m)$ is $\langle (x_1\cdots x_n)^m\rangle$.
\end{corollary}

In \cite{Gobet Toric}, Gobet defines a notion of generalised reflection for $J$-groups and the corresponding notion of reflection isomorphism. When the $J$-group is finite, the set of its generalised reflections agrees with the set of its complex reflections. This makes $J$-reflection groups not only an algebraic generalisation of complex reflection groups of rank two, but a generalisation as (abstract) reflection groups. 

\begin{definition}[{\cite[Definition 2.2]{Gobet Toric}}]\label{reflections}
Let $H:=J\begin{pmatrix} k & n & m\\ k' & n'& m' \end{pmatrix}$ be a $J$-group. Define the set $R(H)$ of \textbf{algebraic reflections} of $H$ to be the set of conjugates in $J(k,n,m)$ of non-trivial powers of $s^{k'},t^{n'}$ or $u^{m'}$.
\end{definition}
\begin{lemma}[{\cite[(Proof of) Lemma 2.5]{Gobet Toric}}]\label{LemmaOrderIntro}
The order of $s,t$ and $u$ in $J(k,n,m)$ are respectively $k,n$ and $m$.
\end{lemma}
\begin{remark}\label{PresRef}
Combining Theorem \ref{GeneralPres} and Lemma \ref{LemmaOrderIntro}, the (non trivial\footnote{The generators $y$ and $z$ can be trivial if $b$ or $c$ are equal to 1.}) generators of Presentation \eqref{GeneralPresW} are algebraic reflections.
\end{remark}

\begin{definition}[\texorpdfstring{\cite[Definition 2.3]{Gobet Toric}}{}]\label{DefReflectionIso}
Let $H_1$, $H_2$ be two $J$-groups. A group homomorphism $H_1\xto\varphi H_2$ is a \textbf{reflection morphism} if $\varphi(R(H_1))\subset R(H_2)\cup\{1_{H_2}\}$. Moreover, the groups $H_1$ and $H_2$ are said to be \textbf{isomorphic in reflection} if there exists a group isomorphism $H_1\to H_2$ sending $R(H_1)$ onto $R(H_2)$. In this case, we write $$H_1\cong_{ref} H_2.$$
\end{definition}

\begin{remark}[{\cite[Remark 2.4]{Gobet Toric}}]\label{isopermute}
The reflection isomorphism type of $J\begin{pmatrix} k & n & m\\ k'& n' &m' \end{pmatrix}$ is invariant under column permutations, thus if two $J$-groups $H_1$ and $H_2$ are such that $C(H_1)=C(H_2)$, they are isomorphic in reflection.
\end{remark}

We describe this isomorphism explicitly for $J$-reflection groups in terms of Presentation \eqref{GeneralPresW}. Let then $k,b,n,c,m\in \N^*$ with $k,bn,cm\geq 2$ and $n\wedge m=1$. With the notation of Theorem \ref{GeneralPres}, write $a_1,\dots,a_m,p,q$ for the elements of $W_c^b(k,cm,bn)$ corresponding to $x_1,\dots,x_n,y,z$ in Presentation \eqref{GeneralPresW}. Writing $f$ for the induced isomorphism $W_b^c(k,bn,cm)\to W_c^b(k,cm,bn)$, we have the following:

\begin{proposition}[{\cite[Proposition 4.5]{VCRG}}]\label{prop1}
For all $i\in [n]$, writing $i-1=g_im+h_i$ the euclidean division of $i-1$ by $m$ we have 
\begin{equation}
f(x_i)=((a_1\cdots a_mp)^{g_i}a_1\cdots a_{h_i})a_{h_i+1}^{-1}((a_1\cdots a_mp)^{g_i}a_1\cdots a_{h_i})^{-1}.
\end{equation}

Moreover, we have
\begin{equation}
f(y)=q^{-1}, \, f(z)=p^{-1}.
\end{equation}
\end{proposition}

The braid group of a complex reflection group $W$ is defined as the fundamental group of the so called regular orbit-space of $W$ (see \cite[Section 2.B]{BMR} for a precise definition). It follows immediately from the definition that whenever two complex reflection groups are isomorphic in reflection, their braid groups are isomorphic. Thus, when we define (combinatorially) braid groups associated to $J$-reflection groups, a desirable asset is that two groups that are isomorphic in reflection share the same braid group. In order to do so, a classification of $J$-reflection groups up to reflection isomorphisms is useful, and was done in \cite{VCRG}:

\begin{theorem}[{\cite[Theorem 3.11]{VCRG}}]\label{Classification3}
The $J$-reflection groups $W_1$ and $W_2$ are isomorphic in reflection if and only if $C(W_1)=C(W_2)$.
\end{theorem}

The last concept we introduce in this subsection is that of conjugacy classes of reflecting hyperplanes. We will see in the next section that the number of conjugacy classes of reflecting hyperplanes of $J$-reflection groups partly dictates the isomorphism type of their associated braid groups.

\begin{definition}[{\cite[Definition 2.7]{Gobet Toric} and \cite[Definition 3.1]{VCRG}}]\label{DefReflections}
Let $H$ be a $J$-group. Let $\sim$ be the equivalence relation on $R(H)$ generated by the relations $r^a\sim r^b$ for all $1\leq a,b<o(r)$, $r\in R(H)$ and write $[r]$ for the equivalence class of $r$ in $R(H)$. Now, define the set $\mathcal H(H)$ of \textbf{reflecting hyperplanes} of $H$ to be $\{[r]\}_{r\in R(H)}$.\\
Finally, for $r\in R(H)$, define $[r]_\mathfrak c$ as $\cup_{h\in H}[hrh^{-1}]$ and write $\mathcal H_\mathfrak c(H)$ for $\{[r]_\mathfrak c\}_{r\in R(H)}$. The elements of $\mathcal H_\mathfrak c(H)$ are called \textbf{conjugacy classes of reflecting hyperplanes}
\end{definition}

\begin{proposition}[{\cite[Proposition 3.3 and Corollary 3.5]{VCRG}}]\label{NumberConj}
The number of conjugacy classes of reflecting hyperplanes of the $J$-reflection group $W_b^c(k,bn,cm)$ is $3-\delta_{1,b}-\delta_{1,c}$. More precisely, with the notation of Theorem \ref{GeneralPres} the set $\mathcal H_\mathfrak c(W_b^c(k,bn,cm))$ is equal to $\{[x_1]_\mathfrak c,[y]_\mathfrak c,[z]_\mathfrak c\}\backslash \{\emptyset\}$, where we define $[1]_\mathfrak c$ to be $\emptyset$. Moreover, the number of conjugacy classes of reflecting hyperplanes of $J$-reflection groups is invariant under reflection isomorphisms.
\end{proposition}

\section{\texorpdfstring{Braid groups associated to $J$-reflection groups}{}}

For a finite rank two complex reflection group $W$, \cite[Theorem 2.27]{BMR} states that its BMR presentation is such that removing the torsion relations of this presentation gives a presentation of the braid group of $W$. In this section, we define the braid group associated to a $J$-reflection group in this way, so that braid groups of finite $J$-reflection groups always coincide with their complex braid groups. Moreover, we identify the (abstract) isomorphism types of all braid groups associated to $J$-reflection groups. This allows us to generalise a pattern already visible for rank two complex reflection groups: the isomorphism type of braid groups is influenced by the number of conjugacy classes of the corresponding complex reflection groups. Moreover, the identification of their isomorphism type shows that braid groups associated to $J$-reflection groups are well defined up to reflection isomorphism type and that they are Garside groups. 

\subsection{\texorpdfstring{Definition of braid groups associated to $J$-reflection groups}{}}
In this subsection, we define the braid groups associated to $J$-reflection groups. In order to give this definition, we need to distinguish cases depending on whether $b$ or $c$ are equal to $1$. Indeed, we define braid groups of $J$-reflection groups as groups presented by generators and relations by removing the torsion relations of Presentation \eqref{GeneralPresW} (after removing from the presentation the generators equal to the identity for the cases where $1\in\{b,c\}$). 
\begin{definition}\label{BraidPresentation1}
Let $n,m\in \N^*$ be two coprime integers and write $m=qn+r$ with $0\leq q$ and $0\leq r\leq n-1$. Let $\B_*^*(n,m)$ be the group defined by the following presentation:
\begin{subequations}
\label{classicalBraidPres1} 
\begin{align}
&(1) \,\, \mathrm{Generators}\!:\,  \{x_1,\dots,x_n,y,z\};\notag\\
&(2) \,\, \mathrm{Relations}\!: \,\notag\\
  &  x_1\cdots x_nyz=zx_1\cdots x_ny,\label{classicalBraidPres1:1}\\
       & x_{i+1}\cdots x_nyz\delta^{q-1}x_1\cdots x_{i+r}=x_i\cdots x_nyz\delta^{q-1}x_1\cdots x_{i+r-1}, \, \forall 1\leq i \leq n-r,\label{classicalBraidPres1:2}\\
      &  x_{i+1}\cdots x_nyz\delta^qx_1\cdots x_{i+r-n}=x_i\cdots x_nyz\delta^qx_1\cdots x_{i+r-n-1},\, \forall n-r+1\leq i \leq n,\label{classicalBraidPres1:3}
\end{align}
\end{subequations}
where $\delta$ denotes $x_1\cdots x_ny$.\\
For $k,b,c\in \N_{\geq 2}$ we define this group to be \textbf{the braid group associated to} $W_b^c(k,bn,cm)$ and when useful, we also refer to it as $\mathcal B(W_b^c(k,bn,cm))$.
\end{definition}

\begin{definition}\label{BraidPresentation2}
Let $n,m\in \N^*$ be two coprime integers and assume $m\geq 2$. Moreover, write $m=qn+r$ with $0\leq q$ and $0\leq r\leq n-1$. Let $\B_*(n,m)$ be the group defined by the following presentation:
\begin{subequations}\label{BraidPresy1}
\begin{align}
&(1) \,\, \mathrm{Generators}\!:\,  \{x_1,\dots,x_n,y\};\notag\\
&(2) \,\, \mathrm{Relations}\!: \notag\,\\
       & x_{i+1}\cdots x_ny\delta^{q-1}x_1\cdots x_{i+r}=x_i\cdots x_ny\delta^{q-1}x_1\cdots x_{i+r-1}, \, \forall 1\leq i \leq n-r,\label{BraidPresy1:1}\\
      &  x_{i+1}\cdots x_ny\delta^qx_1\cdots x_{i+r-n}=x_i\cdots x_ny\delta^qx_1\cdots x_{i+r-n-1},\, \forall n-r+1\leq i \leq n.\label{BraidPresy1:2}
\end{align}
\end{subequations}
where $\delta$ denotes $x_1\cdots x_ny$.\\
For $k,b\geq 2$, we define this group to be \textbf{the braid group associated to} $W_b(k,bn,m)$ and when useful, we also refer to it as $\B (W_b(k,bn,m))$. 
\end{definition}

\begin{definition}\label{BraidPresentation3}
Let $n,m\in \N^*$ be two coprime integers and assume $n\geq 2$. Moreover, write $m=qn+r$ with $0\leq q$ and $0\leq r\leq n-1$. Let $\B^*(n,m)$ be the group defined by the following presentation:
\begin{subequations}\label{BraidPresz1}
\begin{align}
&(1) \,\, \mathrm{Generators}\!:\,  \{x_1,\dots,x_n,z\};\notag\\
&(2) \,\, \mathrm{Relations}\!: \,\notag\\
  &  x_1\cdots x_nz=zx_1\cdots x_n,\label{BraidPresz1:1}\\
       & x_{i+1}\cdots x_nz\delta^{q-1}x_1\cdots x_{i+r}=x_i\cdots x_nz\delta^{q-1}x_1\cdots x_{i+r-1}, \, \forall 1\leq i \leq n-r,\label{BraidPresz1:2}\\
      &  x_{i+1}\cdots x_nz\delta^qx_1\cdots x_{i+r-n}=x_i\cdots x_nz\delta^qx_1\cdots x_{i+r-n-1},\, \forall n-r+1\leq i \leq n.\label{BraidPresz1:3}
\end{align}
\end{subequations}
where $\delta$ denotes $x_1\cdots x_n$.\\
For $k,c\geq 2$, we define this group to be \textbf{the braid group associated to} $W^c(k,n,cm)$ and when useful, we also refer to it as $\B (W^c(k,n,cm))$. 
\end{definition}

\begin{definition}\label{BraidToric}
Let $n,m$ be two coprime integers larger than or equal to $2$. We define $\B(n,m)$ to be the group defined by the presentation
\begin{equation}\label{BraidPresyz1}
\langle x_1,\dots,x_n\,|\, x_i\cdots x_{i+m-1}=x_j\cdots x_{j+m-1} \, \forall 1\leq i<j\leq n\rangle,
\end{equation}
where indices are taken modulo $n$. For $k,n,m\geq 2$, the group $\B(n,m)$ is defined to be \textbf{the braid group associated to} $W(k,n,m)$. Note that $\B(n,m)$ is already defined as $\B(W(k,n,m))$ in \cite[Introduction]{Gobet Toric}.
\end{definition}

As intended from this generalisation, whenever $W_b^c(k,bn,cm)$ is finite the corresponding presentation among \eqref{classicalBraidPres1}-\eqref{BraidPresyz1} is the BMR presentation of its complex braid group (see \cite[Theorem 2.27 and Tables 1-3]{BMR}). In Subsections \ref{Section33} and \ref{Section34}, we show that the isomorphism type of the braid group associated to a given $J$-reflection group is invariant under reflection isomorphisms, using Theorem \ref{Classification3}.

\begin{remark}\label{CommutativeBraidQuotients}
Let $n,m\geq 2$ be two coprime integers. We have the following commutative square of quotients:

\begin{equation}\label{CommFINDANAMEbraids}
\xymatrix{\B_*^*(n,m) \ar[r]^{y=1}\ar[d]^{z=1} &\B^*(n,m)\ar[d]_{z=1} & \\
\B_*(n,m) \ar[r]^{y=1}& \B(n,m) & 
}\end{equation}

\end{remark}

\begin{remark}
If $n=1$ (respectively $m=1$) we still have an epimorphism 
\begin{equation}\label{QuotientBraidsNorM1}\B_*^*(n,m)\xto{z=1} \B_*(n,m)\, \text{ (respectively $\B_*^*(n,m)\xto{y=1} \B^*(n,m)$)},
\end{equation} 
but the Square \eqref{CommFINDANAMEbraids} contains non-defined groups.
\end{remark}

\subsection{\texorpdfstring{A central element in braid groups associated to $J$-reflection groups}{}}
In this subsection, we determine a central element of braid groups associated to $J$-reflection groups, which in fact generates the center of these braid groups, as is shown in Subsections \ref{Section33} and \ref{Section34}. Moreover, we exhibit a generating family of cardinality three for $\mathcal B_*^*(n,m)$. In this subsection, we use the notation of Presentations \eqref{classicalBraidPres1}-\eqref{BraidPresyz1}.

%%%%%%%%%%%%%%%%%%%%%%%%%%%%%%%%%%%%%%%%%%%%%%%%%%%%%%%%%%%%
%%%%%%%%%%%%%%%%%%%%%%%%%%%%%%%%%%%%%%%%%%%%%%%%%%%%%%%%%%%%
%%%%%%%%%%%%%%%%%%%%%%%%%%%%%%%%%%%%%%%%%%%%%%%%%%%%%%%%%%%%%%%%%%%%%%%%%%%%%%%%%%%%%%%%%%%%%%%%%%%%%%%%%%%%%%%%%%%%%%%%%%%%%%%%%%%%%%%%%%%%%%%%%%%%%%%%%%%%%%%%%%%%%%%%%%%%%%%%%%%%
%%%%%%%%%%%%%%%%%%%%%%%%%%%%%%%%%%%%%%%%%%%%%%%%%%%%%%%%%%%%
%%%%%%%%%%%%%%%%%%%%%%%%%%%%%%%%%%%%%%%%%%%%%%%%%%%%%%%%%%%%%%%%%%%%%%%%%%%%%%%%%%%%%%%%%%%%%%%%%%%%%%%%%%%%%%%%%%%%%%%%
%%%%%%%%%%%%%%%%%%%%%%%%%%%%%%%%%%%%%%%%%%%%%%%%%%%%%%%%%%%%
%%%%%%%%%%%%%%%%%%%%%%%%%%%%%%%%%%%%%%%%%%%%%%%%%%%%%%%%%%%%
%%%%%%%%%%%%%%%%%%%%%%%%%%%%%%%%%%%%%%%%%%%%%%%%%%%%%%%%%%%%%%%%%%%%%%%%%%%%%%%%%%%%%%%%%%%%%%%%%%%%%%%%%%%%%%%%%%%%%%%%%%%%%%%%%%%%%%%%%%%%%%%%%%%%%%%%%%%%%%%%%%%%%%%%%%%%%%%%%%%%
%%%%%%%%%%%%%%%%%%%%%%%%%%%%%%%%%%%%%%%%%%%%%%%%%%%%%%%%%%%%
%%%%%%%%%%%%%%%%%%%%%%%%%%%%%%%%%%%%%%%%%%%%%%%%%%%%%%%%%%%%%%%%%%%%%%%%%%%%%%%%%%%%%%%%%%%%%%%%%%%%%%%%%%%%%%%%%%%%%%%%%%%%%%%%%%%%%%%%%%%%%%%%%%%%%%%%%%%%%%%%%%%%%%%%%%%%%%%%%%%%

\begin{definition}\label{defDelta}
Let $n,m\geq 1$ be two coprime integers and write $m=qn+r$ with $0\leq q$ and $0\leq r\leq n-1$.
Define $w\in \B_*^*(n,m)$ to be the element represented by one of the words in $$\{z\delta^qx_1\cdots x_r\}\cup \{x_i\cdots x_nyz\delta^{q-1}x_1\cdots x_{i+r-1}\}_{i\in [n-r+1]},$$ which is well defined by \eqref{classicalBraidPres1:1} and \eqref{classicalBraidPres1:2}. Moreover, define $W$ to be the element $wy\in \B_*^*(n,m)$ (equivalently, using \eqref{classicalBraidPres1:3}, the element represented by one of the words in $\{x_i\cdots x_nyz\delta^qx_1\cdots x_{i+r-n-1}\}_{i\in \llbracket n-r+1,n+1\rrbracket}$).\\
Finally, define $\Delta$ to be $w^{n-r}W^r$.
\end{definition}
%%%%%%%%%%%%%%%%%%%%%%%%%%%%%%%%%%%%%%%%%%%%%%%%%%%%%%%%%%%%
%%%%%%%%%%%%%%%%%%%%%%%%%%%%%%%%%%%%%%%%%%%%%%%%%%%%%%%%%%%%%%%%%%%%%%%%%%%%%%%%%%%%%%%%%%%%%%%%%%%%%%%%%%%%%%%%%%%%%%%%%%%%%%%%%%%%%%%%%%%%%%%%%%%%%%%%%%%%%%%%%%%%%%%%%%%%%%%%%%%%
%%%%%%%%%%%%%%%%%%%%%%%%%%%%%%%%%%%%%%%%%%%%%%%%%%%%%%%%%%%%
%%%%%%%%%%%%%%%%%%%%%%%%%%%%%%%%%%%%%%%%%%%%%%%%%%%%%%%%%%%%%%%%%%%%%%%%%%%%%%%%%%%%%%%%%%%%%%%%%%%%%%%%%%%%%%%%%%%%%%%%%%%%%%%%%%%%%%%%%%%%%%%%%%%%%%%%%%%%%%%%%%%%%%%%%%%%%%%%%%%%
%%%%%%%%%%%%%%%%%%%%%%%%%%%%%%%%%%%%%%%%%%%%%%%%%%%%%%%%%%%%
%%%%%%%%%%%%%%%%%%%%%%%%%%%%%%%%%%%%%%%%%%%%%%%%%%%%%%%%%%%%%%%%%%%%%%%%%%%%%%%%%%%%%%%%%%%%%%%%%%%%%%%%%%%%%%%%%%%%%%%%%%%%%%%%%%%%%%%%%%%%%%%%%%%%%%%%%%%%%%%%%%%%%%%%%%%%%%%%%%%%
%%%%%%%%%%%%%%%%%%%%%%%%%%%%%%%%%%%%%%%%%%%%%%%%%%%%%%%%%%%%
%%%%%%%%%%%%%%%%%%%%%%%%%%%%%%%%%%%%%%%%%%%%%%%%%%%%%%%%%%%%%%%%%%%%%%%%%%%%%%%%%%%%%%%%%%%%%%%%%%%%%%%%%%%%%%%%%%%%%%%%
%%%%%%%%%%%%%%%%%%%%%%%%%%%%%%%%%%%%%%%%%%%%%%%%%%%%%%%%%%%%
%%%%%%%%%%%%%%%%%%%%%%%%%%%%%%%%%%%%%%%%%%%%%%%%%%%%%%%%%%%%
%%%%%%%%%%%%%%%%%%%%%%%%%%%%%%%%%%%%%%%%%%%%%%%%%%%%%%%%%%%%%%%%%%%%%%%%%%%%%%%%%%%%%%%%%%%%%%%%%%%%%%%%%%%%%%%%%%%%%%%%%%%%%%%%%%%%%%%%%%%%%%%%%%%%%%%%%%%%%%%%%%%%%%%%%%%%%%%%%%%%
%%%%%%%%%%%%%%%%%%%%%%%%%%%%%%%%%%%%%%%%%%%%%%%%%%%%%%%%%%%%
%%%%%%%%%%%%%%%%%%%%%%%%%%%%%%%%%%%%%%%%%%%%%%%%%%%%%%%%%%%%%%%%%%%%%%%%%%%%%%%%%%%%%%%%%%%%%%%%%%%%%%%%%%%%%%%%%%%%%%%%

The first goal of this section is to show the following result:
\begin{proposition}\label{DeltaCentral}
Let $n,m\geq 1$ be two coprime integers. The element $\Delta\in \B_*^*(n,m)$ defined as in Definition \ref{defDelta} is central.
\end{proposition}

In order to prove Proposition \ref{DeltaCentral}, we show some preliminary results: 

\begin{lemma}\label{xvw}
Let $n,m\geq 1$ be two coprime integers and $i\in [n-r]$ where $m=qn+r$ with $0\leq q$ and $0\leq r\leq n-1$. Let $w$ be as in Definition \ref{defDelta}. Then $x_iw=wx_{i+r}$. Moreover, we have $yw=wy$.
\end{lemma}

\begin{proof}
For all $i\in [n-r]$, we have
\begin{equation*}
\begin{aligned}
x_iw&\underset{\eqref{classicalBraidPres1:2}}{=}x_i(x_{i+1}\cdots x_nyz\delta^{q-1}x_1\cdots x_{i+r})\\ &=(x_i\cdots x_nyz\delta^{q-1}x_1\cdots x_{i+r-1})x_{i+r}\underset{\eqref{classicalBraidPres1:2}}=wx_{i+r}.
\end{aligned} 
\end{equation*}
Moreover, we have
\begin{equation*}
    \begin{aligned}
        yw&\underset{\eqref{classicalBraidPres1:2}}=y(x_1\cdots x_nyz\delta^{q-1}x_1\cdots x_r)\underset{\eqref{classicalBraidPres1:1}}=yz\delta^qx_1\cdots x_r\underset{\eqref{classicalBraidPres1:3}}=x_{n-r+1}\cdots x_nyz\delta^q\\&=(x_{n-r+1}\cdots x_nyz\delta^{q-1}x_1\cdots x_n)y\underset{\eqref{classicalBraidPres1:2}}=wy.
    \end{aligned}
\end{equation*}
This concludes the proof.
\end{proof}
%%%%%%%%%%%%%%%%%%%%%%%%%%%%%%%%%%%%%%%%%%%%%%%%%%%%%%%%%%%%
%%%%%%%%%%%%%%%%%%%%%%%%%%%%%%%%%%%%%%%%%%%%%%%%%%%%%%%%%%%%
%%%%%%%%%%%%%%%%%%%%%%%%%%%%%%%%%%%%%%%%%%%%%%%%%%%%%%%%%%%%%%%%%%%%%%%%%%%%%%%%%%%%%%%%%%%%%%%%%%%%%%%%%%%%%%%%%%%%%%%%%%%%%%%%%%%%%%%%%%%%%%%%%%%%%%%%%%%%%%%%%%%%%%%%%%%%%%%%%%%%
%%%%%%%%%%%%%%%%%%%%%%%%%%%%%%%%%%%%%%%%%%%%%%%%%%%%%%%%%%%%
%%%%%%%%%%%%%%%%%%%%%%%%%%%%%%%%%%%%%%%%%%%%%%%%%%%%%%%%%%%%%%%%%%%%%%%%%%%%%%%%%%%%%%%%%%%%%%%%%%%%%%%%%%%%%%%%%%%%%%%%
%%%%%%%%%%%%%%%%%%%%%%%%%%%%%%%%%%%%%%%%%%%%%%%%%%%%%%%%%%%%
%%%%%%%%%%%%%%%%%%%%%%%%%%%%%%%%%%%%%%%%%%%%%%%%%%%%%%%%%%%%
%%%%%%%%%%%%%%%%%%%%%%%%%%%%%%%%%%%%%%%%%%%%%%%%%%%%%%%%%%%%%%%%%%%%%%%%%%%%%%%%%%%%%%%%%%%%%%%%%%%%%%%%%%%%%%%%%%%%%%%%%%%%%%%%%%%%%%%%%%%%%%%%%%%%%%%%%%%%%%%%%%%%%%%%%%%%%%%%%%%%
%%%%%%%%%%%%%%%%%%%%%%%%%%%%%%%%%%%%%%%%%%%%%%%%%%%%%%%%%%%%
%%%%%%%%%%%%%%%%%%%%%%%%%%%%%%%%%%%%%%%%%%%%%%%%%%%%%%%%%%%%%%%%%%%%%%%%%%%%%%%%%%%%%%%%%%%%%%%%%%%%%%%%%%%%%%%%%%%%%%%%%%%%%%%%%%%%%%%%%%%%%%%%%%%%%%%%%%%%%%%%%%%%%%%%%%%%%%%%%%%%
%%%%%%%%%%%%%%%%%%%%%%%%%%%%%%%%%%%%%%%%%%%%%%%%%%%%%%%%%%%%
%%%%%%%%%%%%%%%%%%%%%%%%%%%%%%%%%%%%%%%%%%%%%%%%%%%%%%%%%%%%%%%%%%%%%%%%%%%%%%%%%%%%%%%%%%%%%%%%%%%%%%%%%%%%%%%%%%%%%%%%%%%%%%%%%%%%%%%%%%%%%%%%%%%%%%%%%%%%%%%%%%%%%%%%%%%%%%%%%%%%
%%%%%%%%%%%%%%%%%%%%%%%%%%%%%%%%%%%%%%%%%%%%%%%%%%%%%%%%%%%%
%%%%%%%%%%%%%%%%%%%%%%%%%%%%%%%%%%%%%%%%%%%%%%%%%%%%%%%%%%%%%%%%%%%%%%%%%%%%%%%%%%%%%%%%%%%%%%%%%%%%%%%%%%%%%%%%%%%%%%%%%%%%%%%%%%%%%%%%%%%%%%%%%%%%%%%%%%%%%%%%%%%%%%%%%%%%%%%%%%%%
\begin{lemma}\label{xvW}
Let $n,m\geq 1$ be coprime integers and $k\in [r]$ where $m=qn+r$ with $0\leq q$ and $0\leq r\leq n-1$. Let $W$ be as in Definition \ref{defDelta}. Then $x_{n-r+k}W=Wx_{k}$. Moreover, we have $yW=Wy$.
\end{lemma}

\begin{proof}
For all $k\in[r]$, we have
\begin{equation*}
    \begin{aligned}
        x_{n-r+k}W&\underset{\eqref{classicalBraidPres1:3}}=x_{n-r+k}x_{n-r+k+1}\cdots x_nyz\delta^qx_1\cdots x_{k}\\&=(x_{n-r+k}\cdots x_nyz\delta^qx_1\cdots x_{k-1})x_k\underset{\eqref{classicalBraidPres1:3}}=Wx_k.
    \end{aligned}
\end{equation*}
Moreover, by Lemma \ref{xvw} $w$ commutes with $y$, so that $W=wy$ commutes with $y$.
\end{proof}
%%%%%%%%%%%%%%%%%%%%%%%%%%%%%%%%%%%%%%%%%%%%%%%%%%%%%%%%%%%%
%%%%%%%%%%%%%%%%%%%%%%%%%%%%%%%%%%%%%%%%%%%%%%%%%%%%%%%%%%%%%%%%%%%%%%%%%%%%%%%%%%%%%%%%%%%%%%%%%%%%%%%%%%%%%%%%%%%%%%%%%%%%%%%%%%%%%%%%%%%%%%%%%%%%%%%%%%%%%%%%%%%%%%%%%%%%%%%%%%%%
%%%%%%%%%%%%%%%%%%%%%%%%%%%%%%%%%%%%%%%%%%%%%%%%%%%%%%%%%%%%
%%%%%%%%%%%%%%%%%%%%%%%%%%%%%%%%%%%%%%%%%%%%%%%%%%%%%%%%%%%%%%%%%%%%%%%%%%%%%%%%%%%%%%%%%%%%%%%%%%%%%%%%%%%%%%%%%%%%%%%%
%%%%%%%%%%%%%%%%%%%%%%%%%%%%%%%%%%%%%%%%%%%%%%%%%%%%%%%%%%%%
%%%%%%%%%%%%%%%%%%%%%%%%%%%%%%%%%%%%%%%%%%%%%%%%%%%%%%%%%%%%
\begin{remark}\label{Wr=0}
If $n=m=1$, we have $r=0$ so that the statement of Lemma \ref{xvW} is vacuously true.
\end{remark}
%%%%%%%%%%%%%%%%%%%%%%%%%%%%%%%%%%%%%%%%%%%%%%%%%%%%%%%%%%%%%%%%%%%%%%%%%%%%%%%%%%%%%%%%%%%%%%%%%%%%%%%%%%%%%%%%%%%%%%%%%%%%%%%%%%%%%%%%%%%%%%%%%%%%%%%%%%%%%%%%%%%%%%%%%%%%%%%%%%%%
%%%%%%%%%%%%%%%%%%%%%%%%%%%%%%%%%%%%%%%%%%%%%%%%%%%%%%%%%%%%
%%%%%%%%%%%%%%%%%%%%%%%%%%%%%%%%%%%%%%%%%%%%%%%%%%%%%%%%%%%%%%%%%%%%%%%%%%%%%%%%%%%%%%%%%%%%%%%%%%%%%%%%%%%%%%%%%%%%%%%%

\begin{remark}\label{remarkMod}
If $i\in [n-r]$, we can use the equality $x_iw=wx_{i+r}$ and if $i\in \llbracket n-r+1,n\rrbracket$, we can use the equality $x_iW=Wx_{i+r-n}$. In both cases, we see that the index of the $x_l$ in the right of these equalities is the unique element $j$ of $[n]$ such that $j\equiv i+r\,(\text{mod $n$})$.
\end{remark}

Remark \ref{remarkMod} motivates the following definition:

%%%%%%%%%%%%%%%%%%%%%%%%%%%%%%%%%%%%%%%%%%%%%%%%%%%%%%%%%%%%
%%%%%%%%%%%%%%%%%%%%%%%%%%%%%%%%%%%%%%%%%%%%%%%%%%%%%%%%%%%%
%%%%%%%%%%%%%%%%%%%%%%%%%%%%%%%%%%%%%%%%%%%%%%%%%%%%%%%%%%%%%%%%%%%%%%%%%%%%%%%%%%%%%%%%%%%%%%%%%%%%%%%%%%%%%%%%%%%%%%%%%%%%%%%%%%%%%%%%%%%%%%%%%%%%%%%%%%%%%%%%%%%%%%%%%%%%%%%%%%%%
%%%%%%%%%%%%%%%%%%%%%%%%%%%%%%%%%%%%%%%%%%%%%%%%%%%%%%%%%%%%
%%%%%%%%%%%%%%%%%%%%%%%%%%%%%%%%%%%%%%%%%%%%%%%%%%%%%%%%%%%%%%%%%%%%%%%%%%%%%%%%%%%%%%%%%%%%%%%%%%%%%%%%%%%%%%%%%%%%%%%%
%%%%%%%%%%%%%%%%%%%%%%%%%%%%%%%%%%%%%%%%%%%%%%%%%%%%%%%%%%%%
%%%%%%%%%%%%%%%%%%%%%%%%%%%%%%%%%%%%%%%%%%%%%%%%%%%%%%%%%%%%
%%%%%%%%%%%%%%%%%%%%%%%%%%%%%%%%%%%%%%%%%%%%%%%%%%%%%%%%%%%%%%%%%%%%%%%%%%%%%%%%%%%%%%%%%%%%%%%%%%%%%%%%%%%%%%%%%%%%%%%%%%%%%%%%%%%%%%%%%%%%%%%%%%%%%%%%%%%%%%%%%%%%%%%%%%%%%%%%%%%%
\begin{definition}
We say that $x_i$ is \textbf{integrable against} $w$ if $i\in [n-r]$ (resp. against $W$ if $i\in \llbracket n-r+1,n\rrbracket$), and in this case we say that $x_i$ \textbf{integrates} as $x_j$, where $j$ is the unique element of $[n]$ such that $j\equiv i+r \, (\text{mod}\, n)$.
\end{definition}
%%%%%%%%%%%%%%%%%%%%%%%%%%%%%%%%%%%%%%%%%%%%%%%%%%%%%%%%%%%%
%%%%%%%%%%%%%%%%%%%%%%%%%%%%%%%%%%%%%%%%%%%%%%%%%%%%%%%%%%%%%%%%%%%%%%%%%%%%%%%%%%%%%%%%%%%%%%%%%%%%%%%%%%%%%%%%%%%%%%%%%%%%%%%%%%%%%%%%%%%%%%%%%%%%%%%%%%%%%%%%%%%%%%%%%%%%%%%%%%%%
%%%%%%%%%%%%%%%%%%%%%%%%%%%%%%%%%%%%%%%%%%%%%%%%%%%%%%%%%%%%
%%%%%%%%%%%%%%%%%%%%%%%%%%%%%%%%%%%%%%%%%%%%%%%%%%%%%%%%%%%%%%%%%%%%%%%%%%%%%%%%%%%%%%%%%%%%%%%%%%%%%%%%%%%%%%%%%%%%%%%%%%%%%%%%%%%%%%%%%%%%%%%%%%%%%%%%%%%%%%%%%%%%%%%%%%%%%%%%%%%%

To be able to prove Proposition \ref{DeltaCentral}, we need one last technical lemma:
\begin{lemma}\label{Delta=Delta}
Let $n,m\geq1$ be two coprime integers and let $\Delta$ be as in Definition \ref{defDelta}. Then $\Delta=\delta^mz^n$.
\end{lemma}

\begin{proof}
In this proof, we construct a word and show that this word both represents $\Delta$ and $\delta^mz^n$ in $\mathcal B_*^*(n,m)$.
For $l\in \llbracket0,n-1\rrbracket$ we define a word $w_l\in \{x_1,\dots,x_n,y,z,\delta,\delta^{-1}\}^*$ inductively as $w_0=x_1\cdots x_nyz\delta^{q}\delta^{-1}x_1\cdots x_r$ and 
\begin{equation*}w_{l+1}=\begin{cases}
    w_lx_{j+1}\cdots x_nyz\delta^{q}\delta^{-1}x_1\cdots x_{j+r} \, \text{if $w_l$ ends by $x_j$, \,$j\in [n-r-1]$,}\\
    w_lx_{n-r+1}\cdots x_nyz\delta^q, \, \text{if $w_l$ ends with $x_{n-r}$},
    \\
    w_lx_{j+1}\cdots x_nyz\delta^qx_1\cdots x_{j+r-n} \, \text{if $w_l$ ends by $x_j$, \, $j\in \llbracket n-r+1,n-1\rrbracket$}.
\end{cases}\end{equation*}
Note that by $\eqref{classicalBraidPres1:2}$ and $\eqref{classicalBraidPres1:3}$, the element represented by $w_{l+1}$ in $\B_*^*(n,m)$ is equal to $w_lw$ if $w_l$ ends with $x_j$, $j\in [n-r-1]$ and is equal to $w_lW$ otherwise.\\
\fbox{The word $w_{n-1}$ represents $\Delta$ in $\B_*^*(n,m)$:}
Since indices are taken modulo $n$, if the last letter of $w_l$ is $x_j$ with $j\neq n-r$, the last letter of $w_{l+1}$ is $x_{j+r}$. In addition, we have that $n$ and $r$ are coprime since $n$ and $m$ are coprime, which shows that the only element $l$ of $\llbracket 0,n-2\rrbracket$ such that $(l+1)r\equiv n-r \, (\text{mod} \, n)$ is $n-2$. In particular, as the last letter of $w_0$ is $x_r$, we see that the last letter of $w_l$ is $x_{(l+1)r}$ for all $l=0,\dots,n-2$ and the last letter of $w_{n-1}$ is $\delta$. Moreover, since $n$ and $r$ are coprime, the set of last letters of the words $w_0,\dots,w_{n-2}$ is $\{x_{(l+1)r}\}_{l\in \llbracket 0,n-2\rrbracket}=\{x_1,\dots,x_{n-1}\}$.
This shows that amongst the last letter of $w_0,\dots,w_{n-2}$, each element of the set $\{x_{n-r},\dots,x_{n-1}\}$ appears exactly once. Thus, for $l\in \llbracket 1,n-2\rrbracket$ there are exactly $r$ times where $w_{l+1}$ represents $w_lW$ in $\B_*^*(n,m)$ and $n-r-1$ times where $w_{l+1}$ represents $w_lw$ in $\B_*^*(n,m)$.\\
Using Lemma \ref{xvw}, the elements $w$ and $W=wy$ commute so that $w_{n-1}$ represents $w_0w^{n-r-1}W^{r}=w^{n-r}W^r=\Delta$ since $w_0$ represents $w$ by $\eqref{classicalBraidPres1:2}$. It remains to show that $w_{n-1}$ represents $\delta^mz^n$ in $\B_*^*(n,m)$.\\
\fbox{The word $w_{n-1}$ represents $\delta^mz^n$ in $\B_*^*(n,m)$:}
Observe that whenever $w_l$ ends with $x_j$, $j\in [n-r-1]$, the word length (in terms of the alphabet $\{x_1,\dots,x_n,y,z,\delta,\delta^{-1}\}$) of $w_{l+1}$ is equal to that of $w_l$ plus $n+r+q+3$. Otherwise, it is equal to that of $w_l$ plus $r+q+2$. Since the length of $w_0$ is $n+q+r+3$, we get that $w_{n-1}$ is a word starting by $x_1$ of length 
\begin{equation}\label{Length1}
    \ell(w_{n-1})=(n+r+q+3)+(n-r-1)(n+r+q+3)+r(r+q+2)=nq+3n+n^2-r,
\end{equation}
with exactly $n$ occurrences of $z$, each of which comes after a $y$ and before a power of $\delta$. Removing from $w_{n-1}$ every instance of the letters $\delta,\delta^{-1}$ and $z$ leaves a word $\tilde w_{n-1}$ consisting of a succession of the word $x_1\cdots x_ny$. For any word $m$ and letter $x$, denote by $|m|_x$ the number of instances of $x$ in $m$. With this notation, we have
\begin{equation}\label{Length2}
|w_{n-1}|_{\delta}=q(n-r)+qr=qn,\, |w_{n-1}|_{\delta^{-1}}=n-r,\, |w_{n-1}|_z=(n-r)+r=n.
\end{equation}
Combining equations \eqref{Length1} and \eqref{Length2}, the word $\tilde w_{n-1}$ obtained from $w_{n-1}$ after removing every instances of the letters $\delta,\delta^{-1}$ and $z$ has length 
\begin{equation}\label{Length3}
nq+3n+n^2-r-qn-(n-r)-n=n(n+1).    
\end{equation}
Since $\ell(x_1\cdots x_ny)=n+1$, the word $\tilde w_{n-1}$ is a succession of $n$ copies of $x_1\cdots x_ny$. Finally, by $\eqref{classicalBraidPres1:1}$ the elements $z$ and $\delta=x_1\cdots x_ny$ commute so that the element represented by $w_{n-1}$ in $\B_*^*(n,m)$ is $\delta^{n+qn-(n-r)}z^n=\delta^{qn+r}z^n=\delta^mz^n$. This concludes the proof.
\end{proof}
%%%%%%%%%%%%%%%%%%%%%%%%%%%%%%%%%%%%%%%%%%%%%%%%%%%%%%%%%%%%
%%%%%%%%%%%%%%%%%%%%%%%%%%%%%%%%%%%%%%%%%%%%%%%%%%%%%%%%%%%%%%%%%%%%%%%%%%%%%%%%%%%%%%%%%%%%%%%%%%%%%%%%%%%%%%%%%%%%%%%%%%%%%%%%%%%%%%%%%%%%%%%%%%%%%%%%%%%%%%%%%%%%%%%%%%%%%%%%%%%%
%%%%%%%%%%%%%%%%%%%%%%%%%%%%%%%%%%%%%%%%%%%%%%%%%%%%%%%%%%%%
%%%%%%%%%%%%%%%%%%%%%%%%%%%%%%%%%%%%%%%%%%%%%%%%%%%%%%%%%%%%%%%%%%%%%%%%%%%%%%%%%%%%%%%%%%%%%%%%%%%%%%%%%%%%%%%%%%%%%%%%%%%%%%%%%%%%%%%%%%%%%%%%%%%%%%%%%%%%%%%%%%%%%%%%%%%%%%%%%%%%
%%%%%%%%%%%%%%%%%%%%%%%%%%%%%%%%%%%%%%%%%%%%%%%%%%%%%%%%%%%%
%%%%%%%%%%%%%%%%%%%%%%%%%%%%%%%%%%%%%%%%%%%%%%%%%%%%%%%%%%%%%%%%%%%%%%%%%%%%%%%%%%%%%%%%%%%%%%%%%%%%%%%%%%%%%%%%%%%%%%%%
%%%%%%%%%%%%%%%%%%%%%%%%%%%%%%%%%%%%%%%%%%%%%%%%%%%%%%%%%%%%
%%%%%%%%%%%%%%%%%%%%%%%%%%%%%%%%%%%%%%%%%%%%%%%%%%%%%%%%%%%%
%%%%%%%%%%%%%%%%%%%%%%%%%%%%%%%%%%%%%%%%%%%%%%%%%%%%%%%%%%%%%%%%%%%%%%%%%%%%%%%%%%%%%%%%%%%%%%%%%%%%%%%%%%%%%%%%%%%%%%%%%%%%%%%%%%%%%%%%%%%%%%%%%%%%%%%%%%%%%%%%%%%%%%%%%%%%%%%%%%%%
%%%%%%%%%%%%%%%%%%%%%%%%%%%%%%%%%%%%%%%%%%%%%%%%%%%%%%%%%%%%
%%%%%%%%%%%%%%%%%%%%%%%%%%%%%%%%%%%%%%%%%%%%%%%%%%%%%%%%%%%%%%%%%%%%%%%%%%%%%%%%%%%%%%%%%%%%%%%%%%%%%%%%%%%%%%%%%%%%%%%%
\begin{proof}[Proof of Proposition \ref{DeltaCentral}]
It is enough to show that $\Delta$ commutes with every element of $\{x_1,\dots,x_n,y,z\}$. First, the element $\Delta$ is a product of powers of $w$ and $W$, thus by combining Lemmas \ref{xvw} and \ref{xvW} we get that $y$ and $\Delta$ commute.
Similarly, by Lemma \ref{Delta=Delta} we see that $z$ and $\Delta$ commute since $\Delta$ is a product of powers of $z$ and $\delta$, both of which commute with $z$.\\
Let then $i\in [n]$. Integrating $x_i$ $n$ times amounts to using the equality $x_jw=wx_{j+r}$ $n-r$ times and the equality $x_jW=Wx_{j+r}$ $r$ times depending on whether $x_j$ is integrable against $w$ (which happens $n-r$ times) or $W$ (which happens $r$ times). In other words, since $w$ and $W$ commute, we have $x_iw^{n-r}W^r=w^{n-r}W^rx_{i+nr}=w^{n-r}W^rx_i$. This shows that $x_i$ commutes with $\Delta$ for all $i\in [n]$, which concludes the proof.
\end{proof}
%%%%%%%%%%%%%%%%%%%%%%%%%%%%%%%%%%%%%%%%%%%%%%%%%%%%%%%%%%%%
%%%%%%%%%%%%%%%%%%%%%%%%%%%%%%%%%%%%%%%%%%%%%%%%%%%%%%%%%%%%
%%%%%%%%%%%%%%%%%%%%%%%%%%%%%%%%%%%%%%%%%%%%%%%%%%%%%%%%%%%%%%%%%%%%%%%%%%%%%%%%%%%%%%%%%%%%%%%%%%%%%%%%%%%%%%%%%%%%%%%%%%%%%%%%%%%%%%%%%%%%%%%%%%%%%%%%%%%%%%%%%%%%%%%%%%%%%%%%%%%%
%%%%%%%%%%%%%%%%%%%%%%%%%%%%%%%%%%%%%%%%%%%%%%%%%%%%%%%%%%%%
\begin{remark}[Notation]\label{Remark/Notation}
The image of $\Delta$ under any of the quotients in the Square \eqref{CommFINDANAMEbraids} is central by Proposition \ref{DeltaCentral}. In any of the groups $\B_*(n,m)$, $\B^*(n,m)$ and $\B(n,m)$, we still call $\Delta$ the resulting element in the quotient. Moreover, using Lemma \ref{Delta=Delta} we have the following:\\
$\bullet$ In $\B_*(n,m)$, the element $\Delta$ is $(x_1\cdots x_ny)^m$.\\
$\bullet$ In $\B^*(n,m)$, the element $\Delta$ is $(x_1\cdots x_n)^mz^n$.\\
$\bullet$ In $\B(n,m)$, the element $\Delta$ is $(x_1\cdots x_n)^m$. Note that in this case, the element $\Delta$ generates the center of $\B(n,m)$ (combine \cite[Ch.3, Section C, exercise 4]{Rolfsen} and \cite[Lemma 3.1 and Remark 4.13]{Gobet Garside}).
\end{remark}
%%%%%%%%%%%%%%%%%%%%%%%%%%%%%%%%%%%%%%%%%%%%%%%%%%%%%%%%%%%%%%%%%%%%%%%%%%%%%%%%%%%%%%%%%%%%%%%%%%%%%%%%%%%%%%%%%%%%%%%%
%%%%%%%%%%%%%%%%%%%%%%%%%%%%%%%%%%%%%%%%%%%%%%%%%%%%%%%%%%%%
%%%%%%%%%%%%%%%%%%%%%%%%%%%%%%%%%%%%%%%%%%%%%%%%%%%%%%%%%%%%
%%%%%%%%%%%%%%%%%%%%%%%%%%%%%%%%%%%%%%%%%%%%%%%%%%%%%%%%%%%%%%%%%%%%%%%%%%%%%%%%%%%%%%%%%%%%%%%%%%%%%%%%%%%%%%%%%%%%%%%%%%%%%%%%%%%%%%%%%%%%%%%%%%%%%%%%%%%%%%%%%%%%%%%%%%%%%%%%%%%%
%%%%%%%%%%%%%%%%%%%%%%%%%%%%%%%%%%%%%%%%%%%%%%%%%%%%%%%%%%%%
\subsection{\texorpdfstring{Small generating sets for braid groups associated to $J$-reflection groups}{}}
The goal of this subsection is devoted to proving the following lemma:
\begin{lemma}\label{ThreeGenerate}
The three elements $x_1,x_2\cdots x_ny$ and $z$ generate $\B_*^*(n,m)$.    
\end{lemma}
%%%%%%%%%%%%%%%%%%%%%%%%%%%%%%%%%%%%%%%%%%%%%%%%%%%%%%%%%%%%%%%%%%%%%%%%%%%%%%%%%%%%%%%%%%%%%%%%%%%%%%%%%%%%%%%%%%%%%%%%%%%%%%%%%%%%%%%%%%%%%%%%%%%%%%%%%%%%%%%%%%%%%%%%%%%%%%%%%%%%
%%%%%%%%%%%%%%%%%%%%%%%%%%%%%%%%%%%%%%%%%%%%%%%%%%%%%%%%%%%%

For all $i\in[n]$ we introduce the element $X_i:=x_1\cdots \hat x_i\cdots x_ny\in \B_*^*(n,m)$, with the convention that $X_{i+n}=X_i$ for all $i\in \Z$. 

\begin{lemma}\label{GenerateXi}
The set $\{X_1,\dots,X_n,z,x_1\}\subset \B_*^*(n,m)$ generates $\B_*^*(n,m)$.
\end{lemma}

\begin{proof}
We show by induction on $i\in [n]$ that $\{x_1,\dots,x_i\}\subset \langle X_1,\dots,X_n,x_1\rangle$.\\
For $i=1$, this is trivial. Assume then that $\{x_1,\dots,x_i\}\subset\langle X_1,\dots,X_n,x_1\rangle$ for some $i\in [n-1]$. Then we have 
\begin{equation*}
    \begin{aligned}
        x_{i+1}&=(x_1\cdots x_i)^{-1}(x_1\cdots x_ny)(x_{i+2}\cdots x_ny)^{-1}(x_1\cdots x_i)^{-1}(x_1\cdots x_i)\\&=(x_1\cdots x_i)^{-1}x_1X_1X_{i+1}^{-1}(x_1\cdots x_i),
    \end{aligned}
\end{equation*}
which shows that $x_{i+1}\in \langle x_1,\dots,x_i,X_1,X_{i+1}\rangle\subset\langle X_1,\dots,X_n,x_1\rangle$ by induction hypothesis.\\
We get that $\{x_1,\dots,x_n,z\}\subset \langle X_1,\dots,X_n,z,x_1\rangle$. Since $y=(x_2\cdots x_n)^{-1}X_1$, this shows that 
\begin{equation*}\langle X_1,\dots,X_n,z,x_1\rangle=\B_*^*(n,m),\end{equation*} which concludes the proof.
\end{proof}

%%%%%%%%%%%%%%%%%%%%%%%%%%%%%%%%%%%%%%%%%%%%%%%%%%%%%%%%%%%%%%%%%%%%%%%%%%%%%%%%%%%%%%%%%%%%%%%%%%%%%%%%%%%%%%%%%%%%%%%%%%%
%%%%%%%%%%%%%%%%%%%%%%%%%%%%%%%%%%%%%%%%%%%%%%%%%%%%%%%%%%%%
%%%%%%%%%%%%%%%%%%%%%%%%%%%%%%%%%%%%%%%%%%%%%%%%%%%%%%%%%%%%%%%%%%%%%%%%%%%%%%%%%%%%%%%%%%%%%%%%%%%%%%%%%%%%%%%%%%%%%%%%%%%%%%%%%%%%%%%%%%%%%%%%%%%%%%%%%%%%%%%%%%%%%%%%%%%%%%%%%%%%
%%%%%%%%%%%%%%%%%%%%%%%%%%%%%%%%%%%%%%%%%%%%%%%%%%%%%%%%%%%%
%%%%%%%%%%%%%%%%%%%%%%%%%%%%%%%%%%%%%%%%%%%%%%%%%%%%%%%%%%%%%%%%%%%%%%%%%%%%%%%%%%%%%%%%%%%%%%%%%%%%%%%%%%%%%%%%%%%%%%%%%%%%%%%%%%%%%%%%%%%%%%%%%%%%%%%%%%%%%%%%%%%%%%%%%%%%%%%%%%%%
%%%%%%%%%%%%%%%%%%%%%%%%%%%%%%%%%%%%%%%%%%%%%%%%%%%%%%%%%%%%
\begin{proof}[Proof of Lemma \ref{ThreeGenerate}]
Using Lemma \ref{GenerateXi}, it is enough to show that $\{X_1,\dots,X_n\}\subset \langle x_1,x_2\cdots x_ny,z\rangle$ to conclude the proof.\\
First, one sees that $\delta=x_1(x_2\cdots x_ny)\in \langle x_1,x_2\cdots x_ny,z\rangle$. Now, observe that for all $i\in [n-r]$, the relation 
\begin{equation*}
    x_{i+1}\cdots x_nyz\delta^{q-1}x_1\cdots x_{i+r}=x_i\cdots x_nyz\delta^{q-1}x_1\cdots x_{i+r-1}
\end{equation*}
coming from $\eqref{classicalBraidPres1:2}$ yields 
\begin{equation*}
\begin{aligned} &(x_1\cdots x_{i-1})(x_{i+1}\cdots x_nyz\delta^{q-1}x_1\cdots x_{i+r})(x_{i+r-1}\cdots x_ny)\\ &= (x_1\cdots x_{i-1})(x_i\cdots x_nyz\delta^{q-1}x_1\cdots x_{i+r-1})(x_{i+r-1}\cdots x_ny).
\end{aligned}\end{equation*}
Using \eqref{classicalBraidPres1:1}, the right hand side is equal to $z\delta^qX_{i+r}$, hence we get $X_iz\delta^q=z\delta^qX_{i+r}$. Similarly, for $i\in \llbracket n-r+1,n\rrbracket$ the relation 
\begin{equation*}
    x_{i+1}\cdots x_nyz\delta^qx_1\cdots x_{i+r-n}=x_i\cdots x_nyz\delta^qx_1\cdots x_{i+r-n-1}
\end{equation*}
coming from \eqref{classicalBraidPres1:3} yields $X_iz\delta^{q+1}=z\delta^{q+1}X_{i+r-n}$.\\Since $X_{i+r-n}=X_{i+r}$ and $\{z\delta^q,z\delta^{q+1}\}\subset \langle x_1,x_2\cdots x_ny,z\rangle,$ we see that for all $i\in [n]$, the element $X_i$ belongs to $\langle x_1,x_2\cdots x_ny,z\rangle$ if and only if $X_{i+r}$ does. Since $n$ and $r$ are coprime, we get that $\{X_1,\dots,X_n\}\subset \langle x_1,x_2\cdots x_ny,z\rangle$ if and only if $X_1\in \langle x_1,x_2\cdots x_ny,z\rangle$, which holds since $X_1=x_2\cdots x_ny$. This concludes the proof.
\end{proof}
\subsection{\texorpdfstring{The isomorphism type of $\B_*^*(n,m)$}{}}\label{Section33}
It is known (see \cite[Theorems 1 and 2]{Bannai}) that the braid groups of rank two complex reflection groups with three conjugacy classes of reflecting hyperplanes are isomorphic to the group defined by the presentation 
\begin{equation}\label{defG33}
    \left\langle\begin{array}{c|cc}
     s,t,u \, & \, stu=tus=ust
    \end{array}\right\rangle.
\end{equation}
Following the notation of \cite{GarnierHoso}, we denote by $G(3,3)$ the group defined by Presentation \eqref{defG33}. The goal of this subsection is to show the following theorem, generalising the existing pattern for rank two complex reflection groups:

\begin{theorem}\label{BraidG33}
Let $n,m\in \N^*$ be two coprime integers. Then the group $\B_*^*(n,m)$ is abstractly isomorphic to $G(3,3)$. In particular, using Proposition \ref{NumberConj}, if a $J$-reflection group has three conjugacy classes of reflecting hyperplanes its associated braid group is isomorphic to $G(3,3)$.
\end{theorem}

A consequence of Theorems \ref{Classification3} and \ref{BraidG33} is the following theorem, which is proven at the end of this subsection:
\begin{theorem}\label{BraidIsoType}
If two $J$-reflection groups are isomorphic in reflection, their associated braid groups are isomorphic.
\end{theorem}

In the proof of Theorem \ref{BraidG33}, we exhibit an isomorphism sending $\Delta\in \B_*^*(n,m)$ to $stu\in G(3,3)$. In particular, we have the following corollary:
\begin{corollary}\label{CenterBraid3}
Let $n,m\in \N^*$ be two coprime integers. The center of $\B_*^*(n,m)$ is cyclic, generated by $\Delta$. In particular, using Theorem \ref{CenterFINDANAME}, for $k,b,c\in \N_{\geq2}$ the center of $\B_*^*(n,m)=\B(W_b^c(k,bn,cm))$ is sent onto the center of $W_b^c(k,bn,cm)$ under the natural quotient $\B(W_b^c(k,bn,cm))\to W_b^c(k,bn,cm)$.
\end{corollary}

\begin{proof}
The group $G(3,3)$ has center $stu$ (one way to see this is to observe that $G(3,3)/\langle stu\rangle$ is isomorphic to the free group on two generators, which has no center), which will show that $Z(\B_*^*(n,m))=\langle\Delta\rangle$ using the isomorphism we exhibit for the proof of Theorem \ref{BraidG33}.\\
Moreover, Theorem \ref{CenterFINDANAME} shows that the image of $\Delta$ under the natural quotient $\B(W_b^c(k,bn,cm))\to W_b^c(k,bn,cm)$ generates the center of $W_b^c(k,bn,cm)$.
\end{proof}

\begin{corollary}\label{Braid3nm}
Let $n,m\in \N^*$ be two coprime integers. The groups $\B_*^*(n,m)$ and $\B_*^*(m,n)$ are isomorphic. In particular, using Theorem \ref{Classification3} and Proposition \ref{NumberConj}, if two $J$-reflection groups with three conjugacy classes of reflecting hyperplanes are isomorphic in reflection, they have the same associated braid group.
\end{corollary}

\begin{proof}
By Theorem \ref{BraidG33} the groups $\B_*^*(n,m)$ and $\B_*^*(m,n)$ are both isomorphic to $G(3,3)$, hence they are isomorphic to each other. Moreover, Theorem \ref{Classification3} shows that two $J$-reflection groups are isomorphic in reflection if and only if their columns are the same. In particular, using Proposition \ref{NumberConj}, if two $J$-reflection groups with three conjugacy classes of reflecting hyperplanes are isomorphic in reflection their associated braid groups are amongst $\{\B_*^*(n,m),\B_*^*(m,n)\}$, which are isomorphic to each other. This concludes the proof.
\end{proof}

In order to prove Theorem \ref{BraidG33} we use the following lemma, for which the author found no reference in the literature.

\begin{definition}
Let $n,m\in \N^*$ be two coprime integers. If $m\neq 1$, define $n_{(m)}$ to be the unique integer in $[m-1]$ such that $\overline{nn_{(m)}}=1\in (\Z/m)^\times$. If $m=1$, define $n_{(m)}$ to be 1.
\end{definition}

\begin{lemma}\label{CoprimeLemma}
Let $n,m\in \N^*$ be two coprime integers. Then we have $$nn_{(m)}+mm_{(n)}=nm+1.$$
\end{lemma}

\begin{proof}
By definition, the number $nn_{(m)}-1$ is divisible by $m$ so that $nn_{(m)}+mm_{(n)}-1$ is divisible by $m$. Similarly, it is divisible by $n$ and since $n,m$ are coprime we get that $nn_{(m)}+mm_{(n)}-1$ is divisible by $nm$. \\Finally, we have $1\leq nn_{(m)},mm_{(n)}<nm$ so that $1\leq nn_{(m)}+mm_{(n)}-1<2nm$, hence $nn_{(m)}+mm_{(n)}-1=nm$ (the case $nn_{(m)}+mm_{(n)}-1=1$ corresponds to $n=m=1$). This concludes the proof.
\end{proof}
%%%%%%%%%%%%%%%%%%%%%%%%%%%%%%%%%%%%%%%%%%%%%%%%%%%%%%%%%%%%
%%%%%%%%%%%%%%%%%%%%%%%%%%%%%%%%%%%%%%%%%%%%%%%%%%%%%%%%%%%%
%%%%%%%%%%%%%%%%%%%%%%%%%%%%%%%%%%%%%%%%%%%%%%%%%%%%%%%%%%%%%%%%%%%%%%%%%%%%%%%%%%%%%%%%%%%%%%%%%%%%%%%%%%%%%%%%%%%%%%%%%%%%%%%%%%%%%%%%%%%%%%%%%%%%%%%%%%%%%%%%%%%%%%%%%%%%%%%%%%%%
%%%%%%%%%%%%%%%%%%%%%%%%%%%%%%%%%%%%%%%%%%%%%%%%%%%%%%%%%%%%
%%%%%%%%%%%%%%%%%%%%%%%%%%%%%%%%%%%%%%%%%%%%%%%%%%%%%%%%%%%%%%%%%%%%%%%%%%%%%%%%%%%%%%%%%%%%%%%%%%%%%%%%%%%%%%%%%%%%%%%%
%%%%%%%%%%%%%%%%%%%%%%%%%%%%%%%%%%%%%%%%%%%%%%%%%%%%%%%%%%%%
%%%%%%%%%%%%%%%%%%%%%%%%%%%%%%%%%%%%%%%%%%%%%%%%%%%%%%%%%%%%
%%%%%%%%%%%%%%%%%%%%%%%%%%%%%%%%%%%%%%%%%%%%%%%%%%%%%%%%%%%%%%%%%%%%%%%%%%%%%%%%%%%%%%%%%%%%%%%%%%%%%%%%%%%%%%%%%%%%%%%%%%%%%%%%%%%%%%%%%%%%%%%%%%%%%%%%%%%%%%%%%%%%%%%%%%%%%%%%%%%%
%%%%%%%%%%%%%%%%%%%%%%%%%%%%%%%%%%%%%%%%%%%%%%%%%%%%%%%%%%%%
%%%%%%%%%%%%%%%%%%%%%%%%%%%%%%%%%%%%%%%%%%%%%%%%%%%%%%%%%%%%%%%%%%%%%%%%%%%%%%%%%%%%%%%%%%%%%%%%%%%%%%%%%%%%%%%%%%%%%%%%%%%%%%%%%%%%%%%%%%%%%%%%%%%%%%%%%%%%%%%%%%%%%%%%%%%%%%%%%%%%
%%%%%%%%%%%%%%%%%%%%%%%%%%%%%%%%%%%%%%%%%%%%%%%%%%%%%%%%%%%%
%%%%%%%%%%%%%%%%%%%%%%%%%%%%%%%%%%%%%%%%%%%%%%%%%%%%%%%%%%%%%%%%%%%%%%%%%%%%%%%%%%%%%%%%%%%%%%%%%%%%%%%%%%%%%%%%%%%%%%%%%%%%%%%%%%%%%%%%%%%%%%%%%%%%%%%%%%%%%%%%%%%%%%%%%%%%%%%%%%%%

We now introduce two morphisms 
\!\!\!\begin{tikzcd}
	{\B_*^*(n,m)} & {G(3,3)}
	\arrow["\varphi", shift left, from=1-1, to=1-2]
	\arrow["\psi", shift left, from=1-2, to=1-1]
\end{tikzcd}\!\!\! and prove that they are inverse to each other, which will prove Theorem \ref{BraidG33}.
%%%%%%%%%%%%%%%%%%%%%%%%%%%%%%%%%%%%%%%%%%%%%%%%%%%%%%%%%%%%
%%%%%%%%%%%%%%%%%%%%%%%%%%%%%%%%%%%%%%%%%%%%%%%%%%%%%%%%%%%%%%%%%%%%%%%%%%%%%%%%%%%%%%%%%%%%%%%%%%%%%%%%%%%%%%%%%%%%%%%%%%%%%%%%%%%%%%%%%%%%%%%%%%%%%%%%%%%%%%%%%%%%%%%%%%%%%%%%%%%%
%%%%%%%%%%%%%%%%%%%%%%%%%%%%%%%%%%%%%%%%%%%%%%%%%%%%%%%%%%%%
%%%%%%%%%%%%%%%%%%%%%%%%%%%%%%%%%%%%%%%%%%%%%%%%%%%%%%%%%%%%%%%%%%%%%%%%%%%%%%%%%%%%%%%%%%%%%%%%%%%%%%%%%%%%%%%%%%%%%%%%%%%%%%%%%%%%%%%%%%%%%%%%%%%%%%%%%%%%%%%%%%%%%%%%%%%%%%%%%%%%

\begin{proposition}\label{BtoG(3,3)}
Let $n,m\in \N^*$ be two coprime integers and write $m=qn+r$ with $q\geq 0$ and $0\leq r\leq n-1$. The map $\{x_1,\dots,x_n,y,z\}\xto\varphi G(3,3)$ sending $x_i$ to $s^{i-1}us^{1-i}$, $y$ to $s^n(ust)^{m_{(n)}-n}$ and $z$ to $t^m(ust)^{n_{(m)}-m}$ extends to a morphism $\B_*^*(n,m)\xto\varphi G(3,3)$.
\end{proposition}

\begin{proof}
It is enough to show that the images of $x_1,\dots,x_n,y,z$ by $\varphi$ satisfy the relations of Presentation \eqref{classicalBraidPres1}. First, observe that for all $i,j\in [n]$ with $1\leq i\leq j+1$, we have 
\begin{equation}\label{eq1BtoG(3,3)}
\varphi(x_i)\cdots\varphi(x_j)=\prod_{k=i-1}^{j-1}s^kus^{-k}=s^{i-1}(us)^{j-i+1}s^{-j},
\end{equation}
where the second equality holds for $i=j+1$ since $s^{(j+1)-1}(us)^{j-(j+1)+1}s^{-j}=1$.

\noindent
\fbox{$x_1\cdots x_nyz=zx_1\cdots x_ny$:}\\
We have 
\begin{equation*}
\varphi(x_1)\cdots\varphi(x_n)\varphi(y)\underset{\eqref{eq1BtoG(3,3)}}=(us)^ns^{-n}s^n(ust)^{m_{(n)}-n} 
=(us)^n(ust)^{m_{(n)}-n}.
\end{equation*}
From now on, by $\varphi(\delta)$ we mean $\varphi(x_1)\cdots \varphi(x_n)\varphi(y)$, which by the above then reads 
\begin{equation}\label{ImDelta}
\varphi(\delta)=(us)^n(ust)^{m_{(n)}-n}.
\end{equation}
Since $ust=tus$ we have that $t$ commutes with $us$. Moreover, the element $ust$ is central so that $\varphi(\delta)$ commutes with $t$ but again since $ust$ is central, we get that $\varphi(\delta)$ commutes with $t^m(ust)^{n_{(m)}-m}=\varphi(z)$. This shows that the relation $x_1\cdots x_nyz=zx_1\cdots x_ny$ is respected by $\varphi$. \\
\noindent
\fbox{$x_{i+1}\cdots x_{n}yz\delta^{q-1}x_1\cdots x_{i+r}=x_i\cdots x_{n}yz\delta^{q-1}x_1\cdots x_{i+r-1}$, $i\in [n-r]$:}\\
We show that the image by $\varphi$ of $x_i\cdots x_nyz\delta^{q-1}x_1\cdots x_{i+r-1}$ is the same for all $i\in [n-r+1]$.\\
In the following calculation, we use the fact that $ust$ is central for establishing the third and the last equalities.\\
For all $i\in [n-r+1]$, we have
\begin{equation*}
\begin{aligned}
&\varphi(x_{i})\cdots \varphi(x_n)\varphi(y)\varphi(z)\varphi(\delta)^{q-1}\varphi(x_1)\cdots \varphi(x_{i+r-1})\\
&\underset{\eqref{eq1BtoG(3,3)}}{=}s^{i-1}(us)^{n-i+1}s^{-n}s^n(ust)^{m_{(n)}-n}t^m(ust)^{n_{(m)}-m}\varphi(\delta)^{q-1}(us)^{i+r-1}s^{1-r-i} \\
&\underset{\eqref{ImDelta}}{=}s^{i-1}(us)^{n-i+1}(ust)^{m_{(n)}-n}t^m(ust)^{n_{(m)}-m}((us)^n(ust)^{m_{(n)}-n})^{q-1}(us)^{i+r-1}s^{1-r-i}\\
&\underset{(us)t=t(us)}{=}s^{i-1}(us)^{n-i+1+(q-1)n+i+r-1}t^{m}(ust)^{m_{(n)}-n+n_{(m)}-m+(q-1)(m_{(n)}-n)}s^{1-r-i} \\
&=s^{i-1}(us)^mt^m(ust)^{m_{(n)}-n+n_{(m)}-m+(q-1)(m_{(n)}-n)}s^{1-r-i}\, \, \text{since} \, \, qn+r=m\\
&=s^{i-1}(ust)^{m_{(n)}-n+n_{(m)}-m+(q-1)(m_{(n)}-n)+m}s^{1-r-i}\, \, \text{since} \,\, (us)t=t(us)\\
&=s^{-r}(ust)^{n_{(m)}+q(m_{(n)}-n)},
\end{aligned}
\end{equation*}
which does not depend on $i$.

\noindent
\fbox{$x_{i+1}\cdots x_{n}yz\delta^{q}x_1\cdots x_{i+r-n}=x_i\cdots x_{n}yz\delta^{q}x_1\cdots x_{i+r-n-1}$, $i\in \llbracket n-r+1,n\rrbracket$:}\\
We show that the image by $\varphi$ of $x_{i+1}\cdots x_nyz\delta^qx_1\cdots x_{i+r-n}$ is the same for all $i\in \llbracket n-r,n\rrbracket$.\\
In the following calculation, we use the fact that $ust$ is central for establishing the second, third and the last equalities.\\
For all $i=n-r,\dots,n$, we have
\begin{equation*}
\begin{aligned}
&\varphi(x_{i+1})\cdots \varphi(x_n)\varphi(y)\varphi(z)\varphi(\delta)^q\varphi(x_1)\cdots \varphi(x_{i+r-n})\\
&\underset{\eqref{eq1BtoG(3,3)}}{=}s^i(us)^{n-i}s^{-n}s^n(ust)^{m_{(n)}-n}t^m(ust)^{n_{(m)}-m}\varphi(\delta)^q(us)^{i+r-n}s^{n-r-i} \\
&\underset{\eqref{ImDelta}}{=}s^i(us)^{n-i}s^{-n}s^n(ust)^{m_{(n)}-n}t^m(ust)^{n_{(m)}-m}(us)^{qn}(ust)^{q(m_{(n)}-n)}(us)^{i+r-n}s^{n-r-i} \\
&=s^i(us)^{n-i+qn+i+r-n}t^{m}(ust)^{m_{(n)}-n+n_{(m)}-m+q(m_{(n)}-n)}s^{n-r-i} \, \, \text{since} \,\, (us)t=t(us)\\
&=s^i(us)^mt^m(ust)^{m_{(n)}-n+n_{(m)}-m+q(m_{(n)}-n)}s^{n-r-i}\, \, \text{since} \, \, qn+r=m\\
&=s^i(ust)^{m_{(n)}-n+n_{(m)}-m+q(m_{(n)}-n)+m}s^{n-r-i}\, \, \text{since} \,\, (us)t=t(us)\\
&=s^{n-r}(ust)^{n_{(m)}+(q+1)(m_{(n)}-n)},
\end{aligned}
\end{equation*}
which does not depend on $i$. This concludes the proof.
\end{proof}

\begin{proposition}\label{G(3,3)toB}
Let $n,m\in \N^*$ be two coprime integers and write $m=qn+r$ with $q\geq 0$ and $0\leq r\leq n-1$. The map $\{s,t,u\}\xto\psi \mathcal B_*^*(n,m)$
sending $s$ to $x_2\cdots x_nyz^{n-m_{(n)}}\delta^{n_{(m)}-1}$, $t$ to $z^{m_{(n)}}\delta^{m-n_{(m)}}$ and $u$ to $x_1$ extends to a morphism $G(3,3)\xto\psi \B_*^*(n,m)$.
\end{proposition}

\begin{proof}
It is enough to show that the images of $s,t$ and $u$ satisfy $\psi(s)\psi(t)\psi(u)=\psi(t)\psi(u)\psi(s)=\psi(u)\psi(s)\psi(t)$.
We have
\begin{equation*}
\begin{aligned}
\psi(t)\psi(u)\psi(s)&=z^{m_{(n)}}\delta^{m-n_{(m)}}x_1x_2\cdots x_nyz^{n-m_{(n)}}\delta^{n_{(m)}-1}
\\&\underset{\delta z=z\delta}=\delta^mz^n 
=\Delta \, \, \text{by Lemma \ref{Delta=Delta}},\\
\psi(u)\psi(s)\psi(t)&=x_1x_2\cdots x_n yz^{n-m_{(n)}}\delta^{n_{(m)}-1}z^{m_{(n)}}\delta^{m-n_{(m)}}
\\&\underset{\delta z=z\delta}=\delta^mz^n =\Delta \, \, \text{by Lemma \ref{Delta=Delta}},\\
\end{aligned}
\end{equation*}
and
\begin{equation*}
    \begin{aligned}
        \psi(s)\psi(t)\psi(u)&=x_2\cdots x_n yz^{n-m_{(n)}}\delta^{n_{(m)}-1}z^{m_{(n)}}\delta^{m-n_{(m)}}x_1\\
&=x_1^{-1}x_1x_2\cdots x_n yz^{n-m_{(n)}}\delta^{n_{(m)}-1}z^{m_{(n)}}\delta^{m-n_{(m)}}x_1\\
&=x_1^{-1}\psi(u)\psi(s)\psi(t)x_1
=x_1^{-1}\Delta x_1
=\Delta \, \, \text{by Proposition \ref{DeltaCentral}}.
    \end{aligned}
\end{equation*}
This concludes the proof.
\end{proof}

We are now ready to prove Theorem \ref{BraidG33}:
\begin{proof}[Proof of Theorem \ref{BraidG33}]
Using Propositions \ref{BtoG(3,3)} and \ref{G(3,3)toB}, it is enough to show that $\varphi\circ \psi=Id_{G(3,3)}$ and that $\psi$ is surjective, which will show that $\psi$ and $\varphi$ are inverse to each other.\\\\
\fbox{$\varphi\circ\psi=Id_{G(3,3)}$:}\\\\
\underline{$\varphi\circ \psi(s)=s$:}
In the following calculation, we use the fact that $ust$ is central for establishing the third equality. We have 
\begin{equation*}
\begin{aligned}
\varphi\circ\psi(s)&=\varphi(x_2\cdots x_nyz^{n-m_{(n)}}\delta^{n_{(m)}-1})\\
\underset{\eqref{eq1BtoG(3,3)}}{=}s&(us)^{n-1}s^{-n}s^n(ust)^{m_{(n)}-n}(t^m(ust)^{n_{(m)}-m})^{n-m_{(n)}}(\varphi(\delta))^{n_{(m)}-1}\\
\underset{\eqref{ImDelta}}{=}s&(us)^{n-1}(ust)^{m_{(n)}-n}(t^m(ust)^{n_{(m)}-m})^{n-m_{(n)}}((us)^n(ust)^{m_{(n)}-n})^{n_{(m)}-1}\\
\underset{(us)t=t(us)}{=}&s(us)^{n-1+n(n_{(m)}-1)}t^{m(n-m_{(n)})}(ust)^{(m_{(n)}-n)(1-(n_{(m)}-m)+(n_{(m)}-1))}\\
=s(u&s)^{nn_{(m)}-1}t^{mn-mm_{(n)}}(ust)^{m_{(n)}m-mn}\\
=s(u&s)^{mn-mm_{(n)}}t^{mn-mm_{(n)}}(ust)^{m_{(n)}m-mn} \, \text{by Lemma \ref{CoprimeLemma}}\\
=s(u&st)^{mn-mm_{(n)}+m_{(n)}m-mn} \,\, \text{since} \,\,(us)t=t(us) \\
=s\,\,\,\,\,\,&
\end{aligned}
\end{equation*}

\noindent
\underline{$\varphi\circ\psi(t)=t$:}
In the following calculation, we use the fact that $ust$ is central for establishing the third equality.
We have
\begin{equation*}
\begin{aligned}
\varphi\circ\psi(t)&=\varphi(z^{m_{(n)}}\delta^{m-n_{(m)}})\\
&=(t^m(ust)^{n_{(m)}-m})^{m_{(n)}}((us)^n(ust)^{m_{(n)}-n})^{m-n_{(m)}}\\
&=tt^{mm_{(n)}-1}(us)^{nm-nn_{(m)}}(ust)^{(n_{(m)}-m)(m_{(n)}-(m_{(n)}-n))}\\
&=tt^{nm-nn_{(m)}}(us)^{nm-nn_{(m)}}(ust)^{(n_{(m)}-m)(m_{(n)}-(m_{(n)}-n))}\, \text{by Lemma \ref{CoprimeLemma}}\\
&=t(ust)^{n(m-n_{(m)})+n(n_{(m)}-m)}\,\, \text{since} \,\,(us)t=t(us)  \\
&=t
\end{aligned}
\end{equation*}

\noindent
\underline{$\varphi\circ\psi(u)=u$:}
We have $\varphi\circ\psi(u)=\varphi(x_1)=u$.\\
This shows that $\varphi\circ\psi=Id_{G(3,3)}$.\\

\noindent
\fbox{$\psi$ is surjective:}\\\\
By Lemma \ref{ThreeGenerate}, it is enough to show that $\{x_1,x_2\cdots x_ny,z\}\subset \im(\psi)$.\\\\
\noindent
\underline{$x_1\in \im(\psi):$} We have $\psi(u)=x_1$.\\

\noindent
\underline{$z\in \im(\psi)$:} Recall from the proof of Proposition \ref{G(3,3)toB} that $\psi(ust)=\Delta$. We thus have
\begin{equation*}
\begin{aligned}
\psi(t^m(ust)^{n_{(m)}-m}) &=(z^{m_{(n)}}\delta^{m-n_{(m)}})^m\Delta^{n_{(m)}-m}\\
&=(z^{m_{(n)}}\delta^{m-n_{(m)}})^m(\delta^mz^n)^{n_{(m)}-m} \, \text{since $\Delta=\delta^mz^n$}\\
&\underset{z\delta=\delta z}{=}z^{mm_{(n)}+n(n_{(m)}-m)}\delta^{m(m-n_{(m)})+m(n_{(m)}-m)}\\
&=z \, \, \text{since} \,\, mm_{(n)}+nn_{(m)}-nm=1 \, \, \text{by Lemma \ref{CoprimeLemma} and} \, \, \delta z=z\delta.
\end{aligned}
\end{equation*}

\noindent
\underline{$x_2\cdots x_ny\in \im(\psi)$:} We have seen that $z\in \im(\psi)$. Moreover, by definition we have $z^{m_{(n)}}\delta^{m-n_{(m)}}=\psi(t)$ and by Lemma \ref{Delta=Delta} we have $z^n\delta^m=\Delta=\psi(ust)$, which implies $\{\delta^{m-n_{(m)}},\delta^m\}\subset\langle z,z^{m_{(n)}}\delta^{m-n_{(m)}},z^n\delta^m\rangle\subset \im(\psi)$. In addition, since $n_{(m)}$ is coprime with $m$ we have that $m$ and $m-n_{(m)}$ are coprime, hence $\delta\in \im(\psi)$.
Since $x_1\in \im(\psi)$, we get that $x_1^{-1}\delta=x_2\cdots x_ny\in \im(\psi)$.\\
This concludes the proof.\end{proof}

%%%%%%%%%%%%%%%%%%%%%%%%%%%%%%%%%%%%%%%%%%%%%%%%%%%%%%%%%%%%%%%%%%%%%%%%%%%%%%%%%%%%%%%%%%%%%%%%%%%%%%%
%%%%%%%%%%%%%%%%%%%%%%%%%%%%%%%%%%%%%%%%%%%%%%%%%%%%%%%%%%%%%%%%%%%%%%%%%%%%%%%%%%%%%%%%%%%%%%%%%%%%%%%
%%%%%%%%%%%%%%%%%%%%%%%%%%%%%%%%%%%%%%%%%%%%%%%%%%%%%%%%%%%%%%%%%%%%%%%%%%%%%%%%%%%%%%%%%%%%%%%%%%%%%%%
%%%%%%%%%%%%%%%%%%%%%%%%%%%%%%%%%%%%%%%%%%%%%%%%%%%%%%%%%%%%%%%%%%%%%%%%%%%%%%%%%%%%%%%%%%%%%%%%%%%%%%%
%%%%%%%%%%%%%%%%%%%%%%%%%%%%%%%%%%%%%%%%%%%%%%%%%%%%%%%%%%%%%%%%%%%%%%%%%%%%%%%%%%%%%%%%%%%%%%%%%%%%%%%
%%%%%%%%%%%%%%%%%%%%%%%%%%%%%%%%%%%%%%%%%%%%%%%%%%%%%%%%%%%%%%%%%%%%%%%%%%%%%%%%%%%%%%%%%%%%%%%%%%%%%%%
%%%%%%%%%%%%%%%%%%%%%%%%%%%%%%%%%%%%%%%%%%%%%%%%%%%%%%%%%%%%%%%%%%%%%%%%

\begin{corollary}\label{braid2nm}
Let $n\in \N^*$ and $m\in\N_{\geq 2}$ be two coprime integers. Then the groups $\B_*(n,m)$ and $\B^*(m,n)$ are isomorphic, and they admit the presentation 
\begin{equation}\label{PresBraid2conj}
\left\langle\begin{array}{c|cc}
s,t,u \, & \, stu=tus=ust,\, t^m=(ust)^{m-n_{(m)}}
\end{array}\right\rangle.
\end{equation}
\end{corollary}

\begin{proof}
Combining Theorem \ref{BraidG33} and Proposition \ref{BtoG(3,3)} with Diagram \eqref{CommFINDANAMEbraids}, a presentation for $\B_*(n,m)$ is   
\begin{equation}\label{PresBraid21}
\left\langle\begin{array}{c|cc}
s,t,u \, & \, stu=tus=ust,\, t^m=(ust)^{m-n_{(m)}}
\end{array}\right\rangle.
\end{equation}
Similarly, a presentation for $\B^*(m,n)$ is 
\begin{equation}\label{PresBraid22}
\left\langle\begin{array}{c|cc}
s,t,u \, & \, stu=tus=ust,\, s^m=(ust)^{m-n_{(m)}}
\end{array}\right\rangle.
\end{equation}
Relabelling $t$ for $s$, $u$ for $t$ and $s$ for $u$ in Presentation \eqref{PresBraid22} yields the desired isomorphism. 
\end{proof}
%%%%%%%%%%%%%%%%%%%%%%%%%%%%%%%
%%%%%%%%%%%%%%%%%%%%%%%%%%%%%%%%%%%%%%%%%%%%%%%%%%%%%%%%%%%%%%%%%%%%%%%%%%%%%%%%%%%%%%%%%%%%%%%%%%%%%%%
%%%%%%%%%%%%%%%%%%%%%%%%%%%%%%%%%%%%%%%%%%%%%%%%%%%%%%%%%%%%%%%%%%%%%%%%%%%%%%%%%%%%%%%%%%%%%%%%%%%%%%%
%%%%%%%%%%%%%%%%%%%%%%%%%%%%%%%%%%%%%%%%%%%%%%%%%%%%%%%%%%%%%%%%%%%%%%%%%%%%%%%%%%%%%%%%%%%%%%%%%%%%%%%

\begin{proof}[Proof of Theorem \ref{BraidIsoType}]
By Proposition \ref{NumberConj}, we can separate the proof depending on the number of conjugacy classes of reflecting hyperplanes of $J$-reflection groups.\\
If we consider $J$-reflection groups having three conjugacy classes of reflecting hyperplanes, Theorem \ref{BraidIsoType} is the content of Corollary \ref{Braid3nm}.\\
If we consider $J$-reflection groups having two conjugacy classes of reflecting hyperplanes, Theorem \ref{Classification3} shows by the same argument than the one given in Corollary \ref{Braid3nm} that the pairs of braid groups to consider are of the form $\{\B_*(n,m),\B^*(m,n)\}$. In this case, Corollary \ref{braid2nm} shows that these groups are isomorphic.\\
Finally, if we consider $J$-reflection groups having one conjugacy class of reflecting hyperplanes, we are in the case of toric reflection groups. In this case, the result is already known to be true, as discussed in \cite[Corollary 1.3]{Gobet Toric}.
\end{proof}

\subsection{\texorpdfstring{The isomorphism type of $\B_*(n,m)$}{}}\label{Section34}
It is known (see \cite[Theorems 1 and 2]{Bannai}) that the braid groups of rank two complex reflection groups with two conjugacy classes of reflection hyperplanes are isomorphic to dihedral Artin groups of even type. Recall that for $d\in \N^*$, the dihedral Artin group of type $2d$ is the group with presentation 
\begin{equation}\label{Dihedral}
\left\langle\begin{array}{c|cc}
a,b \, & \, (ab)^d=(ba)^d
\end{array}\right\rangle.
\end{equation}
Following the notation of \cite{GarnierHoso}, we denote by $G(2,2d)$ the group defined by Presentation \eqref{Dihedral}. The goal of this subsection is to show the following theorem, generalising the existing pattern for rank two complex reflection groups: 

\begin{theorem}\label{braidG2n}
Let $n,m\in \N^*$ be two coprime integers and assume $m\geq 2$. Then the group $\B_*(n,m)$ is isomorphic to $G(2,2m)$.
\end{theorem}

\begin{corollary}\label{BraidIsGarside}
The associated braid group of any $J$-reflection group is Garside. In particular, it is torsion free.
\end{corollary}
\begin{proof}
With the nomenclature of \cite{GarnierHoso}, Definition \ref{BraidToric} and Theorems \ref{BraidG33} and \ref{braidG2n} show that the braid groups associated to $J$-reflection groups are circular groups. These groups are known to be Garside groups (see \cite[Example 5]{Origines}), which concludes the proof.
\end{proof}

In the proof of Theorem \ref{braidG2n}, we exhibit an isomorphism sending $\Delta\in \B_*(n,m)$ (as in Remark \ref{Remark/Notation}) to $(ab)^m\in G(2,2m)$. In particular, we have the following corollary:
\begin{corollary}\label{CenterBraidtoFINDANAME}
For any $J$-reflection group $H$, the center of its associated braid group $\B(H)$ is $\langle \Delta\rangle\cong\Z$, where $\Delta$ is as in Remark \ref{Remark/Notation}. Moreover, the image of $Z(\B(H))$ under the natural quotient $\B(H)\to H$ is the center of $H$.
\end{corollary}

\begin{proof}
By Proposition \ref{NumberConj}, we can separate the proof depending on the number of conjugacy classes of reflecting hyperplanes of $J$-reflection groups.\\
If we consider $J$-reflection groups having three conjugacy classes of reflecting hyperplanes, Corollary \ref{CenterBraidtoFINDANAME} is the content of Corollary \ref{CenterBraid3}.\\
We now consider $J$-reflection groups having two conjugacy classes of reflecting hyperplanes.\\
First, it is known that the center of $G(2,2d)$ is $\langle(ab)^d\rangle\cong \Z$ (see \cite[Theorem 7.2]{BrieskornSaito}), which will show that $Z(\B_*(n,m))=\langle\Delta\rangle$ using the isomorphism we exhibit for the proof of Theorem 4.17. Now, since $\Delta\in \B_*^*(n,m)$ is represented by $stu$ in $G(3,3)$, we get that $\Delta\in \B_*(n,m)$ is also represented by $stu$ in Presentation \eqref{PresBraid21}. Similarly, the element $\Delta\in \B^*(m,n)$ is represented by $stu$ in Presentation \eqref{PresBraid22}. Since the isomorphism presented in Corollary \ref{PresBraid2conj} between the groups defined by Presentations \eqref{PresBraid21} and \eqref{PresBraid22} preserves $stu$, we conclude that  $Z(\B^*(m,n))=\langle\Delta\rangle\cong \Z$. Moreover, Theorem \ref{CenterFINDANAME} shows that in both cases the image of $\Delta$ under the natural quotient $\B(H)\to H$ generates the center of $H$. \\
Finally, if we consider $J$-reflection groups having one conjugacy class of reflecting hyperplanes, we are in the case of toric reflection groups. In this case, the result follows by combining Remark \ref{Remark/Notation} and Corollary \ref{CenterSpecialCases}.
\end{proof}

\begin{remark}\label{TruncatedDihedral}
Specialising to the $J$-reflection groups of the form $W(k,2,m)$ with $m$ odd or $W_n(k,n,m)$, we obtain the family of truncated Artin groups of dihedral type (see \cite[Theorem 1.1]{Gobet Toric} and \cite[Corollary 2.31]{VCRG}). In particular, Corollary \ref{CenterBraidtoFINDANAME} shows that the center of a truncated Artin group of dihedral type is the image of the center of the corresponding dihedral Artin group under the quotient.
\end{remark}

We now introduce two morphisms 
\begin{tikzcd}
	{G(2,2m)} & {\B_*(n,m)}
	\arrow["\varphi", shift left, from=1-1, to=1-2]
	\arrow["\psi", shift left, from=1-2, to=1-1]
\end{tikzcd}
and prove that they are inverse to each other, which will prove Theorem \ref{braidG2n}. In order to construct these morphisms, from now on we consider Presentation \eqref{PresBraid2conj} for $\B_*(n,m)$. We start by proving the following lemma:

%%%%%%%%%%%%%%%%%%%%%%%%%%%%%%%%%%%%%%%%%%%%%%%%%%%%%%%%%%%%%%%%%%%%%%%%%%%%%%%%%%%%%%%%%%%%%%%%%%%%%%%
%%%%%%%%%%%%%%%%%%%%%%%%%%%%%%%%%%%%%%%%%%%%%%%%%%%%%%%%%%%%%%%%%%%%%%%%%%%%%%%%%%%%%%%%%%%%%%%%%%%%%%%
%%%%%%%%%%%%%%%%%%%%%%%%%%%%%%%%%%%%%%%%%%%%%%%%%%%%%%%%%%%%%%%%%%%%%%%%%%%%%%%%%%%%%%%%%%%%%%%%%%%%%%%
%%%%%%%%%%%%%%%%%%%%%%%%%%%%%%%%%%%%%%%%%%%%%%%%%%%%%%%%%%%%%%%%%%%%%%%%%%%%%%%%%%%%%%%%%%%%%%%%%%%%%%%
%%%%%%%%%%%%%%%%%%%%%%%%%%%%%%%%%%%%%%%%%%%%%%%%%%%%%%%%%%%%%%%%%%%%%%%%%%%%%%%%%%%%%%%%%%%%%%%%%%%%%%%
%%%%%%%%%%%%%%%%%%%%%%%%%%%%%%%%%%%%%%%%%%%%%%%%%%%%%%%%%%%%%%%%%%%%%%%%%%%%%%%%%%%%%%%%%%%%%%%%%%%%%%%
%%%%%%%%%%%%%%%%%%%%%%%%%%%%%%%%%%%%%%%%%%%%%%%%%%%%%%%%%%%%%%%%%%%%%%%%%%%%%%%%%%%%%%%%%%%%%%%%%%%%%%%
%%%%%%%%%%%%%%%%%%%%%%%%%%%%%%%%%%%%%%%%%%%%%%%%%%%%%%%%%%%%%%%%%%%%%%%%%%%%%%%%%%%%%%%%%%%%%%%%%%%%%%%
%%%%%%%%%%%%%%%%%%%%%%%%%%%%%%%%%%%%%%%%%%%%%%%%%%%%%%%%%%%%%%%%%%%%%%%%%%%%%%%%%%%%%%%%%%%%%%%%%%%%%%%
%%%%%%%%%%%%%%%%%%%%%%%%%%%%%%%%%%%%%%%%%%%%%%%%%%%%%%%%%%%%%%%%%%%%%%%%%%%%%%%%%%%%%

\begin{lemma}\label{LemmaForIso2Conj1}
Let $n,m\in \N^*$ be two coprime integers and assume $m\geq 2$.\\Let $x,y\in \Z$ be such that $xn_{(m)}+y(m-n_{(m)})=1$ (these exist since $m\wedge n_{(m)}=1$, hence $m\wedge (m-n_{(m)})=1$). In Presentation \eqref{PresBraid2conj}, the equality 
\begin{equation}\label{ustxy}
(us)^{mx}t^{my}=ust
\end{equation}
holds. In particular, the element $(us)^{mx}t^{my}$ is central.
\end{lemma}

%%%%%%%%%%%%%%%%%%%%%%%%%%%%%%%%%%%%%%%%%%%%%%%%%%%%%%%%%%%%%%%%%%%%%%%%%%%%%%%%%%%%%%%%%%%%%%%%%%%%%%%
%%%%%%%%%%%%%%%%%%%%%%%%%%%%%%%%%%%%%%%%%%%%%%%%%%%%%%%%%%%%%%%%%%%%%%%%%%%%%%%%%%%%%%%%%%%%%%%%%%%%%%%
%%%%%%%%%%%%%%%%%%%%%%%%%%%%%%%%%%%%%%%%%%%%%%%%%%%%%%%%%%%%%%%%%%%%%%%%%%%%%%%%%%%%%%%%%%%%%%%%%%%%%%%
%%%%%%%%%%%%%%%%%%%%%%%%%%%%%%%%%%%%%%%%%%%%%%%%%%%%%%%%%%%%%%%%%%%%%%%%%%%%%%%%%%%%%%%%%%%%%%%%%%%%%%%
%%%%%%%%%%%%%%%%%%%%%%%%%%%%%%%%%%%%%%%%%%%%%%%%%%%%%%%%%%%%%%%%%%%%%%%%%%%%%%%%%%%%%%%%%%%%%%%%%%%%%%%
%%%%%%%%%%%%%%%%%%%%%%%%%%%%%%%%%%%%%%%%%%%%%%%%%%%%%%%%%%%%%%%%%%%%%%%%%%%%%%%%%%%%%%%%%%%%%%%%%%%%%%%
%%%%%%%%%%%%%%%%%%%%%%%%%%%%%%%%%%%%%%%%%%%%%%%%%%%%%%%%%%%%%%%%%%%%%%%%%%%%%%%%%%%%%%%%%%%%%%%%%%%%%%%
%%%%%%%%%%%%%%%%%%%%%%%%%%%%%%%%%%%%%%%%%%%%%%%%%%%%%%%%%%%%%%%%%%%%%%%%%%%%%%%%%%%%%%%%%%%%%%%%%%%%%%%
%%%%%%%%%%%%%%%%%%%%%%%%%%%%%%%%%%%%%%%%%%%%%%%%%%%%%%%%%%%%%%%%%%%%%%%%%%%%%%%%%%%%%%%%%%%%%%%%%%%%%%%
%%%%%%%%%%%%%%%%%%%%%%%%%%%%%%%%%%%%%%%%%%%%%%%%%%%%%%%%%%%%%%%%%%%%%%%%%%%%%%%%%%%%%

\begin{proof}
The equality $t^m=(ust)^{m-n_{(m)}}$ implies that $t^m=(us)^{m-n_{(m)}}t^{m-n_{(m)}}$ since $(us)t=t(us)$, which in turn implies 
\begin{equation}\label{eqLemmaForIso2Conj1}
t^{n_{(m)}}=(us)^{m-n_{(m)}}.
\end{equation} Now, this implies that
\begin{equation}\label{eqLemmaForIso2Conj2}
\begin{aligned}
(us)^m=&(us)^{m-n_{(m)}}(us)^{n_{(m)}}\underset{\eqref{eqLemmaForIso2Conj1}}=t^{n_{(m)}}(us)^{n_{(m)}}\underset{t(us)=(us)t}=(ust)^{n_{(m)}}.
\end{aligned}
\end{equation}

We thus get that
\begin{equation*}
\begin{aligned}
(us)^{mx}t^{my}&\underset{\eqref{eqLemmaForIso2Conj2}}=(ust)^{xn_{(m)}}(ust)^{y(m-n_{(m)})}
=(ust)^{xn_{(m)}+y(m-n_{(m)})}\\
&=ust \, \text{since $xn_{(m)}+y(m-n_{(m)})=1$.}
\end{aligned}
\end{equation*}
\end{proof}

%%%%%%%%%%%%%%%%%%%%%%%%%%%%%%%%%%%%%%%%%%%%%%%%%%%%%%%%%%%%%%%%%%%%%%%%%%%%%%%%%%%%%%%%%%%%%%%%%%%%%%%
%%%%%%%%%%%%%%%%%%%%%%%%%%%%%%%%%%%%%%%%%%%%%%%%%%%%%%%%%%%%%%%%%%%%%%%%%%%%%%%%%%%%%%%%%%%%%%%%%%%%%%%
%%%%%%%%%%%%%%%%%%%%%%%%%%%%%%%%%%%%%%%%%%%%%%%%%%%%%%%%%%%%%%%%%%%%%%%%%%%%%%%%%%%%%%%%%%%%%%%%%%%%%%%
%%%%%%%%%%%%%%%%%%%%%%%%%%%%%%%%%%%%%%%%%%%%%%%%%%%%%%%%%%%%%%%%%%%%%%%%%%%%%%%%%%%%%%%%%%%%%%%%%%%%%%%
%%%%%%%%%%%%%%%%%%%%%%%%%%%%%%%%%%%%%%%%%%%%%%%%%%%%%%%%%%%%%%%%%%%%%%%%%%%%%%%%%%%%%%%%%%%%%%%%%%%%%%%
%%%%%%%%%%%%%%%%%%%%%%%%%%%%%%%%%%%%%%%%%%%%%%%%%%%%%%%%%%%%%%%%%%%%%%%%%%%%%%%%%%%%%%%%%%%%%%%%%%%%%%%
%%%%%%%%%%%%%%%%%%%%%%%%%%%%%%%%%%%%%%%%%%%%%%%%%%%%%%%%%%%%%%%%%%%%%%%%%%%%%%%%%%%%%%%%%%%%%%%%%%%%%%%
%%%%%%%%%%%%%%%%%%%%%%%%%%%%%%%%%%%%%%%%%%%%%%%%%%%%%%%%%%%%%%%%%%%%%%%%%%%%%%%%%%%%%%%%%%%%%%%%%%%%%%%
%%%%%%%%%%%%%%%%%%%%%%%%%%%%%%%%%%%%%%%%%%%%%%%%%%%%%%%%%%%%%%%%%%%%%%%%%%%%%%%%%%%%%%%%%%%%%%%%%%%%%%%
%%%%%%%%%%%%%%%%%%%%%%%%%%%%%%%%%%%%%%%%%%%%%%%%%%%%%%%%%%%%%%%%%%%%%%%%%%%%%%%%%%%%%

\begin{lemma}\label{Iso2Conj1}
Let $n,m\in \N^*$ be two coprime integers and assume $m\geq 2$.\\
Let $x,y\in \Z$ be such that $xn_{(m)}+y(m-n_{(m)})=1$. The map $\{a,b\}\xto\varphi \B_*(n,m)$ sending $a$ to $(us)^xt^ys^{-1}$ and $b$ to $s$ extends to a morphism $G(2,2m)\to \B_*(n,m)$.
\end{lemma}

\begin{proof}
It is enough to show that the images of $a$ and $b$ by $\varphi$ satisfy $(\varphi(a)\varphi(b))^m=(\varphi(b)\varphi(a))^m$. We have
\begin{equation*}
\begin{aligned}
(\varphi(a)\varphi(b))^m=&((us)^xt^ys^{-1}s)^m\underset{(us)t=t(us)}=(us)^{mx}t^{my}=ust \, \, \text{by Lemma \ref{LemmaForIso2Conj1},}
\end{aligned}
\end{equation*}
and
\begin{equation*}
\begin{aligned}
(\varphi(b)\varphi(a))^m=&(s(us)^xt^ys^{-1})^m\underset{(us)t=t(us)}=s(us)^{mx}t^{my}s^{-1}\\&=ust \,\,\text{by using both assertions of Lemma \ref{LemmaForIso2Conj1}.}
\end{aligned}
\end{equation*} 
This concludes the proof.
\end{proof}

\begin{lemma}\label{Iso2Conj2}
Let $n,m\in \N^*$ be two coprime integers and assume $m\geq 2$. The map $\{s,t,u\}\xto\psi G(2,2m)$ sending $s$ to $b$, $t$ to $(ab)^{m-n_{(m)}}$ and $u$ to $(ab)^{n_{(m)}}b^{-1}$ extends to a morphism $B_*(n,m)\to G(2,2m)$.
\end{lemma}

\begin{proof}
It is enough to show that the images of $s,t$ and $u$ satisfy the relations of Presentation \eqref{Dihedral}. \\
\fbox{The relation $stu=tus=ust$:}\\
We have
\begin{equation*}
\begin{aligned}
\psi(s)\psi(t)\psi(u)&=b(ab)^{m-n_{(m)}}(ab)^{n_{(m)}}b^{-1} =(ab)^{m} \, \, \text{since} \, \, (ab)^m \, \, \text{is central},\\
\psi(t)\psi(u)\psi(s)&=(ab)^{m-n_{(m)}}(ab)^{n_{(m)}}b^{-1}b=(ab)^{m},\\
\psi(u)\psi(s)\psi(t)&=(ab)^{n_{(m)}}b^{-1}b(ab)^{m-n_{(m)}}=(ab)^{m}.
\end{aligned}
\end{equation*}
\fbox{The relation $t^m=(stu)^{m-n_{(m)}}$:}\\
We have
\begin{equation*}
\begin{aligned}
\psi(t)^m=(ab)^{m(m-n_{(m)})}=(\psi(s)\psi(t)\psi(u))^{m-n_{(m)}}.
\end{aligned}
\end{equation*}
This concludes the proof.
\end{proof}

We are now ready to prove Theorem \ref{braidG2n}:
\begin{proof}[Proof of Theorem \ref{braidG2n}]

Using Lemmas \ref{Iso2Conj1} and \ref{Iso2Conj2}, it is enough to show that $\varphi\circ\psi=Id_{\B_*(n,m)}$ and $\psi\circ\varphi=Id_{G(2,2m)}$.\\\\
\fbox{$\varphi\circ\psi=Id_{\B_*(n,m)}$:}\\\\
\underline{$\varphi\circ\psi(s)=s$:} We have $\varphi\circ\psi(s)=\varphi(b)=s.$\\\\
\noindent
\underline{$\varphi\circ\psi(t)=t$:} We have
\begin{equation*}
\begin{aligned}
\varphi\circ\psi(t)&=\varphi((ab)^{m-n_{(m)}})=((us)^xt^y)^{m-n_{(m)}}
\underset{(us)t=t(us)}=(us)^{x(m-n_{(m)})}t^{y(m-n_{(m)})}\\&\underset{(us)^{m-n_{(m)}}=t^{n_{(m)}}}=t^{xn_{(m)}}t^{y(m-n_{(m)})}=t \, \, \text{since} \, \, xn_{(m)}+y(m-n_{(m)})=1.
\end{aligned}
\end{equation*}

\noindent
\underline{$\varphi\circ\psi(u)=u$:} We have
\begin{equation*}
\begin{aligned}
\varphi\circ\psi(u)&=\varphi((ab)^{n_{(m)}}b^{-1})
=((us)^xt^y)^{n_{(m)}}s^{-1}
\underset{(us)t=t(us)}=(us)^{xn_{(m)}}t^{yn_{(m)}}s^{-1}\\
&=(us)^{xn_{(m)}}(us)^{y(m-n_{(m)})}s^{-1}\, \text{since $(us)^{m-n_{(m)}}=t^{n_{(m)}}$}\\
&=u\, \, \text{since} \, \, xn_{(m)}+y(m-n_{(m)})=1.
\end{aligned}
\end{equation*}

\noindent
\fbox{$\psi\circ\varphi=Id_{G(2,2m)}$:}\\\\
\underline{$\psi\circ\varphi(a)=a$:} We have
\begin{equation*}
\begin{aligned}
\psi\circ\varphi(a)&=\psi((us)^xt^ys^{-1})=(ab)^{xn_{(m)}}(ab)^{y(m-n_{(m)})}b^{-1}
\\&=a \, \, \text{since} \, \, xn_{(m)}+y(m-n_{(m)})=1.
\end{aligned}
\end{equation*}

\noindent
\underline{$\psi\circ\varphi(b)=b$:} We have $\psi\circ\varphi(b)=\psi(s)=b$.\\
This concludes the proof.
\end{proof}

\subsection{Reflection isomorphisms of braid groups}

The goal of this subsection is to show that the reflection isomorphism given in Proposition \ref{prop1} lifts to braid groups of $J$-reflection groups.

\begin{definition}
Let $\theta$ be the automorphism of $G(3,3)$ sending $s$ to $t^{-1}$, $t$ to $s^{-1}$ and $u$ to $u^{-1}$.
\end{definition}

\begin{remark}\label{rem1}
Recall from Presentation \eqref{defG33} that in $G(3,3)$ we have $t(us)=(us)t$. Thus, for all $N\in\N$ we have $t^N(us)^N=(ust)^N\in Z(G(3,3))$. In particular, for all $N\in\N$ we have $t^{-N}u^{-1}t^N=(us)^Nu^{-1}(us)^{-N}$.
\end{remark}

\begin{notation}
For all $n,m\in\N^*$ with $n\wedge m=1$, write $\{a_1,\dots,a_m,p,q\}$ to be the elements of $\B_*^*(m,n)$ corresponding to $\{x_1,\dots,x_n,y,z\}$ in Presentation \eqref{classicalBraidPres1}. Moreover, we denote by $\varphi_{n,m}$ and $\varphi_{m,n}$ the isomorphisms $\B_*^*(n,m)\to G(3,3)$ and $\B_*^*(m,n)\to G(3,3)$ given in Proposition \ref{BtoG(3,3)}.
\end{notation}

\begin{lemma}\label{lem2}
Let $n,m\in\N^*$ with $n\wedge m=1$.
For all $i\in [n]$, writing $i-1=g_im+h_i$ the euclidean division of $i-1$ by $m$ we have 
\begin{equation}
\varphi_{m,n}(((a_1\cdots a_mp)^{g_i}a_1\cdots a_{h_i})a_{h_i+1}^{-1}((a_1\cdots a_mp)^{g_i}a_1\cdots a_{h_i})^{-1})=(us)^{i-1}u^{-1}(us)^{1-i}.
\end{equation}
\end{lemma}

\begin{proof}
We have 
\begin{equation*}
\begin{aligned}
    \varphi_{m,n}(a_1\cdots a_mp)=u(s u s^{-1})\cdots (s^{m-1}u s^{1-m})s^m(ust)^{n_{(m)}-m}=(us)^m(ust)^{n_{(m)}-m}.
\end{aligned}
\end{equation*}
Moreover, for all $i\in[n]$ we have 
\begin{equation*}
    \begin{aligned}
        \varphi_{m,n}(a_1\cdots a_{h_i})&=u(sus^{-1})\cdots(s^{h_i-1}us^{1-h_{i}})\\
        &=(us)^{h_i}s^{-h_i}
    \end{aligned}
\end{equation*}
\noindent
This shows that
\begin{equation*}
\begin{aligned}
&\varphi_{m,n}(((a_1\cdots a_mp)^{g_i}a_1\cdots a_{h_i})a_{h_i+1}^{-1}((a_1\cdots a_mp)^{g_i}a_1\cdots a_{h_i})^{-1})\\
&=((us)^m(ust)^{n_{(m)}-m})^{g_i}(us)^{h_i}s^{-h_i}s^{h_i}u^{-1}s^{-h_i}(((us)^m(ust)^{n_{(m)}-m})^{g_i}(us)^{h_i}s^{-h_i})^{-1}\\
&=((us)^m)^{g_i}(us)^{h_i}u^{-1}s^{-h_i}(((us)^m)^{g_i}(us)^{h_i}s^{-h_i})^{-1} \, \text{since $ust$ is central}\\
&=(us)^{i-1}u^{-1}(us)^{1-i}.
\end{aligned}
\end{equation*}
This concludes the proof.
\end{proof}

\begin{proposition}
Denote by $\pi_1$ the natural  quotient $\B_*^*(n,m)\to W_b^c(k,bn,cm)$ and by $\pi_2$ the natural quotients $\B_*^*(m,n)\to W_c^b(k,cm,bn)$. Then the following hold:\\
(i) For all $w\in \{x_1,\dots,x_n,y,z\}$, the element $ \varphi_{m,n}^{-1}\circ\theta\circ\varphi_{n,m}(w)$ is conjugate in $\B_*^*(m,n)$ to a power of an element of $\{a_1,\dots,a_m,p,q\}$.\\
(ii) The following square is commutative:
\begin{equation}\begin{tikzcd}
	{\B_*^*(n,m)} & {\B_*^*(m,n)} \\
	{W_b^c(k,bn,cm)} & {W_c^b(k,cm,bn)}
	\arrow["{\varphi_{m,n}^{-1}\circ\theta\circ\varphi_{n,m}}", from=1-1, to=1-2]
	\arrow["{\pi_1}"', from=1-1, to=2-1]
	\arrow["{f}"', from=2-1, to=2-2]
	\arrow["{\pi_2}", from=1-2, to=2-2]
\end{tikzcd}\end{equation}
Under both horizontal arrows, the image of the conjugates of powers of generators are precisely the conjugates of powers of generators.
\end{proposition}

\begin{proof}
\fbox{(i)} Let $i\in [n]$ and write $i-1=g_im+h_i$ the euclidean division of $i-1$ by $m$. Then we have
\begin{equation}\label{eq11}
    \begin{aligned}
        &\varphi_{m,n}^{-1}\circ\theta\circ\varphi_{n,m}(x_i)=\varphi_{m,n}^{-1}\circ\theta(s^{i-1}us^{1-i})
        =\varphi_{m,n}^{-1}(t^{1-i}u^{-1}t^{i-1})\\&=\varphi_{m,n}^{-1}((us)^{i-1}u^{-1}(us)^{1-i}) \, \text{by Remark \ref{rem1}}\\
        &=((a_1\cdots a_mp)^{g_i}a_1\cdots a_{h_i})a_{h_i+1}^{-1}((a_1\cdots a_mp)^{g_i}a_1\cdots a_{h_i})^{-1} \, \text{by Lemma \ref{lem2}.}
    \end{aligned}
\end{equation}
Moreover, we have
\begin{equation}\label{eq12}
    \varphi_{m,n}^{-1}\circ\theta\circ\varphi_{n,m}(y)=\varphi_{m,n}^{-1}\circ\theta(s^n(ust)^{m_{(n)}-n})=\varphi_{m,n}^{-1}(t^{-n}(ust)^{n-m_{(n)}})=q^{-1}.
\end{equation}
Similarly, we have \begin{equation}\label{eq13}\varphi_{m,n}^{-1}\circ\theta\circ\varphi_{n,m}(z)=p^{-1}.
\end{equation}
\fbox{(ii)} It follows from combining Proposition \ref{prop1} with Equations \eqref{eq11}-\eqref{eq13}.
\end{proof}

\section{\texorpdfstring{Classical and dual braid monoids associated to $J$-reflection groups}{}}

In this section, we define the classical braid monoid associated to $J$-reflection groups and give another presentation for their braid groups (see Proposition \ref{DualPresentationsBraidGroups}). This presentation coincides with the dual presentation given in \cite{DualBessis} when the group is a real reflection group (in our case, a dihedral group), which in turn corresponds to the presentations given in \cite{BessisCorran}, \cite{CorranPicantin} and \cite{NeaimeInterval} for groups of the form $G(e,e,r)$. Moreover, whenever the group is a toric reflection group, we retrieve the dual presentation defined in \cite{DualBessis} and \cite{Gobet Toric}.

\subsection{The classical braid monoid}
Given $n,m\in\N^*$ with $n\wedge m=1$ and $n\leq m$, write $m=qn+r$ to be the euclidean division of $m$ by $n$. We then have $q-1\geq 0$ so that Presentations \eqref{classicalBraidPres1}-\eqref{BraidToric} are positive. It allows us to give the following definition:

\begin{definition}\label{ClassicalMonoid}
Let $n,m\in\N^*$ with $n\wedge m=1$ and $n\leq m$. \\
Denote by $\M_*^*(n,m)$ the monoid with Presentation \eqref{classicalBraidPres1}.\\
If $m\geq2$, denote by $\M_*(n,m)$ the monoid with Presentation \eqref{BraidPresy1}.\\
If $n\geq 2$, denote by $\M^*(n,m)$ the monoid with Presentation \eqref{BraidPresz1}.\\
If $n,m\geq 2$, denote by $\M(n,m)$ the monoid with Presentation \eqref{BraidToric}.\\
We call these monoids \textbf{classical} and say that they are \textbf{associated} to the braid groups with the respective presentations.
\end{definition}

\begin{example}
For $n=1,m=2$, setting $x=x_1$, Presentation \eqref{classicalBraidPres1} reads 
  \begin{equation*}\left\langle\begin{array}{l|cl}
   					x,y,z\,	 & \, zxy=xyz   \\
    &\, xyzxy=yzxyx
				
                          \end{array}\right\rangle,\end{equation*} 
thus we retrieve the Garside monoid associated to the BMR presentation of $G_{15}$ as well as that of $G(4c,4,2)$ (see \cite[Theorem 2.27 and Tables 1-3]{BMR} for the presentation and \cite[Example 3.25]{GrosBouquinBleu} for the fact that the monoid is Garside).
\end{example}

\begin{example}
For $n=2,m=3$, Presentation \eqref{classicalBraidPres1} reads 
  \begin{equation*}\left\langle\begin{array}{l|cl}
   					x_1,x_2\,	 & \, zx_1x_2y=x_1x_2yz   \\
 \,\, y,z  &\, x_1x_2yzx_1=x_2yzx_1x_2,\, x_2yzx_1x_2y=yzx_1x_2yx_1
				
                          \end{array}\right\rangle.\end{equation*}

\end{example}

\begin{remark}
Whenever $W_b^c(k,bn,cm)$ is finite, the presentation of the classical braid monoid associated to $\B(W_b^c(k,bn,cm))$ coincides with its BMR presentation.
\end{remark}

\begin{remark}\label{PositiveDelta}
In the case $n\leq m$, Presentation \eqref{classicalBraidPres1} is positive so that the proofs of Lemmas \ref{xvw} and \ref{xvW} as well as Remark \ref{remarkMod} only use positive words, hence they are still valid for $\M_*^*(n,m)$.
Finally, the word $w_0$ as well as the words in the set $\{x_{j+1}\cdots x_nyz\delta^q\delta^{-1}x_1\cdots x_{j+r}\}_{j\in[n-r-1]}$ can be seen as representing elements of the monoid $\M_*^*(n,m)$. Following the same proof steps as in Lemma \ref{Delta=Delta}, we thus get that in the case $n\leq m$, the element $\Delta:=w^{n-r}W^r\in \M_*^*(n,m)$ is equal to $\delta^mz^n$, and we conclude that it is central in $\M_*^*(n,m)$ with the same argument as the one in Proposition \ref{DeltaCentral}.
\end{remark}

\subsection{\texorpdfstring{Dual presentation for braid groups associated to $J$-reflection groups}{}}
The goal of this subsection is to give new presentations for braid groups associated to $J$-reflection groups. Those presentations will be called dual, as they correspond to the dual presentations in the sense of \cite{DualBessis} and \cite{Gobet Toric}. The main result of this subsection takes the following form:
    \begin{proposition}\label{DualPresentationsBraidGroups}
Let $n,m\in \N^*$ with $n\leq m$ and assume $n\wedge m=1$. Then $\B_*^*(n,m)\cong \B_*^*(m,n)$ admits the following presentation:
\begin{subequations}\label{dual3}
\begin{align}
    &(1) \,\, \mathrm{Generators}\!:\,  \{x_1,\dots,x_m,y,z_1,\dots,z_m\};\notag\\
&(2) \,\, \mathrm{Relations}\!: \,\notag\\
  &x_iz_{i+1}=z_ix_i , \, \forall 1\leq i\leq m-1, \, \label{dual3:1}\\ &x_{m}yz_1=z_{m}x_{m}y,\label{dual3:2}\\
  &z_{i+1}x_{i+1}\cdots x_{i+n}=z_ix_i\cdots x_{i+n-1} , \, \forall 1\leq i\leq m-n,\label{dual3:3}\\
       & z_{i+1}x_{i+1}\cdots x_{m}yx_1\cdots x_{i+n-m}=z_ix_i\cdots x_{m}yx_1\cdots x_{i+n-m-1} ,(*),\label{dual3:4}\\
      & yz_1x_1\cdots x_{n}=z_{m}x_{m}yx_1\cdots x_{n-1}.\label{dual3:5}\\
      &(*) \forall m-n+1\leq i\leq m-1 \notag
\end{align}
\end{subequations}
If $m\geq 2$, the group $\B_*(n,m)\cong \B^*(m,n)$ admits the following presentation:
\begin{equation}\label{dual3z}
\begin{aligned}
    &(1) \,\, \mathrm{Generators}\!:\,  \{x_1,\dots,x_m,z_1,\dots,z_m\};\\
&(2) \,\, \mathrm{Relations}\!: \,\\
  &x_iz_{i+1}=z_ix_i , \, \forall 1\leq i\leq m-1, \,\,\\
  &x_mz_1=z_mx_m,\\
  &z_{i+1}x_{i+1}\cdots x_{i+n}=z_ix_i\cdots x_{i+n-1} , \, \forall 1\leq i\leq m-1,\,\,\,\,\,\,\,\,\,\,\,\,\,\,\,\,\,\,\,\,\,\,\,\,\,\,\,\,\,\,\,\,\,\,\,\,\,\,\,\,\,\,\,\,\,\,\,\,\,\,\,\,\,\,\,\,\,\,\,\,\,\,\,\,\,\,\,\,\,
\end{aligned}
\end{equation}
where indices are taken modulo $m$.\\\\
If $n\geq 2$, the group $\B^*(n,m)\cong \B_*(m,n)$ admits the following presentation: 
\begin{equation}\label{dual3y}
\begin{aligned}
    &(1) \,\, \mathrm{Generators}\!:\,  \{x_1,\dots,x_m,y\};\\
&(2) \,\, \mathrm{Relations}\!: \,\\
  &x_{i+1}\cdots x_{i+n}=x_i\cdots x_{i+n-1} , \, \forall 1\leq i\leq m-n,\\
       & x_{i+1}\cdots x_{m}yx_1\cdots x_{i+n-m}=x_i\cdots x_{m}yx_1\cdots x_{i+n-m-1},  \,   \forall m-n+1\leq i\leq m-1,\,\,\,\,\,\,\,\,\,\,\,\,\\
       & yx_1\cdots x_n=x_myx_1\cdots x_{n-1}.
\end{aligned}
\end{equation}

\end{proposition}

Before giving the proof of Proposition \ref{DualPresentationsBraidGroups}, we observe the following: 

\begin{remark}\label{dualexists}
Presentation \eqref{dual3z} gives the following presentation of $\B_*(1,m)$:
\begin{equation}\label{dual3zGood}\left\langle
       \begin{array}{l|cl}
           x_1,\dots,x_{m}     \,        & x_iz_{i+1}=z_ix_i \, \, \forall i=1,\dots,m-1,\,\, x_mz_1=z_mx_m, \\
    z_1,\dots,z_{m}     \,  &\,     z_{i+1}x_{i+1}=z_ix_i \, \, \forall i=1,\dots,m-1 \\
     
                          \end{array}
     \right\rangle.\end{equation}
Relabeling $z_i$ by $a_{2i-1}$ and $x_i$ by $a_{2i}$, one retrieves the dual presentation of $\B_*(1,m)=\B(I_2(2m))$ given in \cite{DualBessis}, \cite{BessisCorran} and \cite{CorranPicantin}. 
\end{remark}

\begin{remark}
In presentations \eqref{dual3z} and \eqref{dual3y},
using either of the quotients $$\B_*(m,n)\xrightarrow{y=1}\B(m,n)\xleftarrow{z_1=1}\B^*(m,n)$$ one retrieves the dual presentation of $\B(n,m)$ that was defined in \cite{Gobet Toric}.
\end{remark}

\begin{proof}[Proof of Proposition \ref{DualPresentationsBraidGroups}]
For $n=m=1$, Presentations \eqref{classicalBraidPres1} and \eqref{dual3} both are the same as Presentation \eqref{defG33}.
Assume now that $1\leq n<m$. In this case the Euclidean division of $n$ by $m$ is $n=0*m+n$ so that Presentation \eqref{classicalBraidPres1} for $\B^*_*(m,n)$ is as follows:
\begin{subequations}\label{PresClassicalMN}
\begin{align}
&(1) \,\, \mathrm{Generators}\!:\,  \{x_1,\dots,x_m,y,z\};\notag\\
&(2) \,\, \mathrm{Relations}\!: \,\notag\\
  &  x_1\cdots x_myz=zx_1\cdots x_my,\label{PresClassicalMN:1}\\
       & x_{i}\cdots x_myz\delta^{-1}x_1\cdots x_{i+n-1}=x_1\cdots x_myz\delta^{-1}x_1\cdots x_{n}, \, \forall\, 2\leq i \leq m-n+1,\label{PresClassicalMN:2}\\
      &  x_{i+1}\cdots x_myzx_1\cdots x_{i+n-m}=x_i\cdots x_myzx_1\cdots x_{i+n-m-1},\, \forall m-n+1\leq i \leq m-1,\label{PresClassicalMN:3}\\
      & yzx_1\cdots x_n=x_myzx_1\cdots x_{n-1}.\label{PresClassicalMN:4}
\end{align}
\end{subequations}
where $\delta$ denotes $x_1\cdots x_my$.\\
For all $1\leq i \leq m-n+1$, we have 
\begin{equation}\label{transfo1dual}
\begin{aligned}
    x_i\cdots x_myz\delta^{-1}x_1\cdots x_{i+n-1}&=(x_1\cdots x_{i-1})^{-1}\delta z\delta^{-1}x_1\cdots x_{i+n-1}\\
    &=(x_1\cdots x_{i-1})^{-1}zx_1\cdots x_{i+n-1}\,\, \text{since $\delta z=z\delta$.}
\end{aligned}
\end{equation}
Equation \eqref{transfo1dual} implies that the set of relations \eqref{PresClassicalMN:1}+\eqref{PresClassicalMN:2} is equivalent to the set of relations 
\begin{equation}
    \label{transfo2dual}
    \begin{aligned}
        (zx_1\cdots x_n=((x_1\cdots x_i)^{-1}zx_1\cdots x_i)x_{i+1}\cdots x_{i+n-1} \,\,\forall \, 1\leq i\leq m-n)+\eqref{PresClassicalMN:1}.
    \end{aligned}
\end{equation}
Writing $z_i$ to be $((x_1\cdots x_{i-1})^{-1}zx_1\cdots x_{i-1})$ for all $i\in[m-n+1]$, the set of relations \eqref{transfo2dual} then reads
\begin{equation}
    \label{transfo3dual}
    \begin{aligned}
        (z_1x_1\cdots x_n=z_{i+1}x_{i+1}\cdots x_{i+n-1} \forall \, 1\leq i\leq m-n)+\eqref{PresClassicalMN:1}.
    \end{aligned}
\end{equation}
Similarly, for all $m-n+1\leq i\leq m$ we have 
\begin{equation}\label{transfo4dual}
\begin{aligned}
    x_i\cdots x_myzx_1\cdots x_{i+n-m-1}&=(x_1\cdots x_{i-1})^{-1}\delta zx_1\cdots x_{i+n-m-1}\\
    &=(x_1\cdots x_{i-1})^{-1}z\delta x_1\cdots x_{i+n-m-1}\,\, \text{since $\delta z=z\delta$.}
\end{aligned}
\end{equation}
Writing $z_i$ to be $((x_1\cdots x_{i-1})^{-1}zx_1\cdots x_{i-1})$ for all $i\in\llbracket m-n+2,m\rrbracket$, Equation \eqref{transfo4dual} implies that the set of relations \eqref{PresClassicalMN:1}+\eqref{PresClassicalMN:3} is equivalent to the set of relations 
\begin{equation}
    \label{transfo5dual}
    \begin{aligned}
        (z_{m-n+1}x_{m-n+1}\cdots x_my=z_ix_{i}\cdots x_myx_1\cdots x_{i+n-m-1} \,\,\forall \, m-n+2\leq i\leq &m)\\
        &+\eqref{PresClassicalMN:1},
    \end{aligned}
\end{equation}
In addition, the relation $x_1\cdots x_myz=zx_1\cdots x_my$ is equivalent to 
\begin{equation}\label{transfo6dual}
\begin{aligned}
    x_myz&=((x_1\cdots x_{m-1})^{-1}zx_1\cdots x_{m-1})x_my\\
    &=z_mx_my.
\end{aligned}    
\end{equation}
Finally, the set of relations 
\begin{equation}\label{transfo7dual}
    z_i=((x_1\cdots x_{i-1})^{-1}zx_1\cdots x_{i-1}),\, \forall i=1,\dots,m
\end{equation}
is equivalent to the set of relations
\begin{equation}\label{transfo8dual}
    x_iz_{i+1}=z_ix_i \, \forall 1\leq i \leq m-1.
\end{equation}
Thus, a presentation of $\B_*^*(m,n)$ is given by the generators $\{x_1,\dots,x_m,y,z_1,\dots,z_m\}$ and the set of relations \eqref{transfo3dual}+\eqref{transfo5dual}+\eqref{transfo6dual}+\eqref{transfo8dual}+\eqref{PresClassicalMN:4}, which is precisely Presentation \eqref{dual3}. Since $\B_*^*(m,n)\cong \B_*^*(n,m)$ by Lemma \ref{Braid3nm}, this shows that Presentation \eqref{dual3} is a presentation of $\B_*^*(n,m)$.\\ 
Now, the quotients $\B^*(m,n)\xleftarrow{y=1} \B_*^*(m,n)\xrightarrow{z_1=1}\B_*(m,n)$ provide Presentations \eqref{dual3z} and \eqref{dual3y} for $\B^*(m,n)$ and $\B_*(m,n)$ respectively. Finally, by Lemma \ref{braid2nm} we have $\B_*(m,n)\cong\B^*(n,m)$ and $\B^*(m,n)\cong\B_*(n,m)$, which concludes the proof. 
\end{proof}

%%%%%%%%%%%%%%%%%%%%%%%%%%%%%%%%%%%%%%%%%%%%%%%%%%%%%%%%%%%%%%
%%%%%%%%%%%%%%%%%%%%%%%%%%%%%%%%%%%%%%%%%%%%%%%%%%%%%%%%%%%%%%%%%%%%%%%%%%%%%%%%%%%%%%%%%%%%%%%%%%%%%%%
%%%%%%%%%%%%%%%%%%%%%%%%%%%%%%%%%%%%%%%%%%%%%%%%%%%%%%%%%%%%%%%%%%%%%%%%%%%%%%%%%%%%%%%%%%%%%%%%%%%%%%%
%%%%%%%%%%%%%%%%%%%%%%%%%%%%%%%%%%%%%%%%%%%%%%%%%%%%%%%%%%%%%%%%%%%%%%%%%%%%%%%%%%%%%%%%%%%%%%%%%%%%%%%
\subsection{The dual braid monoid}

\begin{definition}\label{DefDualMonoid}
Let $n,m\in \N^*$ with $n\wedge m=1$ and $n\leq m$.\\
Denote by $\D_*^*(n,m)$ the monoid with Presentation \eqref{dual3}.\\
If $m\geq 2$, denote by $\D_*(n,m)$ the monoid with Presentation \eqref{dual3z}.\\
If $n\geq 2$, denote by $\D^*(n,m)$ the monoid with Presentation \eqref{dual3y}.\\
If $n,m\geq 2$, denote by $\D(n,m)$ the monoid with the same presentation as $G(m,n)$.\\
We call these monoids \textbf{dual} respectively to $\M_*^*(n,m)$, $\M_*(n,m)$, $\M^*(n,m)$ and $\M(n,m)$. We also say that they are \textbf{associated} to the same braid group as their classical counterparts.
\end{definition}
\noindent
Among rank two complex reflection groups, many of the dual monoids defined above are already known (see \cite{DualBessis}, \cite{BessisKP} and \cite{Gobet Toric}). The rank two complex reflection groups for which these monoids seem to be new are those appearing in the following table:
\begin{table}[hbt]
\centering
\begin{tabular}{|c|c|}
\hline
Group & Dual presentation \\
\hline
$G(cd,d,2)$& $\mathcal D^*(2,d)=\left\langle
       \begin{array}{l|cl}
           x_1,\dots,x_{d}     \,        & x_ix_{i+1}=x_jx_{j+1}\, \, \forall 1\leq i<j\leq d-1\\
   \,\,\,\,\,\,\,\,\,\,\,\,   y     \,   &\,    x_{d-1}x_dy=x_dyx_1=yx_1x_2

                          \end{array}
     \right\rangle$
\\
\hline
$G(2cd,2d,2)$&$\mathcal D_*^*(1,d)=\left\langle
       \begin{array}{l|cl}
           x_1,\dots,x_{d},y     \,        & x_iz_{i+1}=z_ix_i=z_{i+1}x_{i+1} \, \, \forall i\in [d-1]\\
      z_1,\dots,z_{d}     \,   &\,  x_dyz_1=yz_1x_1=z_dx_dy

                          \end{array}
     \right\rangle$
\\
\hline
$G_{13}$&$\mathcal D_*(2,3)=\left\langle
       \begin{array}{l|cl}
           x_1,x_2,x_3     \,        & x_1z_2=z_1x_1, \, x_2z_3=z_2x_2,\, x_3z_1=z_3x_3 \\
      z_1,z_2,z_3   \,   &\,   z_1x_1x_2=z_2x_2x_3=z_3x_3x_1

                          \end{array}
     \right\rangle$
\\
\hline
$G_{15}$& $\mathcal D_*^*(1,2)=\left\langle
       \begin{array}{l|cl}
           x_1,x_{2},y     \,        & z_1x_1=z_2x_2=x_1z_2\\
      z_1,z_{2}     \,   &\,    x_2yz_1=yz_1x_1=z_2x_2y

                          \end{array}
     \right\rangle$
\\
\hline

\end{tabular}
\caption{New dual monoids for some rank two complex reflection groups}
\end{table}

\section{\texorpdfstring{Some results about homogeneous Garside monoids}{}}\label{SecGarside}
In this section, we recall the definition of a Garside monoid, some known results about them and construct a framework that will simplify the proofs that the monoids we manipulate are Garside.

\subsection{Summary on Garside monoids}
\noindent
This subsection closely follows \cite[Section 2]{Gobet Monoid 1}. For a more complete exposition of the topic, the author refers to \cite{GrosBouquinBleu}, where all results mentioned below can be found.

\begin{definition}
In a monoid $M$, if $u,v,w\in M$ are such that $uv=w$, we say that $w$ is a \textbf{right-multiple} of $u$ and a \textbf{left-multiple} of $v$, whereas $u$ is a \textbf{left-divisor} of $w$ and $v$ is a \textbf{right-divisor} of $w$.
\end{definition}

\begin{definition}\label{Cancellative}
A monoid $M$ is \textbf{left-} (respectively \textbf{right-}) \textbf{cancellative} if for all $a,b,c\in M$, the equality $ab=ac$ (respectively $ba=ca$) implies $b=c$. A monoid is \textbf{cancellative} if it is both left- and right-cancellative. 
\end{definition}

\begin{definition}\label{OreDef}
A monoid $M$ is an \textbf{Ore monoid} if it is cancellative and admits common left-multiples (that is, for all $a,b\in M$, there exists $a',b'\in M$ such that $a'a=b'b$).
\end{definition}

\begin{theorem}\label{Ore}
An Ore monoid $M$ admits a group of fraction $G(M)$ in which it embeds. Moreover, given any monoid presentation for $M$, the corresponding group presentation is a presentation for $G(M)$.
\end{theorem}

\begin{definition}
The divisibility of a monoid $M$ is said to be \textbf{Noetherian} if there exists a function $\lambda:M\to \N$ such that \\
(i) $\forall a,b\in M$, we have $\lambda(ab)\geq \lambda(a)+\lambda(b)$,\\
(ii) 
$\lambda^{-1}(0)=\{1_M\}$.
\end{definition}
\noindent
In order to define Garside monoids, we wish to talk about divisibility as a partial order. Divisibility relation does not always define a poset on a monoid, but it does in the following setting:

\begin{lemma}\label{Poset}
Let $M$ be a left- (respectively right-) cancellative monoid such that $M^\times=\{1_M\}$. Writing $a\leq_L b$ (respectively $a\leq_Rb$ ) if $b$ is a right- (respectively left-) multiple of $a$, the relation $\leq_L$ (respectively $\leq_R$) defines a partial order on $M$.
\end{lemma}

\begin{remark}\label{NoetPoset}
If $M$ is a cancellative monoid with Noetherian divisibility, the pairs $(M,\leq_L)$ and $(M,\leq_R)$ are posets.
\end{remark}

\begin{notation}\label{divisors}
Let $M$ be a monoid and $m\in M$. We denote by  $\text{Div}_L(m)$ the set of left-divisors of $m$. Similarly, we denote by $\text{Div}_R(m)$ the set of right-divisors of $m$. Whenever these two sets coincide, we simply denote it by $\text{Div}(m)$.
\end{notation}

\begin{definition}
An element $\Delta$ of a monoid $M$ is a \textbf{Garside element} for $M$ if the following hold:\\
(i) $\text{Div}_L(\Delta)=\text{Div}_R(\Delta)$,\\
(ii) $\langle \text{Div}(\Delta)\rangle=M$,\\
(iii) $\text{Div}(\Delta)$ is finite.\\
In this case, the elements of $\text{Div}(\Delta)$ are called \textbf{simples}.
\end{definition}

\begin{remark}
If $M$ admits a Garside element, it is finitely generated.
\end{remark}

We are ready to define a Garside monoid:

\begin{definition}\label{Garside}
A monoid $M$ is a \textbf{Garside monoid} if the following hold:\\
(i) Divisibility in $M$ is Noetherian,\\
(ii) $M$ is cancellative,\\
(iii) $(M,\leq_L)$ and $(M,\leq_R)$ are lattices,\\
(iv) $M$ admits a Garside element.
\end{definition}
\noindent
We now state some useful lemmas:

\begin{lemma}\label{CenterDiv}
Let $M$ be a left-cancellative monoid and let $u,v\in M$. If $uv\in Z(M)$, then $uv=vu$.
\end{lemma}

\begin{proof}
Since $uv$ is central, we have $uvu=uuv$, which implies $vu=uv$ by left-cancellativity of $M$.
\end{proof}

\begin{lemma}[{\cite[Lemma 2.3]{DehornoyGarside}}]\label{DeltaCentralPower}
In a left-cancellative monoid $M$, any Garside element $\Delta$ admits a central power.
\end{lemma}

%\begin{proof}
%First, observe that for any simple $u$, there exist a unique simple $\tilde u$ such that $u\tilde u=\Delta$. Existence of such a simple follows by definition of simples whereas uniqueness follows from left-cancellativity.\\
%From this, for all simple $u$, we have
%\begin{eqnarray*}
%u\Delta&=u(\tilde {u}\tilde {\tilde {u}})\\
%&=\Delta \tilde {\tilde {u}},
%\end{eqnarray*}
%thus right multiplication by $\Delta$ defines a permutation $\pi_{\Delta}$ of simples and $u\Delta^k=\Delta^k\pi_{\Delta}^k(u)$. Since $\text{Div}(\Delta)$ is finite, $\pi_{\Delta}$ is of finite order, which concludes the proof.
%\end{proof}

\begin{lemma}[{\cite[Lemma 2.4]{DehornoyGarside}}]\label{DeltaPowerGarside}
Let $M$ be a cancellative monoid with Noetherian divisibility and let $\Delta\in M$ be a Garside element. Then, every non trivial power of $\Delta$ is again a Garside element. 
\end{lemma}

\begin{remark}\label{RemarkDeltaCentral}
Combining Lemmas \ref{DeltaCentralPower} and \ref{DeltaPowerGarside}, we get that a finitely generated cancellative monoid $M$ with Noetherian divisibility admits a Garside element if and only if it admits a central Garside element.
\end{remark}

\begin{lemma}\label{Deltamultiples}
Let $M$ be a cancellative monoid with Noetherian divisibility admitting a Garside element $\Delta$. Then every pair $(a,b)$ of elements of $M$ admits common left- and right-multiples.
\end{lemma}

%\begin{proof}
%By Remark \ref{RemarkDeltaCentral}, up to considering a non trivial power of $\Delta$, one can assume that $\Delta$ is central.\\
%By assumption, $\text{Div}(\Delta)$ generate $M$, so that $b$ can be written as a product of simples $u_1\cdots u_n$. We then have
%\begin{eqnarray*}
%a\Delta^n&=&\Delta^na \, \, \text{since} \, \Delta \, \text{is central}\\
%&=&u_1\tilde {u_1}\Delta^{n-1}a\\
%&=&u_1\Delta^{n-1}\tilde{u_1}a\, \, \text{since} \, \Delta \, \text{is central}\\
%&=&u_1u_2\tilde{u_2}\Delta^{n-2}\tilde{u_1}a\\
%&\vdots&\\
%&=&u_1u_2\cdots u_n\tilde{u_n}\tilde{u_{n-1}}\cdots \tilde{u_1}a\\
%&=&b\tilde{u_n}\tilde{u_{n-1}}\cdots \tilde{u_1}a,
%\end{eqnarray*}
%which proves that $a\Delta^n$ is a common right-multiple of $a$ and $b$. Now, since $u_1\cdots u_n\tilde{u_n}\cdots \tilde{u_1}=\Delta^n$ is central, by Lemma \ref{CenterDiv} we have $\Delta^n=\tilde{u_n}\cdots \tilde{u_1}u_1\cdots u_n$ so that $\Delta^na=a\Delta^n=a\tilde{u_n}\cdots \tilde{u_1}b$ is also a common left-multiple of $a$ and $b$. This concludes the proof.
%\end{proof}

\begin{definition}
Let $M$ be a monoid such that $(M,\leq_L)$ (respectively $(M,\leq_R)$) is a poset. We say that $M$ admits \textbf{conditional left-lcms} (respectively \textbf{right-lcms}) if any two elements that admit a common left-multiple (respectively right-multiple) admit a left- (respectively right-) lcm. The definition for conditional gcds is similar.
\end{definition}

\begin{lemma}\label{LcmGcd}
Let $M$ be a cancellative monoid admitting conditional right-lcms (respectively left-lcms). Then any pair of elements of $M$ admitting a common  left-multiple (respectively right-multiple) admits a right-gcd (respectively left-gcd). 
\end{lemma}

\noindent
We have a criterion for presented monoids with Noetherian divisibility to be  cancellative and to admit conditional left-lcms. In order to state the result, we need to introduce some definitions.

\begin{definition}\label{Complemented}
A monoid presentation $\langle S\,|\, R\rangle$ is \textbf{right-complemented} (respectively \textbf{right-full}) if for all pair of distinct elements $\{s,t\}$ of $S$, there is at most one (respectively exactly one) relation of the form $su=tv$ with $u,v\in S^*$ and if there are no relations of the form $su=sv$ with $s\in S$, $u,v\in S^*$.
\end{definition}

\begin{definition}\label{Theta}
The \textbf{syntactic complement} of a right-complemented monoid presentation $\langle S\,|\, R\rangle$ is the (unique) partial map $\theta:S\times S\to S^*$ such that $(s,t)$ is in the domain of $\theta$ if and only if $R$ contains a relation of the form $su=tv$ with $u,v\in S^*$, in which case $\theta(s,t)=u$ and $\theta(t,s)=v$.
\end{definition}

\begin{lemma}\label{ThetaExtension}
Given a right-complemented monoid presentation $\langle S\,|\, R\rangle$, the map $\theta$ admits a unique minimal extension to a partial map (which we still denote $\theta$ by abuse of notation) $\theta:S^*\times S^*\to S^*$ satisfying the following equalities when both sides are well-defined:\\
(i) $\theta(s,s)=1 \, \forall s\in S$,\\
(ii) $\theta(ab,c)=\theta(b,\theta(a,c))\, \forall a,b,c\in S^*$,\\
(iii) $\theta(a,bc)=\theta(a,b)\theta(\theta(b,a),c)\, \forall a,b,c\in S^*$,\\
(iv) $\theta(1,a)=a$ and $\theta(a,1)=1$ $\forall a\in S^*$.
\end{lemma}

\begin{definition}
Let $\langle S\, | \, R\rangle$ be a right-complemented presentation and let $a,b,c$ be pairwise distinct elements of $S$. The triple $(a,b,c)\in S^3$ is said to satisfy the  \textbf{$\theta$-cube condition} (respectively the \textbf{sharp $\theta$-cube condition}) if either $\cube abc$ and $\cube bac$ are both defined and equal in $M$ (respectively equal in $S^*$), or none is defined.\\
The presentation $\langle S\, | \, R \rangle$ is said to satisfy the $\theta$-cube condition (respectively the sharp $\theta$-cube condition) if all triples $(a,b,c)\in S^3$ of pairwise distinct elements satisfy the $\theta$-cube condition (respectively the sharp $\theta$-cube condition).
\end{definition}
\noindent
We are now ready to introduce the criterion:

\begin{proposition}\label{CubeCondition}
Let $M$ be a monoid with Noetherian divisibility admitting the right-complemented presentation $\langle S\, | \, R\rangle$. If $\langle S\, | \, R\rangle$ satisfies the $\theta$-cube condition, then $M$ is left-cancellative and admits conditional right-lcms. More precisely, the elements $u$ and $v$ admit a common right-multiple if and only if $\theta(u,v)$ exists, in which case $u\theta(u,v)=v\theta(v,u)$ represents the right-lcm of $u$ and $v$.
\end{proposition}

\subsection{Practical framework for some homogeneous presentations}
\noindent
In this section, we prove some results that will simplify the proofs of later sections about Garside monoids.

\begin{definition}
A monoid presentation $\langle S\, | \, R\rangle$ is defined to be \textbf{homogeneous} if there exists a function $\lambda:S\to \N^*$ such that for all $(s_{i_1}\cdots s_{i_n},s_{j_1}\cdots s_{j_m})\in R$ with $\{s_{i_1},\dots,s_{i_n},s_{j_1},\dots,s_{j_m}\}\subset S$, we have $\sum_{k\in [n]}\lambda(s_{i_k})=\sum_{k\in [m]}\lambda(s_{j_k})$. The function $\lambda$ is called the \textbf{weight function}.
\end{definition}
\begin{definition}\label{Magic}
A monoid presentation $\langle S\,|\, R\rangle$ satisfies the property $(C1)$\label{Poulet} if it is right-full and for all $A\subset S$ of cardinal 3, there exists a labeling $\{a_1,a_2,a_3\}$ of $A$ and a word $u\in S^*$ (depending on $A$) such that the following word equalities hold: 
\begin{equation*}
    \theta(a_1,a_3)=\theta(a_1,a_2)u, \,  \theta(a_2,a_3)=\theta(a_2,a_1)u,\, \theta(a_3,a_1)=\theta(a_3,a_2).
\end{equation*}
\end{definition}

\begin{lemma}\label{MagicCube}
Assume that the presentation $\langle S\,|\, R\rangle$ satisfies $(\hyperref[Poulet]{C1})$. Then $\langle S\,|\, R \rangle$ satisfies the sharp $\theta-$cube condition.\\
In particular, if the monoid $M$ with presentation $\langle S\,|\, R \rangle$ has Noetherian divisibility, it is left-cancellative and admits conditional right-lcms.
\end{lemma}

\begin{proof}
Let $\{a,b,c\}\subset S$ be of cardinal 3 and assume without loss of generality that there exists $u\in S^*$ such that $$\theta(a,c)=\theta(a,b)u, \,  \theta(b,c)=\theta(b,a)u,\, \theta(c,a)=\theta(c,b).$$
We then have
$$\cube abc=\theta(\theta(a,b),\theta(a,b)u)=u=\theta(\theta(b,a),\theta(b,a)u)=\cube bac,$$
which shows that $(a,b,c)$ and $(b,a,c)$ satisfy the sharp $\theta$-cube condition. Moreover, we have
$$\cube bca=\theta(\theta(b,a)u,\theta(b,a))=1=\cube cba,$$
hence $(b,c,a)$ and $(c,b,a)$ satisfy the sharp $\theta$-cube condition. Finally, we have
$$\cube acb=\theta(\theta(a,b)u,\theta(a,b))=1=\cube cab,$$
so that $(a,c,b)$ and $(c,a,b)$ satisfy the sharp $\theta$-cube condition. The choice of the subset of cardinal 3 of $S$ being arbitrary, this concludes the proof.
\end{proof}

\noindent
In fact, if a monoid admits a presentation satisfying $(\hyperref[Poulet]{C1})$, under an additional assumption some of its quotients satisfy $(\hyperref[Poulet]{C1})$ as well. To see this, we introduce the following notation:
\begin{notation}
Let $S$ be a set, $w\in S^*$, $X\subset S$. Denote by $w/X$ the word obtained from $w$ after removing every instance of elements of $X$. With this notation, if a monoid $M$ admits the presentation $\langle S\,|\,R\rangle$, then for every $X\subset S$, the monoid $M/\langle x=1\, \forall x\in X\rangle$ admits the presentation $$\langle S\backslash X\,|\, r_1/X=r_2/X,\, \forall (r_1,r_2)\in R\rangle.$$
We denote this presentation by $\langle S\,|\,R\rangle/X$. For $m\in M$, we write $m/X$ for the image of $m$ by the morphism $M\to M/\langle x=1\, \forall x\in X\rangle$. If $u\in S^*$ represents $m$, then $u/X$ represents $m/X$.\\
Moreover, if $\langle S\,|\,R\rangle$ is right-full with syntactic right-complement $\theta$, for every $X\subset S$, we denote by $\langle S\,|\, R\rangle_X$ the presentation $$\langle S\backslash X\,|\, u(\theta(u,v)/X)=v(\theta(v,u)/X), \, \forall u,v\in S\backslash X, u\neq v\rangle.$$
Finally, if $X=\{x\}$, we simply denote $\langle S\, | \, R \rangle/X$ by $\langle S\, | \, R \rangle/x$ and $\langle S\, | \, R \rangle_X$ by $\langle S \, | \, R \rangle_x$.
\end{notation}

\begin{remark}
Let $M$ be the monoid with presentation 
$$\langle a,b \, | \, aba=b^2 \rangle.$$
Then $M/\langle a=1\rangle$ has presentation $\langle b\, | \, b=b^2\rangle$, which is a presentation for a monoid with two elements, $1$ and $b$, whereas $$\langle a,b\, | \, aba=b^2\rangle_a=\langle b\, | \, \emptyset\rangle,$$
which is a presentation for $\N$. In particular, in general $\langle S\, | \, R\rangle/X$ is not isomorphic to $\langle S\, | \, R \rangle_X$, but merely a quotient.
\end{remark}

\begin{lemma}\label{MagicRemark1}
Let $\langle S\,|\,R\rangle$ be a monoid presentation satisfying $(\hyperref[Poulet]{C1})$ and let $X\subset S$. Then $\langle S\,|\,R\rangle_X$ satisfies $(\hyperref[Poulet]{C1})$ as well.
\end{lemma}

\begin{proof}
First, observe that $\langle S\, | \, R\rangle_X$ is right-full. Indeed, for all distinct $t,u\in S\backslash X$, there is exactly one relation of the form $t\cdots=u\cdots$ in $\langle S\, | \, R\rangle_X$, namely $t\theta(t,u)/X=u\theta(u,t)/X$.\\ 
Now, the operation sending $v\in S^*$ to $v/X$ is multiplicative and writing $\theta_X$ for the syntactic complement of $\langle S\,|\,R\rangle_X$, for any distinct $t,u\in S\backslash X$ we have $\theta_X(t,u)=\theta(t,u)/X$. Let then $\{a,b,c\}$ be a triple of distinct elements of $S\backslash X$ and assume without loss of generality that there exists $v\in S^*$ such that $$\theta(a,c)=\theta(a,b)v,\, \theta(b,c)=\theta(b,a)v,\, \theta(c,a)=\theta(c,b).$$
Then by the above discussion we have
$$\theta_X(a,c)=\theta_X(a,b)(v/X), \, \theta_X(b,c)=\theta_X(b,a)(v/X), \, \theta_X(c,a)=\theta_X(c,b),$$
which concludes the proof.
\end{proof}

\begin{remark}\label{MagicNoetherian}
If $M$ is monoid admitting a homogeneous presentation $\langle S\,|\,R\rangle$ with weight function $\lambda$ and if $X\subset S$ is such that $\sum_{x\in X}\lambda(x)|r_1|_x=\sum_{x\in X}\lambda(x)|r_2|_x$ for all $(r_1,r_2)\in R$, the presentation $\langle S\, |\, R\rangle/X$ is homogeneous as well, so that $M/\langle x=1 \forall x\in X\rangle$ has Noetherian divisibility.
\end{remark}

\begin{lemma}\label{MagicOpQuotient}
Let $M$ be a monoid with presentation $\langle S\, |\, R\rangle$. Then for all $X\subset S$, the monoids
$(M/\langle x=1\, \forall x\in X\rangle)^{\mathrm{op}}$ and $M^{\mathrm{op}}/\langle x=1 \, \forall x\in X\rangle$ are isomorphic.
\end{lemma}
\begin{proof}
It follows from the fact that for all $w\in S^*$ we have $(w/X)^{\mathrm{op}}=w^{\mathrm{op}}/X$, where $(s_1\cdots s_n)^{\mathrm{op}}$ is defined as $s_n\cdots s_1$.
\end{proof}

\begin{definition}
Let $P=\langle S\,|\,R\rangle$ be a monoid presentation. A pair $(P,X)$ with $X\subset S$ is said to satisfy the property $(C2)$\label{Champagne} if\\
$\bullet$ The presentation $P$ is homogeneous and satisfies $(\hyperref[Poulet]{C1})$.\\
$\bullet$ For all $(r_1,r_2)\in R$, we have $\sum_{x\in X}\lambda(x)|r_1|_x=\sum_{x\in X}\lambda(x)|r_2|_x$ , where $\lambda$ is the weight function of $P$.\\
$\bullet$ The monoids with presentations $\langle S\,|\,R\rangle_X$ and $\langle S\,|\,R\rangle /X$ are isomorphic.\\
Finally, if $X=\{x\}$, we simply say that $(P,x)$ satisfies $(\hyperref[Champagne]{C2})$. In this case, checking the second condition reduces to verifying that $|r_1|_x=|r_2|_x$ for all $(r_1,r_2)\in R$.
\end{definition}

\begin{lemma}\label{MagicQuotient}
Let $M$ be a monoid with presentation $P=\langle S\, | \, R\rangle$, let $X\subset S$ and assume that $(P,X)$ satisfies $(\hyperref[Champagne]{C2})$.
Then the monoid $M/\langle x=1\, \forall x\in X\rangle$ is left-cancellative and admits conditional right-lcms.
\end{lemma}

\begin{proof}
By Remark \ref{MagicNoetherian}, the monoid $M/\langle x=1 \, \forall x\in X\rangle$ has Noetherian divisibility, hence so does $\langle S\, |\, R\rangle_X$ by assumption. Combining Lemmas \ref{MagicCube} and \ref{MagicRemark1} now concludes the proof.
\end{proof}

\begin{lemma}\label{QuotientGarsideElement}
Let $M$ be a monoid with presentation $P=\langle S\, | \, R \rangle$, let $X\subset S$ and assume that $(P,X)$ satisfies $(\hyperref[Champagne]{C2})$. If $\Delta\in M$ is a central Garside element of $M$, then $\Delta/X$ is a central Garside element of $M/\langle x=1 \, \forall x\in X\rangle$.
\end{lemma}

\begin{proof}
First, Lemma \ref{MagicQuotient} implies that $M/\langle x=1\, \forall x\in X\rangle$ is left-cancellative.
Therefore, Lemma \ref{CenterDiv} applies so that the left- and right- divisors of $\Delta/X$ coincide since $\Delta/X$ is central. Moreover, by Remark \ref{MagicNoetherian}, there are only finitely many divisors of $\Delta/X$ since $M/\langle x=1 \, \forall x\in X\rangle$ has Noetherian divisibility and is finitely generated. Finally, the divisors of $\Delta$ generate $M$ and if $u$ divides $\Delta$ then $u/X$ divides $\Delta/X$, thus the divisors of $\Delta/X$ generate $M/\langle x=1\, \forall x\in X\rangle$. This concludes the proof.
\end{proof}

\begin{remark}
In Lemma \ref{QuotientGarsideElement}, we make the assumption that the Garside element is central in order to simplify the proof. It may be true that the result still holds for non-central Garside elements, though we were not able to prove it. Note however that it is not important for our purpose, since if there exists a Garside element, we can always choose a central Garside element (see Remark \ref{RemarkDeltaCentral}).
\end{remark}

%\begin{proposition}\label{MagicGarside}
%Let $M$ be a Garside monoid. If $X\subset M$ is such that the pair $(M,X)$ satisfies $(\hyperref[Champagne]{C2})$ and $(\hyperref[Champagne]{C2})^{\mathrm{op}}$, then the monoid $M/\langle x=1\, \forall x\in X\rangle$ is Garside.
%\end{proposition}

%\begin{proof}
%First, by Remark \ref{MagicNoetherian}, the monoid $M/\langle x=1\, \forall x\in X\rangle$ has Noetherian divisibility.\\
%Moreover, applying Lemma \ref{MagicQuotient} to $M$ and $M^{\mathrm{op}}
%$ shows that $M/\langle x=1\, \forall x\in X\rangle$ is cancellative and admits conditional left- and right-lcms, since $M^{\mathrm{op}}/\langle x=1\, \forall x\in X\rangle$ is isomorphic to $(M/\langle x=1\, \forall x\in X\rangle)^{\mathrm{op}}$ by Lemma \ref{MagicOpQuotient}.

%Now, combining Remark \ref{RemarkDeltaCentral} and Lemma \ref{QuotientGarsideElement}, the monoid $M/\langle x=1\, \forall x\in X\rangle$ admits a Garside element. Finally, combining Lemmas \ref{Deltamultiples} and \ref{LcmGcd}, any pair of elements of $M$ admits left- and right- gcds and lcms. This concludes the proof.
%\end{proof}

\begin{proposition}\label{MagicGarsideOp}
Let $M$ be a Garside monoid with presentation $P=\langle S \, | \, R \rangle$, let $X\subset S$ and assume that the pair $(P,X)$ satisfies $(\hyperref[Champagne]{C2})$. If the monoid $M/\langle x=1\, \forall x\in X\rangle$ is isomorphic to $(M/\langle x=1\, \forall x\in X\rangle)^\mathrm{op}$, it is Garside.
\end{proposition}

\begin{proof}
First, by Remark \ref{MagicNoetherian}, the monoid $M/\langle x=1\, \forall x\in X\rangle$ has Noetherian divisibility.\\
Moreover, applying Lemma \ref{MagicQuotient} to $M$ shows that $M/\langle  x=1\, \forall x\in X\rangle$ is left-cancellative and admits conditional right-lcms. Since $M/\langle  x=1\, \forall x\in X\rangle$ is isomorphic to $(M/\langle x=1\, \forall x\in X\rangle)^\mathrm{op}$, we thus get that $M$ is cancellative and admits conditional left- and right-lcms.\\
Now, combining Remark \ref{RemarkDeltaCentral} and Lemma \ref{QuotientGarsideElement}, the monoid $M/\langle  x=1\, \forall x\in X\rangle$ admits a Garside element. Finally, combining Lemmas \ref{Deltamultiples} and \ref{LcmGcd}, any pair of elements of $M$ admits left- and right- gcds and lcms. This concludes the proof.
\end{proof}

%%%%%%%%%%%%%%%%%%%%%%%%%%%%%%%%%%%%%%%%%%%%%%%%%%%%%%%%%%%%%%%%%%%%%%%%%%%%%%%%%%%%%%%%%%%%%%%%%%%%%%%
%%%%%%%%%%%%%%%%%%%%%%%%%%%%%%%%%%%%%%%%%%%%%%%%%%%%%%%%%%%%%%%%%%%%%%%%%%%%%%%%%%%%%%%%%%%%%%%%%%%%%%%
%%%%%%%%%%%%%%%%%%%%%%%%%%%%%%%%%%%%%%%%%%%%%%%%%%%%%%%%%%%%%%%%%%%%%%%%%%%%%%%%%%%%%%%%%%%%%%%%%%%%%%%
%%%%%%%%%%%%%%%%%%%%%%%%%%%%%%%%%%%%%%%%%%%%%%%%%%%%%%%%%%%%%%%%%%%%%%%%%%%%%%%%%%%%%%%%%%%%%%%%%%%%%%%
%%%%%%%%%%%%%%%%%%%%%%%%%%%%%%%%%%%%%%%%%%%%%%%%%%%%%%%%%%%%%%%%%%%%%%%%%%%%%%%%%%%%%%%%%%%%%%%%%%%%%%%%%%%%%%%%%%%%%%%%%%%%%%%%
%%%%%%%%%%%%%%%%%%%%%%%%%%%%%%%%%%%%%%%%%%%%%%%%%%%%%%%%%%%%%%%%%%%%%%%%%%%%%%%%%%%%%%%%%%%%%%%%%%%%%%%

\section{\texorpdfstring{The Classical braid monoids are Garside}{}}
\noindent
In this section, when we say that two sets of relations are equivalent, we mean that they are equivalent as \textit{monoid relations} instead of merely group relations.\\
The goal of this section is prove the following theorem:

\begin{theorem}\label{ClassicalGarside}
Let $n,m\in\N^*$ and assume that $n\wedge m=1$ and $n\leq m$. Then the following hold:\\
(i) The monoid $\M_*^*(n,m)$ is Garside.\\
(ii) If $m\geq 2$, the monoid $\M_*(n,m)$ is Garside.\\
(iii) If $n\geq 2$, the monoid $\M^*(n,m)$ is Garside.
\end{theorem}

The proof strategy goes as follows: we show that $\M_*^*(n,m)$ is cancellative and admits conditional left- and right-lcms (Propositions \ref{ClassicalLeftCanc} and \ref{OpLeftCanc}), then we exhibit a Garside element (Lemma \ref{DeltaGarsideM}) and conclude the proof using results of Section \ref{SecGarside}.
%%%%%%%%%%%%%%%%%%%%%%%%%%%%%%%%%%%%%%%%%%%%%%%%%%%%%%%%%%%%%%%%%%%%%%%%%%%%%%%%%%%%%%%%%%%%%%%%%%%%%%%
%%%%%%%%%%%%%%%%%%%%%%%%%%%%%%%%%%%%%%%%%%%%%%%%%%%%%%%%%%%%%%%%%%%%%%%%%%%%%%%%%%%%%%%%%%%%%%%%%%%%%%%
%%%%%%%%%%%%%%%%%%%%%%%%%%%%%%%%%%%%%%%%%%%%%%%%%%%%%%%%%%%%%%%%%%%%%%%%%%%%%%%%%%%%%%%%%%%%%%%%%%%%%%%
%%%%%%%%%%%%%%%%%%%%%%%%%%%%%%%%%%%%%%%%%%%%%%%%%%%%%%%%%%%%%%%%%%%%%%%%%%%%%%%%%%%%%%%%%%%%%%%%%%%%%%%

\subsection{\texorpdfstring{The monoid $\M_*^*(n,m)$ is left-cancellative}{}}

\begin{proposition}\label{ClassicalLeftCanc}
Let $n,m\in\N^*$ with $n\wedge m=1$ and $n\leq m$. The monoid $\M_*^*(n,m)$ is left-cancellative and admits conditional right-lcms.
\end{proposition}
In order to be able to use results of Section \ref{SecGarside}, we enlarge the monoid presentation \eqref{classicalBraidPres1} to get a right-full homogeneous presentation of $\M_*^*(n,m)$. \\

The set of relations \eqref{classicalBraidPres1:2} is equivalent to 
\begin{equation}\label{L2Classical}x_i\cdots x_{n}yz\delta^{q-1}x_1\cdots x_{i+r-1}=x_j\cdots x_{n}yz\delta^{q-1}x_1\cdots x_{j+r-1}, \, \forall 1\leq i<j\leq n-r+1,
\end{equation}
and the set of relations \eqref{classicalBraidPres1:3} is equivalent to 
\begin{equation}\label{L3Classical}x_i\cdots x_nyz\delta^qx_1\cdots x_{i+r-n-1}=x_j\cdots x_ny\delta^qx_1\cdots x_{j+r-n-1},\, \forall n-r+1\leq i<j\leq n+1.
\end{equation}
In addition to Relations \eqref{classicalBraidPres1:1},\eqref{L2Classical} and \eqref{L3Classical}, observe the following:\\
$\bullet$ For all $i\in [n-r+1]$, $j\in\llbracket n-r+2,n+1\rrbracket$, we have
\begin{equation}\label{EnlargeClassicaL1}
\begin{aligned}
x_i\cdots x_nyz\delta^{q-1}x_1\cdots x_{i+r-1}y&\underset{\eqref{L2Classical}}{=}x_{n-r+1}\cdots x_nyz\delta^{q-1}x_1\cdots x_ny \,\\
&=x_{n-r+1}\cdots x_nyz\delta^q
\underset{\eqref{L3Classical}}{=}x_j\cdots x_nyz\delta^q x_1\cdots x_{j+r-n-1}.
\end{aligned}
\end{equation}
Note that for $j=n+1$, Relation \eqref{EnlargeClassicaL1} is $x_i\cdots x_nyz\delta^{q-1}x_1\cdots x_{i+r-1}y=yz\delta^qx_1\cdots x_r$.\\
\noindent
$\bullet$ For all $i\in \llbracket 2,n-r+1\rrbracket$, we have
\begin{equation}\label{EnlargeClassical2}
\begin{aligned}
x_i\cdots x_nyz\delta^{q-1}x_1\cdots x_{i+r-1}&\underset{\eqref{L2Classical}}{=}x_{1}\cdots x_nyz\delta^{q-1}x_1\cdots x_r \underset{\eqref{classicalBraidPres1:1}}{=}z\delta^qx_1\cdots x_r.
\end{aligned}
\end{equation}
$\bullet$ For all $i\in \llbracket n-r+2,n+1\rrbracket$, we have
\begin{equation}\label{EnlargeClassical3}
\begin{aligned}
x_i\cdots x_nyz\delta^{q}x_1\cdots x_{i+r-n-1}&\underset{\eqref{L3Classical}}{=}x_{n-r+1}\cdots x_nyz\delta^q \\
=(x_{n-r+1}\cdots x_nyz\delta^{q-1}&x_1\cdots x_n)y
\underset{\eqref{classicalBraidPres1:1}+\eqref{L2Classical}}{=}z\delta^qx_1\cdots x_ry.
\end{aligned}
\end{equation}
Note that for $i=n+1$, Relation \eqref{EnlargeClassical3} is $yz\delta^qx_1\cdots x_r=z\delta^qx_1\cdots x_ry$.\\
Finally, we have $x_{n-r+1}\cdots x_nyz\delta^q=(x_{n-r+1}\cdots x_nyz\delta^{q-1}x_1\cdots x_n)y$, hence the set of relations given by \eqref{L3Classical} with $i=n-r+1$, $j=n-r+2,\dots,n+1$ is equivalent to the set of relations given by \eqref{EnlargeClassicaL1} with $i=n-r+1$, $j=n-r+2,\dots,n+1$. Therefore, if we add relations \eqref{EnlargeClassicaL1}, we can replace \eqref{L3Classical} by the set of relations
\begin{equation}\label{EnlargeClassical4}
x_i\cdots x_nyz\delta^qx_1\cdots x_{i+r-n-1}=x_j\cdots x_nyz\delta^qx_1\cdots x_{j+r-n-1},  \, \forall n-r+2\leq i<j\leq n+1.
\end{equation}
Using equations \eqref{EnlargeClassicaL1}-\eqref{EnlargeClassical4}, we get the following corollary:
\begin{corollary}\label{ClassicalMonoid2}
Let $n,m\in \N^*$ and assume that $n\leq m$ and $n\wedge m=1$. The monoid $\mathcal M_*^*(n,m)$ admits the following presentation:

\begin{subequations}\label{ClassicalMonoidPres2}
\begin{align}
    &(1) \,\, \mathrm{Generators}\!:\,  \{x_1,\dots,x_n,y,z\};\notag\\
&(2) \,\, \mathrm{Relations}\!: \,\notag\\
  &zx_1\cdots x_ny=x_1\cdots x_nyz,  \label{ClassicalMonoidPres2:1} \\
  &x_{i}\cdots x_{n}yz\delta^{q-1}x_1\cdots x_{i+r-1}=x_j\cdots x_{n}yz\delta^{q-1}x_1\cdots x_{j+r-1}, \, \forall \,1\leq i<j\leq n-r+1, \label{ClassicalMonoidPres2:2}  \\
       &x_{i}\cdots x_{n}yz\delta^{q}x_1\cdots x_{i+r-n-1}=x_j\cdots x_{n}yz\delta^{q}x_1\cdots x_{j+r-n-1}, \, (*),\label{ClassicalMonoidPres2:3} \\
      &x_i\cdots x_nyz\delta^{q-1}x_1\cdots x_{i+r-1}y=x_j\cdots x_{n}yz\delta^{q}x_1\cdots x_{j+r-n-1},\, (**), \label{ClassicalMonoidPres2:4}\\
      &x_{i}\cdots x_{n}yz\delta^{q-1}x_1\cdots x_{i+r-1}=z\delta^qx_1\cdots x_r, \,\forall \, 2\leq i\leq n-r+1,\label{ClassicalMonoidPres2:5}\\
      &x_{i}\cdots x_{n}yz\delta^{q}x_1\cdots x_{i+r-n-1}=z\delta^qx_1\cdots x_ry, \,\forall n-r+2\leq i \leq n+1.\label{ClassicalMonoidPres2:6}\\
      &(*)\, \forall \, n-r+2\leq i<j\leq n+1, \, (**) \,\forall \,i\in[n-r+1],\, j\in \llbracket n-r+2,n+1\rrbracket\notag
\end{align}
\end{subequations}
where $\delta$ denotes $x_1\cdots x_ny$. 
\end{corollary}
\noindent
Now, one sees that Presentation \eqref{ClassicalMonoidPres2} is right-full and homogeneous, hence by Lemma \ref{MagicCube} it is enough to verify that is satisfies $(\hyperref[Poulet]{C1})$ to conclude that $\mathcal M_*^*(n,m)$ is left-cancellative and admits conditional right-lcms.

\begin{remark}\label{EmptyConditions}
In Presentation \eqref{ClassicalMonoidPres2}, it is possible that some conditions on the indices are empty. In this case, there is simply no corresponding relations. For instance, if $n=m=1$, the set of relations \eqref{ClassicalMonoidPres2:6} is empty. 
\end{remark}

\begin{lemma}\label{CubeClassical}
Presentation \eqref{ClassicalMonoidPres2} satisfies $(\hyperref[Poulet]{C1})$.
\end{lemma}
\begin{proof}
First, define $x_i$ to be \textit{small} if $i=2,\dots,n-r+1$ and $x_i$ to be \textit{large} if $i=n-r+2,\dots,n$. In order to increase the clarity of the proof we define the following words, where $\delta$ denotes $x_1\dots x_ny$:
\begin{equation*}
\begin{aligned}
w_{x_1}&:=\delta^{q-1}x_1\cdots x_r\\
w_{x_i}&:=x_{i+1}\cdots x_nyz\delta^{q-1}x_1\cdots x_{i+r-1} \, \, \text{for} \, x_i \, \text{small} \\
w_{x_i}&:=x_{i+1}\cdots x_nyz\delta^qx_1\cdots x_{i+r-n-1}\, \, \text{for}\, x_i \, \text{large}\\
w_y&:=z\delta^qx_1\cdots x_r\\
w_z&:=\delta^{q-1}x_1\cdots x_r=w_{x_1}.
\end{aligned}
\end{equation*}
In the following table, for $k,l\geq 2$ the entry in the $k-$th line and $l-$th column is $\theta(a_k,a_l)$, where $a_k$ is the $k-$th value of the first column and $a_l$ is the $l-$th value of the first line. Moreover, if $k=l$ and $a_k=x_i$ with $x_i$ small (respectively large), the entry is $\theta(x_i,x_j)$ with $x_i,x_j$ small (respectively large) and $i\neq j$. Finally, if $k=l$ and $a_k\in \{x_1,y,z\}$, we simply put a backslash to acknowledge that the entry is not relevant.
\begin{table}[hbt]
\centering
\begin{tabular}{|c|l|c|c|c|c|}
\hline
& $x_1$ &$x_j$ small & $x_j$ large & $y$ & $z$\\
\hline
$x_1$&$\backslash$&$x_2\cdots x_nyzw_{x_1} $&$x_2\cdots x_nyzw_{x_1}y$  &$x_2\cdots x_nyzw_{x_1}y$  &$x_2\cdots x_nyz$
\\
\hline
$x_i$ small&$w_{x_i}$&$w_{x_i}$ &$w_{x_i}y$ &$w_{x_i}y$  &$w_{x_i}$\\
\hline
$x_i$ large&$w_{x_i}$&$w_{x_i}$ &$w_{x_i}$ & $w_{x_i}$ &$w_{x_i}$ 
\\
\hline
$y$&$w_y$& $w_y$ & $w_y$ &$\backslash$& $w_y$ 
\\
\hline
$z$ &$\delta$ & $\delta w_z$ & $ \delta w_zy$ & $\delta w_zy$ & $\backslash$  \\
\hline

\end{tabular}
\caption{Syntactic complement of Presentation \eqref{ClassicalMonoidPres2}}
\end{table}\\
\noindent
In order to prove that every subset $A$ of $\{x_1,\dots,x_n,y,z\}$ containing 3 elements satisfies the assumption of Definition \ref{Magic}, we separate cases depending on the number of small and large $x_k$'s $A$ contains.\\\\
\noindent
\fbox{$(I)$ $A$ contains at most $1$ small and $1$ large $x_k$:} In this case, there are $\binom53=10$ subcases to verify.\\\\
\noindent
\underline{$(I.1)$ $A=\{x_1,x_i,x_j\}$ with $x_i$ small and $x_j$ large:} In this case, a possible ordering to apply Definition \ref{Magic} is
$(a_1,a_2,a_3)=(x_1,x_i,x_j)$ and we have
$$\theta(x_1,x_j)=\theta(x_1,x_i)y,\, \theta(x_i,x_j)=\theta(x_i,x_1)y,\, \theta(x_j,x_1)=\theta(x_j,x_i).$$
\underline{$(I.2)$ $A=\{x_1,x_i,y\}$ with $x_i$ small:} In this case, a possible ordering to apply Definition \ref{Magic} is $(a_1,a_2,a_3)=(x_1,x_i,y)$ and we have
$$\theta(x_1,y)=\theta(x_1,x_i)y,\, \theta(x_i,y)=\theta(x_i,x_1)y,\, \theta(y,x_1)=\theta(y,x_i).$$
\underline{$(I.3)$ $A=\{x_1,x_i,z\}$ with $x_i$ small:}  In this case, a possible ordering to apply Definition \ref{Magic} is $(a_1,a_2,a_3)=(x_1,z,x_i)$ and we have
$$\theta(x_1,x_i)=\theta(x_1,z)w_{x_1},\, \theta(z,x_i)=\theta(z,x_1)w_z=\theta(z,x_1)w_{x_1},\, \theta(x_i,x_1)=\theta(x_i,z).$$
\underline{$(I.4)$ $A=\{x_1,x_j,y\}$ with $x_j$ large:}  In this case, a possible ordering to apply Definition \ref{Magic} is $(a_1,a_2,a_3)=(x_1,x_j,y)$ and we have
$$\theta(x_1,y)=\theta(x_1,x_j),\, \theta(x_j,y)=\theta(x_j,x_1),\, \theta(y,x_1)=\theta(y,x_j).$$
\underline{$(I.5)$ $A=\{x_1,x_j,z\}$ with $x_j$ large:} In this case, a possible ordering to apply Definition \ref{Magic} is $(a_1,a_2,a_3)=(x_1,z,x_j)$ and we have
$$\theta(x_1,x_j)=\theta(x_1,z)w_{x_1}y,\, \theta(z,x_j)=\theta(z,x_1)w_zy=\theta(z,x_1)w_{x_1}y,\, \theta(x_j,x_1)=\theta(x_j,z).$$
\underline{$(I.6)$ $A=\{x_1,y,z\}$:} In this case, a possible ordering to apply Definition \ref{Magic} is $(a_1,a_2,a_3)=(x_1,z,y)$ and we have
$$\theta(x_1,y)=\theta(x_1,z)w_{x_1}y,\, \theta(z,y)=\theta(z,x_1)w_zy=\theta(z,x_1)w_{x_1}y,\, \theta(y,x_1)=\theta(y,z).$$
\underline{$(I.7)$ $A=\{x_i,x_j,y\}$ with $x_i$ small and $x_j$ large:} In this case, a possible ordering to apply Definition \ref{Magic} is\\
$(a_1,a_2,a_3)=(x_i,x_j,y)$ and we have
$$\theta(x_i,y)=\theta(x_i,x_j),\, \theta(x_j,y)=\theta(x_j,x_i),\, \theta(y,x_i)=\theta(y,x_j).$$
\underline{$(I.8)$ $A=\{x_i,x_j,z\}$ with $x_i$ small and $x_j$ large:} In this case, a possible ordering to apply Definition \ref{Magic} is\\
$(a_1,a_2,a_3)=(x_i,z,x_j)$ and we have
$$\theta(x_i,x_j)=\theta(x_i,z)y,\, \theta(z,x_j)=\theta(z,x_i)y,\, \theta(x_j,x_i)=\theta(x_j,z).$$
\underline{$(I.9)$ $A=\{x_i,y,z\}$ with $x_i$ small:} In this case, a possible ordering to apply Definition \ref{Magic} is
$(a_1,a_2,a_3)=(x_i,z,y)$ and we have
$$\theta(x_i,y)=\theta(x_i,z)y,\, \theta(z,y)=\theta(z,x_i)y,\, \theta(y,x_i)=\theta(y,z).$$
\underline{$(I.10)$ $A=\{x_j,y,z\}$ with $x_j$ large:} In this case, a possible ordering to apply Definition \ref{Magic} is
$(a_1,a_2,a_3)=(z,x_j,y)$ and we have
$$\theta(z,y)=\theta(z,x_j),\, \theta(x_j,y)=\theta(x_j,z),\, \theta(y,z)=\theta(y,x_j).$$
\noindent
\fbox{$(II)$ $A$ contains exactly two small $x_k$'s:} In this case, there are $\binom 41=4$ subcases to verify.\\\\
\noindent
\underline{$(II.1)$ $A=\{x_1,x_i,x_j\}$ with $x_i,x_j$ small:} In this case, a possible ordering to apply Definition \ref{Magic} is $(a_1,a_2,a_3)=(x_1,x_i,x_j)$ and we have
$$\theta(x_1,x_j)=\theta(x_1,x_i),\, \theta(x_i,x_j)=\theta(x_i,x_1),\, \theta(x_j,x_1)=\theta(x_j,x_i).$$
\underline{$(II.2)$ $A=\{x_i,x_j,x_k\}$ with $x_i,x_j$ small and $x_k$ large:} In this case, a possible ordering to apply Definition \ref{Magic} is $(a_1,a_2,a_3)=(x_i,x_j,x_k)$ and we have
$$\theta(x_i,x_k)=\theta(x_i,x_j)y,\, \theta(x_j,x_k)=\theta(x_j,x_i)y,\, \theta(x_k,x_i)=\theta(x_k,x_j).$$
\underline{$(II.3)$ $A=\{x_i,x_j,y\}$ with $x_i,x_j$ small:} In this case, a possible ordering to apply Definition \ref{Magic} is $(a_1,a_2,a_3)=(x_i,x_j,y)$ and we have
$$\theta(x_i,y)=\theta(x_i,x_j)y,\, \theta(x_j,y)=\theta(x_j,x_i)y,\, \theta(y,x_i)=\theta(y,x_j).$$
\underline{$(II.4)$ $A=\{x_i,x_j,z\}$ with $x_i,x_j$ small:} In this case, a possible ordering to apply Definition \ref{Magic} is $(a_1,a_2,a_3)=(x_i,x_j,z)$ and we have
$$\theta(x_i,z)=\theta(x_i,x_j),\, \theta(x_j,z)=\theta(x_j,x_i),\, \theta(z,x_i)=\theta(z,x_j).$$
\noindent
\fbox{$(III)$ $A$ contains exactly two large $x_k$'s:} In this case, there are $\binom 41=4$ subcases to verify.\\\\
\noindent
\underline{$(III.1)$ $A=\{x_1,x_i,x_j\}$ with $x_i,x_j$ large:} In this case, a possible ordering to apply Definition \ref{Magic} is $(a_1,a_2,a_3)=(x_1,x_i,x_j)$ and we have
$$\theta(x_1,x_j)=\theta(x_1,x_i),\, \theta(x_i,x_j)=\theta(x_i,x_1),\, \theta(x_j,x_1)=\theta(x_j,x_i).$$
\underline{$(III.2)$ $A=\{x_k,x_i,x_j\}$ with $x_k$ small and $x_i,x_j$ large:} In this case, a possible ordering to apply Definition \ref{Magic} is $(a_1,a_2,a_3)=(x_k,x_i,x_j)$ and we have
$$\theta(x_k,x_j)=\theta(x_k,x_i),\, \theta(x_i,x_j)=\theta(x_i,x_k),\, \theta(x_j,x_k)=\theta(x_j,x_i).$$
\underline{$(III.3)$ $A=\{x_i,x_j,y\}$ with $x_i,x_j$ large:} In this case, a possible ordering to apply Definition \ref{Magic} is $(a_1,a_2,a_3)=(x_i,x_j,y)$ and we have
$$\theta(x_i,y)=\theta(x_i,x_j),\, \theta(x_j,y)=\theta(x_j,x_i),\, \theta(y,x_i)=\theta(y,x_j).$$
\underline{$(III.4)$ $A=\{x_i,x_j,z\}$ with $x_i,x_j$ large:} In this case, a possible ordering to apply Definition \ref{Magic} is $(a_1,a_2,a_3)=(z,x_i,x_j)$ and we have
$$\theta(z,x_j)=\theta(z,x_i),\, \theta(x_i,x_j)=\theta(x_i,z),\, \theta(x_j,z)=\theta(x_j,x_i).$$
\noindent
\fbox{$(IV)$ A contains exactly three small $x_k$'s.} In this case, there is $\binom 40=1$ subcase to verify.\\\\
\noindent
\underline{$(IV.1)$ $A=\{x_i,x_j,x_k\}$ with $x_i,x_j$ and $x_k$ small:} In this case, a possible ordering to apply Definition \ref{Magic} is \\ $(a_1,a_2,a_3)=(x_i,x_j,x_k)$ and we have
$$\theta(x_i,x_k)=\theta(x_i,x_j),\, \theta(x_j,x_k)=\theta(x_j,x_i),\, \theta(x_k,x_i)=\theta(x_k,x_j).$$
\noindent
\fbox{$(V)$ A contains exactly three large $x_k$'s.} In this case, there is $\binom 40=1$ subcase to verify.\\\\
\noindent
\underline{$(V.1)$ $A=\{x_i,x_j,x_k\}$ with $x_i,x_j$ and $x_k$ large:} In this case, a possible ordering to apply Definition \ref{Magic} is \\ $(a_1,a_2,a_3)=(x_i,x_j,x_k)$ and we have
$$\theta(x_i,x_k)=\theta(x_i,x_j),\, \theta(x_j,x_k)=\theta(x_j,x_i),\, \theta(x_k,x_i)=\theta(x_k,x_j).$$
This concludes the proof.
\end{proof}
\begin{proof}[Proof of Proposition \ref{ClassicalLeftCanc}]
Since Presentation \eqref{ClassicalMonoidPres2} is homogeneous, the divisibility of $\M_*^*(n,m)$ is Noetherian. The result thus follows by combining Lemma \ref{CubeClassical} with Lemma \ref{MagicCube}.
\end{proof}

\subsection{\texorpdfstring{The monoid $\M_*^*(n,m)$ is right-cancellative}{}}

\begin{proposition}\label{OpLeftCanc}
Let $n,m\in\N^*$ with $n\wedge m=1$ and $n\leq m$. The monoid $\M_*^*(n,m)$ is right-cancellative and admits conditional left-lcms.
\end{proposition}

To prove proposition \ref{OpLeftCanc}, we show the equivalent statement that $\mathcal M_*^*(n,m)^{\mathrm{op}}$ is left-cancellative and admits conditional right-lcms. In order to do so, we use the following convenient presentation of $\M_*^*(n,m)^{\mathrm{op}}$:

\begin{lemma}\label{OpClassical}
Let $n,m\in \N^*$ with $n\wedge m=1$ and $n\leq m$. The monoid $\mathcal M_*^*(n,m)^{\mathrm{op}}$ admits the following presentation:

\begin{subequations}\label{OpClassicalPres1}
\begin{align}
    &(1) \,\, \mathrm{Generators}\!:\,  \{x_1,\dots,x_n,y,z\}; &&  \notag \\
    &(2) \,\, \mathrm{Relations}\!: \, &&  \notag \\
    &zyx_1\cdots x_n=yx_1\cdots x_nz, && \label{OpClassicalPres1:relation1} \\ % Non numérotée grâce à \tag{}
    &x_i\cdots x_nz\delta^{q-1}yx_1\cdots x_{i+r-1}=x_j\cdots x_nz\delta^{q-1}yx_1\cdots x_{j+r-1}, \forall \, 1\leq i<j\leq n-r+1, && \label{OpClassicalPres1:relation2} \\
    &x_i\cdots x_n\delta z\delta^{q-1}yx_1\cdots x_{i+r-n-1}=x_j\cdots x_n\delta z\delta^{q-1}yx_1\cdots x_{j+r-n-1}, \, (*), && \label{OpClassicalPres1:relation3}  \\
    &(*)\,\forall n-r+1\leq i<j\leq n+1 && \notag
\end{align}
\end{subequations}
where $\delta$ denotes $yx_1\cdots x_n$.
\end{lemma}

\begin{proof}
Using Presentation \eqref{classicalBraidPres1}, we see that $\mathcal M_*^*(n,m)^{\mathrm{op}}$ admits the presentation 
\begin{equation}\label{OpClassicalPres2}
\begin{aligned}
    &(1) \,\, \mathrm{Generators}\!:\,  \{x_1,\dots,x_n,y,z\};\\
&(2) \,\, \mathrm{Relations}\!: \,\\
  &zyx_n\cdots x_1=yx_n\cdots x_1z,   \\
  & x_{i+r}\cdots x_{1}\delta^{q-1}zyx_n\cdots x_{i+1}=x_{i+r-1}\cdots x_{1}\delta^{q-1}zyx_n\cdots x_{i}, \,\forall \,  1\leq i\leq n-r,   \\
       & x_{i+r-n}\cdots x_{1}\delta^{q}zyx_n\cdots x_{i+1}=x_{i+r-n-1}\cdots x_{1}\delta^{q}zyx_n\cdots x_{i}, \,\forall \,  n-r+1\leq i\leq n-1,\\
       & x_r\cdots x_1\delta^qzy=x_{r-1}\cdots x_1\delta^qzyx_n,
\end{aligned}
\end{equation}
where $\delta$ denotes $yx_n\cdots x_1$. Relabelling $x_i$ by $t_{n-i+1}$ for all $i=1,\dots,n$ in Presentation \eqref{OpClassicalPres2} yields the presentation 
\begin{subequations}\label{OpClassicalPres3}
\begin{align}
    &(1) \,\, \mathrm{Generators}\!:\,  \{t_1,\dots,t_n,y,z\};\notag\\
&(2) \,\, \mathrm{Relations}\!: \,\notag\\
  &zyt_1\cdots t_n=yt_1\cdots t_nz, \label{OpClassicalPres3:1}  \\
  &t_{n-i-r+1}\cdots t_{n}\delta^{q-1}zyx_1\cdots t_{n-i}=t_{n-i-r+2}\cdots t_{n}\delta^{q-1}zyt_1\cdots t_{n-i+1}, \,(*),   \label{OpClassicalPres3:2} \\
       & t_{2n-i-r+1}\cdots t_{n}\delta^{q}zyt_1\cdots t_{n-i}=t_{2n-i-r+2}\cdots t_{n}\delta^{q}zyt_1\cdots t_{n-i+1}, \, (**),\label{OpClassicalPres3:3}\\
       &t_{n-r+1}\cdots t_n\delta^qzy=t_{n-r+2}\cdots t_n\delta^qzyx_1,\label{OpClassicalPres3:4}\\
       &(*) \, \forall \, 1\leq i\leq n-r,\,\, (**)\, \forall \,n-r+1\leq i\leq n-1\notag
\end{align}
\end{subequations}
where $\delta$ denotes $yt_1\dots t_n$. Now, the set of relations \eqref{OpClassicalPres3:2} is equivalent to the set of relations
\begin{equation}\label{PresentationRange1}t_{n-i-r+1}\cdots t_{n}\delta^{q-1}zyt_1\cdots t_{n-i}=t_{n-j-r+1}\cdots t_{n}\delta^{q-1}zyt_1\cdots t_{n-j},\, \forall\, 0\leq i<j\leq n-r.
\end{equation}
Writing $a_k=t_{n-k-r+1}\cdots t_n \delta^{q-1}zyt_1\cdots t_{n-k}$ for $k=0,\dots,n-r$, the set of relations given by Equation \eqref{PresentationRange1} is equivalent in saying that all $a_k$'s represent the same element in $\M_*^*(n,m)^{\mathrm{op}}$.\\
We have 
\begin{equation*}
\begin{aligned}
a_0&=t_{n-r+1}\cdots t_n\delta^{q-1}zyt_1\cdots t_n,\, 
a_1=t_{n-r}\cdots t_n \delta^{q-1}zyt_1\cdots t_{n-1},\\
&\cdots,\,
a_{n-r}=t_1\cdots t_n \delta^{q-1}zyt_1\cdots t_{r},
\end{aligned}
\end{equation*}
so that $\{a_k\}_{k=0,\dots,n-r}=\{t_l\cdots t_n\delta^{q-1}zyt_1\cdots t_{l+r-1}\}_{l=1,\dots,n-r+1}$. Thus, the set of relations given by Equation \eqref{PresentationRange1} is the set of relations
$$t_i\cdots t_n\delta^{q-1}zyt_1\cdots t_{i+r-1}=t_j\cdots t_n\delta^{q-1}zyt_1\cdots t_{j+r-1},\,  \forall\, 1\leq i<j\leq n-r+1.$$
Similarly, the set of relations \eqref{OpClassicalPres3:3} is equivalent to the set of relations
\begin{equation}\label{PresentationRange2}t_{2n-i-r+1}\cdots t_n\delta^qzyt_1\cdots t_{n-i}=t_{2n-j-r+1}\cdots t_n\delta^qzyt_1\cdots t_{n-j} , \, \forall\, n-r\leq i<j\leq n-1.
\end{equation}
Writing $b_k=t_{2n-k-r+1}\cdots t_n \delta^{q}zyt_1\cdots t_{n-k}$ for $k=n-r,\dots,n-1$, the set of relations given by Equation \eqref{PresentationRange2} is equivalent in saying that all $b_k$'s represent the same element in $\M_*^*(n,m)^{\mathrm{op}}$.\\
We have 
\begin{equation*}
\begin{aligned}
b_{n-r}&=\delta^{q}zyt_1\cdots t_r,\,
b_{n-r+1}=t_n \delta^{q}zyt_1\cdots t_{r-1},\\
&\cdots,\, 
b_{n-1}=t_{n-r+2}\cdots t_n \delta^{q}zyt_1,
\end{aligned}
\end{equation*}
so that $\{b_k\}_{k=n-r,\dots,n-1}=\{t_l\cdots t_n\delta^{q}zyt_1\cdots t_{l+r-n-1}\}_{l=n-r+2,\dots,n+1}$. Thus, the set of relations given by Equation \eqref{PresentationRange2} is the set of relations

\begin{equation}\label{PresentationRange3}
  t_i\cdots t_n\delta^qzyt_1\cdots t_{i+r-n-1}=t_j\cdots t_n\delta^qzyt_1\cdots t_{j+r-n-1} ,\, \forall \, n-r+2\leq i<j\leq n+1.  
\end{equation}
Finally, combining the relation $$t_{n-r+1}\cdots t_n\delta^qzy=t_{n-r+2}\cdots t_n\delta^qzyx_1$$
with the set of relations \eqref{PresentationRange3} gives the following set of relations: 
\begin{equation*}
  t_i\cdots t_n\delta^qzyt_1\cdots t_{i+r-n-1}=t_j\cdots t_n\delta^qzyt_1\cdots t_{j+r-n-1} ,\, \forall \, n-r+1\leq i<j\leq n+1.  
\end{equation*}
We thus get that $\mathcal M_*^*(n,m)^{\mathrm{op}}$ admits the following presentation:
\begin{equation*}
\begin{aligned}
    &(1) \,\, \mathrm{Generators}\!:\,  \{t_1,\dots,t_n,y,z\};\\
&(2) \,\, \mathrm{Relations}\!: \,\\
  &zyt_1\cdots t_n=yt_1\cdots t_nz,   \\
  &t_i\cdots t_n\delta^{q-1}zyt_1\cdots t_{i+r-1}=t_j\cdots t_n\delta^{q-1}zyt_1\cdots t_{j+r-1},\,  \forall\, 1\leq i<j\leq n-r+1,  \\
       &  t_i\cdots t_n\delta^qzyt_1\cdots t_{i+r-n-1}=t_j\cdots t_n\delta^qzyt_1\cdots t_{j+r-n-1},\, \forall \, n-r+1\leq i<j\leq n+1,
\end{aligned}
\end{equation*}
where $\delta$ denotes $yt_1\cdots t_n$. Using that $\delta z=z\delta$, for all $k\in [n-r+1]$ we have $$t_k\cdots t_n\delta^{q-1}zyt_1\cdots t_{k+r-1}=t_k\cdots t_nz\delta^{q-1}yt_1\cdots t_{k+r-1}$$
and for all $k\in\llbracket n-r+1,n+1\rrbracket $ we have
$$t_k\cdots t_n\delta^qzyt_1\cdots t_{k+r-n-1}=t_k\cdots t_n\delta z\delta^{q-1}yt_1\cdots t_{k+r-n-1}.$$
This concludes the proof, relabeling $t_k$ by $x_k$ for each $k\in [n]$.
\end{proof} 
Now, we will enlarge Presentation \eqref{OpClassicalPres1} to get a right-full homogeneous presentation satisfying $(\hyperref[Poulet]{C1})$.\\
In addition to the relations of Presentation \eqref{OpClassicalPres1}, observe the following:\\
\noindent
$\bullet$ For all $i\in [n-r+1], j\in \llbracket n-r+1,n+1\rrbracket$, we have
\begin{equation}\label{EnlargeClassicalOp1}
\begin{aligned}
x_i\cdots x_nz\delta^{q-1}yx_1\cdots x_{i+r-1}y&\underset{\eqref{OpClassicalPres1:relation2}}{=}x_{n-r+1}\cdots x_nz\delta^{q-1}yx_1\cdots x_ny  \, \\
\underset{\eqref{OpClassicalPres1:relation1}}{=}x_{n-r+1}\cdots x_n\delta z&\delta^{q-1}y
\underset{\eqref{OpClassicalPres1:relation3}}{=}x_j\cdots x_n\delta z\delta^{q-1}yx_1\cdots x_{j+r-n-1}.
\end{aligned}
\end{equation}
Note that for $j=n+1$, Equation \eqref{EnlargeClassicalOp1} reads 
\begin{equation*}
    x_i\cdots x_n \delta z\delta^{q-1}yx_1\cdots x_{i+r-n-1}y=yx_1\cdots x_nz\delta^{q-1}yx_1\cdots x_r.
\end{equation*}

\noindent
$\bullet$ For all $i\in [n-r+1]$ we have
\begin{equation}\label{EnlargeClassicalOp2}
\begin{aligned}
x_i\cdots x_nz\delta^{q-1}yx_1\cdots x_{i+r-1}y&\underset{\eqref{OpClassicalPres1:relation1}+\eqref{OpClassicalPres1:relation2}}{=}x_{n-r+1}\cdots x_n\delta z\delta^{q-1}y\\
&\underset{\eqref{OpClassicalPres1:relation3}}{=}\delta z\delta^{q-1}yx_1\cdots x_r\underset{\eqref{OpClassicalPres1:relation1}}{=}z\delta^{q}yx_1\cdots x_{r}.
\end{aligned}
\end{equation}

\noindent
$\bullet$ For all $i\in \llbracket n-r+2,n\rrbracket$ we have
\begin{equation}\label{EnlargeClassicalOp3}
\begin{aligned}
x_i\cdots x_n\delta z\delta^{q-1}yx_1\cdots x_{i+r-n-1}&\underset{\eqref{OpClassicalPres1:relation3}}{=}\delta z\delta^{q-1}yx_1\cdots x_r\, \underset{\eqref{OpClassicalPres1:relation1}}{=}z\delta^qyx_1\cdots x_r.
\end{aligned}
\end{equation}
Finally, since $\delta z=z\delta$, we have 
\begin{equation*}
    x_{n-r+1}\cdots x_n\delta z\delta^{q-1}y=(x_{n-r+1}\cdots x_nz\delta^{q-1}yx_1\cdots x_n)y.
\end{equation*}
Thus, the set of relations given by \eqref{OpClassicalPres1:relation3} with $i=n-r+1$, $j=n-r+2,\dots,n+1$ is the same set of relations as the one given by \eqref{EnlargeClassicalOp1} with $i=n-r+1$, $j=n-r+2,\dots,n+1$. Therefore, if we add relations \eqref{EnlargeClassicalOp1}, we can replace \eqref{OpClassicalPres1:relation3} by the set of relations
\begin{equation}\label{EnlargeClassicalOp4}
\begin{aligned}
x_i\cdots x_n\delta z\delta^{q-1}yx_1\cdots x_{i+r-n-1}=x_j\cdots x_n\delta z&\delta^{q-1}yx_1\cdots x_{j+r-n-1},\\&\forall n-r+2\leq i<j\leq n+1.
 \end{aligned}
\end{equation}

\noindent
Using equations \eqref{EnlargeClassicalOp1}-\eqref{EnlargeClassicalOp4}, we get the following corollary:
\begin{corollary}\label{OpClassical2}
Let $n,m\in \N^*$ and assume that $n\leq m$ and $n\wedge m=1$. The monoid $\mathcal M_*^*(n,m)^{\mathrm{op}}$ admits the following right-full presentation:

\begin{subequations}\label{OpClassicalPres4}
\begin{align}
    &(1) \,\, \mathrm{Generators}\!:\,  \{x_1,\dots,x_n,y,z\};\notag\\
&(2) \,\, \mathrm{Relations}\!: \,\notag\\
  &zyx_1\cdots x_n=yx_1\cdots x_nz,  \label{OpClassicalPres4:1} \\
  &x_i\cdots x_nz\delta^{q-1}yx_1\cdots x_{i+r-1}=x_j\cdots x_nz\delta^{q-1}yx_1\cdots x_{j+r-1},\, \forall \,1\leq i<j\leq n-r+1, \label{OpClassicalPres4:2}   \\
       &x_i\cdots x_n\delta z\delta^{q-1}yx_1\cdots x_{i+r-n-1}=x_j\cdots x_n\delta z\delta^{q-1}yx_1\cdots x_{j+r-n-1}, \,(*),\label{OpClassicalPres4:3} \\
&  x_i\cdots x_nz\delta^{q-1}yx_1\cdots x_{i+r-1}y=x_j\cdots x_n\delta z\delta^{q-1}yx_1\cdots x_{j+r-n-1},\, (**),\label{OpClassicalPres4:4}  \\
&x_i\cdots x_nz\delta^{q-1}yx_1\cdots x_{i+r-1}y=z\delta^qyx_1\cdots x_r, \,\forall \,1\leq  i\leq n-r+1,\label{OpClassicalPres4:5} \\
&  x_{i}\cdots x_{n}\delta z\delta^{q-1}yx_1\cdots x_{i+r-n-1}=z\delta^qyx_1\cdots x_r, \,\forall \, n-r+2\leq i\leq n,\label{OpClassicalPres4:6} \\
&(*)\,\forall \,  n-r+2\leq i<j\leq n+1,\, (**)\,\forall \, i\in[n-r+1],\, j\in \llbracket n-r+2,n+1\rrbracket\notag
\end{align}
\end{subequations}
where $\delta$ denotes $yx_1\cdots x_n$.
\end{corollary}
\begin{remark}
In Presentation \eqref{OpClassicalPres4}, it is possible that some conditions on the indices are empty. In this case, there are simply no corresponding relations. For instance, if $n=m=1$, the set of relations \eqref{OpClassicalPres4:6} is empty. 
\end{remark}
\begin{lemma}\label{cubeOpClassical}
Presentation \eqref{OpClassicalPres4} satisfies $(\hyperref[Poulet]{C1})$.
\end{lemma}

\begin{proof}
The proof is done in a similar way as the proof of Lemma \ref{CubeClassical}. The complete details of this proof will appear in the author's thesis \cite{These}.
\end{proof}

\begin{proof}[Proof of Proposition \ref{OpLeftCanc}]
It follows by combining Lemma \ref{cubeOpClassical} with Lemma \ref{MagicCube}. 
\end{proof}

\subsection{\texorpdfstring{The monoid $\M_*^*(n,m)$ is Garside}{}}
Recall from Definition \ref{defDelta} that given two coprime integers $n\leq m$, writing $m=qn+r$ with $q\geq 1$ and $0\leq r\leq n-1$ we defined $w$ to be the element of $\mathcal M_*^*(n,m)$ represented by each of the words in the following subset of $\{x_1,\dots,x_n,y,z\}^*$:
\begin{equation}\label{Divisorsw}\{z(x_1\cdots x_ny)^qx_1\cdots x_r\}\cup\{x_i\cdots x_nyz(x_1\cdots x_ny)^{q-1}x_1\cdots x_{i+r-1}\}_{i=1,\dots,n-r+1}.
\end{equation}
Similarly, we defined $W$ to be the element of $\mathcal M_*^*(n,m)$ represented by each of the words in the set 
\begin{equation}\label{DivisorsW}\{x_i\cdots x_nyz(x_1\cdots x_ny)^qx_1\cdots x_{i+r-n-1}\}_{i=n-r+1,\dots,n+1}.
\end{equation}
With this definition, we defined the element $\Delta=w^{n-r}W^r\in  \mathcal M_*^*(n,m)$, and showed in Proposition \ref{DeltaCentral} and Remark \ref{PositiveDelta} that $\Delta$ is central in $\mathcal M_*^*(n,m)$. We now show that it is a Garside element:

\begin{lemma}\label{DeltaGarsideM}
The element $\Delta$ as defined in Remark \ref{PositiveDelta} is a Garside element of $\M_*^*(n,m)$.
\end{lemma}

\begin{proof}
\underline{The sets $\text{Div}_L(\Delta)$ and $\text{Div}_R(\Delta)$ are equal:} By Remark \ref{PositiveDelta}, the $\Delta$ is central. Now, Lemma \ref{ClassicalLeftCanc} implies that $\M_*^*(n,m)$ is left-cancellative, hence $\text{Div}_L(\Delta)=\text{Div}_R(\Delta)$ by Lemma \ref{CenterDiv}.\\
\underline{The set $\text{Div}(\Delta)$ is finite:}
The divisibility of $\M_*^*(n,m)$ is Noetherian since Presentation \eqref{classicalBraidPres1} is homogeneous. Since $\M_*^*(n,m)$ is finitely generated, it follows that every element of $\M_*^*(n,m)$ (in particular, $\Delta$) admits finitely many left- and right-divisors.\\
\underline{The set $\text{Div}(\Delta)$ generates $\M_*^*(n,m)$:} 
By Definition of $w$, $W$ and $\Delta$, we have $\Delta=w^{n-r}W^r=W^rw^{n-r}$.\\
Now, recall with \eqref{Divisorsw} that $\{x_1,\dots,x_{n-r+1},z\}\subset \text{Div}_L(w)\subset \text{Div}(\Delta)$ and with \eqref{DivisorsW} that $\{x_{n-r+2},\dots,x_n,y\}\subset\text{Div}_L(W)\subset\text{Div}(\Delta),$ thus the divisors of $\Delta$ generate $\M_*^*(n,m)$, which concludes the proof.
\end{proof}
\noindent
We can now prove the first statement of Theorem \ref{ClassicalGarside}:
\begin{proof}[Proof of Theorem \ref{ClassicalGarside} (i)]
First, divisibility of $\mathcal M_*^*(n,m)$ is Noetherian since it admits a homogeneous presentation. Moreover, combining Propositions \ref{ClassicalLeftCanc} and \ref{OpLeftCanc} we get that $\mathcal M_*^*(n,m)$ is cancellative and admits conditional left- and right-lcms. Now Lemma \ref{DeltaGarsideM} exhibits a Garside element, which concludes the proof by combining Lemmas \ref{Deltamultiples} and \ref{LcmGcd}.
\end{proof}

\subsection{\texorpdfstring{The monoid $\M_*(n,m)$ is Garside}{}}
The goal of this subsection is to prove Theorem \ref{ClassicalGarside} (ii).

Recall from Definition \ref{ClassicalMonoid} that given two coprime integers $n\leq m$ with $m\geq 2$, writing $m=qn+r$ with $q\geq 1$ and $0\leq r\leq n-1$, the \textbf{classical monoid} associated to $\mathcal B_*(n,m)$ is the monoid $\mathcal M_*(n,m)$ which admits the following presentation:
\begin{subequations}\label{ClassicalMonoidPresy1}
\begin{align}
    &(1) \,\, \mathrm{Generators}\!:\,  \{x_1,\dots,x_n,y\};\notag\\
&(2) \,\, \mathrm{Relations}\!: \,\notag\\
  &x_{i+1}\cdots x_{n}y\delta^{q-1}x_1\cdots x_{i+r}=x_i\cdots x_{n}y\delta^{q-1}x_1\cdots x_{i+r-1}, \,\forall \, 1\leq i\leq n-r \label{ClassicalMonoidPresy1:1}  \\
  &  x_{i+1}\cdots x_ny\delta^qx_1\cdots x_{i+r-n}=x_i\cdots x_ny\delta^qx_1\cdots x_{i+r-n-1},\, \forall n-r+1\leq i \leq n,\label{ClassicalMonoidPresy1:2}
\end{align}
\end{subequations}
where $\delta$ denotes $x_1\cdots x_ny$.
\begin{example}
    For $n=2,m=3$, Presentation \eqref{ClassicalMonoidPresy1} reads
\begin{equation*}\left\langle\begin{array}{l|cl}
   		x_1,x_2\,	 & \, x_1x_2yx_1=x_2yx_1x_2   \\
  \,\,\,\,\,y\,  &\, x_2yx_1x_2y=yx_1x_2yx_1
				
                          \end{array}\right\rangle,\end{equation*} 
thus we retrieve the Garside monoid associated to the BMR presentation of $G_{13}$ (see \cite[Theorem 2.27 and Tables 1-3]{BMR} for the presentation and \cite[Example 13]{ThesePicantin} for the fact that the monoid is Garside).
\end{example}

\begin{remark}\label{remarkproof2}
Remark \ref{CommutativeBraidQuotients} can be seen as a remark about presentations and the mentioned quotients are also valid at the level of monoids. Therefore, the monoids $\mathcal M_*(n,m)$ and $\mathcal M_*^*(n,m)/\langle z=1\rangle$ are isomorphic. Thus, in order to prove Theorem \ref{ClassicalGarside} (ii), it is enough to show that the pair $(\eqref{ClassicalMonoidPres2},z)$ satisfies $(\hyperref[Champagne]{C2})$ and that $\mathcal M_*(n,m)$ is isomorphic to its opposite monoid to apply Proposition \ref{MagicGarsideOp}.
\end{remark}

\begin{lemma}\label{MagicPropertyy}
Writing $P$ for Presentation \eqref{ClassicalMonoidPres2}, the pair $(P,z)$ satisfies $(\hyperref[Champagne]{C2})$.
\end{lemma}

\begin{proof}
Write $P=\langle S\,|\, R\rangle$.\\
\underline{The presentation $P$ is homogeneous and satisfies $(\hyperref[Poulet]{C1})$:} This is the content of Lemma \ref{CubeClassical}.\\
\underline{For all $(r_1,r_2)\in R$, we have $|r_1|_z=|r_2|_z$:} For all $(r_1,r_2)\in R$, we have $|r_1|_z=1=|r_2|_z$.\\
\underline{The monoids with presentations $\langle S\,|\,R\rangle_z$ and $\langle S\, |\, R\rangle/z$ are isomorphic:} In this context, the monoid with presentation $\langle S\, | \, R\rangle/z$ is isomorphic to $\mathcal M_*^*(n,m)/\langle z=1\rangle=\mathcal M_*(n,m)$ and $\langle S\, | \, R\rangle_z$ is 
\begin{subequations}\label{Classical/z}
\begin{align}
    &(1) \,\, \mathrm{Generators}\!:\,  \{x_1,\dots,x_n,y\};\notag\\
&(2) \,\, \mathrm{Relations}\!: \,\notag\\
  & x_{i}\cdots x_{n}y\delta^{q-1}x_1\cdots x_{i+r-1}=x_j\cdots x_{n}y\delta^{q-1}x_1\cdots x_{j+r-1}, \, \forall \,1\leq i<j\leq n-r+1, \label{Classical/z:1}  \\
  &x_{i}\cdots x_{n}y\delta^{q}x_1\cdots x_{i+r-n-1}=x_j\cdots x_{n}y\delta^{q}x_1\cdots x_{j+r-n-1}, \,(*),\label{Classical/z:2}\\
  & x_i\cdots x_ny\delta^{q-1}x_1\cdots x_{i+r-1}y=x_j\cdots x_{n}y\delta^{q}x_1\cdots x_{j+r-n-1},\, (**),\label{Classical/z:3}\\
  &(*)\, \forall \, n-r+2\leq i<j\leq n+1\, (**)\,\forall \, i\in[n-r+1],\, j\in \llbracket n-r+2,n+1\rrbracket\notag
\end{align}
\end{subequations}
where $\delta$ denotes $x_1\cdots x_ny$. 

First, the set of relations \eqref{ClassicalMonoidPresy1:1} and \eqref{Classical/z:1} are equivalent. Moreover, the subset of relations of \eqref{Classical/z:3} obtained by setting $i=n-r+1$ $$x_{n-r+1}\cdots x_ny\delta^{q-1}\underbrace{x_1\cdots x_{n}y}_{=\delta}=x_j\cdots x_ny\delta^qx_1\cdots x_{j+r-n-1},\,\forall \,  j\in \llbracket n-r+2,n+1\rrbracket$$
combined with the set of relations \eqref{Classical/z:2} is equivalent to the set of relations
\begin{equation}\label{/z1}
    x_i\cdots x_ny\delta^qx_1\cdots x_{i+r-n-1}=x_j\cdots x_ny\delta^qx_1\cdots x_{j+r-n-1},\,\forall \,  n-r+1\leq i<j\leq n+1.
\end{equation}
Finally, we show that the set of relations \eqref{Classical/z:3} follows from \eqref{Classical/z:1} and \eqref{/z1}:\\
For all $k\in [n-r+1]$, $l\in \llbracket n-r+2,n+1\rrbracket$, we have
\begin{equation*}
\begin{aligned}
x_k\cdots x_ny\delta^{q-1}x_1\cdots x_{k+r-1}y&\underset{\eqref{Classical/z:1}}{=}x_{n-r+1}\cdots x_ny\delta^{q-1}\underbrace{x_1\cdots x_ny}_{=\delta}\\& \underset{\eqref{/z1}}{=}x_l\cdots x_ny\delta^qx_1\cdots x_{l+r-n-1}.
\end{aligned}
\end{equation*}
In particular, an alternative presentation for the monoid with presentation $\langle S\, |R\, \rangle_z$ is given by
\begin{subequations}\label{Classical/z2}
\begin{align}
    &(1) \,\, \mathrm{Generators}\!:\,  \{x_1,\dots,x_n,y\};\notag\\
&(2) \,\, \mathrm{Relations}\!: \,\notag\\
  & x_{i}\cdots x_{n}y\delta^{q-1}x_1\cdots x_{i+r-1}=x_j\cdots x_{n}y\delta^{q-1}x_1\cdots x_{j+r-1}, \,\forall \,  1\leq i<j\leq n-r+1, \label{Classical/z2:1}\\
  & x_{i}\cdots x_{n}y\delta^{q}x_1\cdots x_{i+r-n-1}=x_j\cdots x_{n}y\delta^{q}x_1\cdots x_{j+r-n-1}, \,(*),\label{Classical/z2:2}\\
  &(*)\,\forall \,  n-r+1\leq i<j\leq n+1\notag
\end{align}
\end{subequations}
where $\delta$ denotes $x_1\cdots x_ny$. Since the set of relations
\eqref{ClassicalMonoidPresy1:1} (respectively \eqref{ClassicalMonoidPresy1:2}) and \eqref{Classical/z2:1} (respectively \eqref{Classical/z2:2}) are equivalent,  Presentation \eqref{Classical/z2} is a presentation for both $\langle S\,|\, R\rangle_z$ and $\mathcal M_*(n,m)=\langle S\, | \, R\rangle/z$. This concludes the proof. 
\end{proof}

\begin{lemma}\label{MagicPropertyyOp}
Let $n\in \N^*$, $m\in \N_{\geq 2}$ and assume that $n\leq m$ and $n\wedge m=1$. The monoid $\mathcal M_*(n,m)$ is isomorphic to $\mathcal M_*(n,m)^{\mathrm{op}}$.
\end{lemma}

\begin{proof}
To see this, observe that $\mathcal M_*(n,m)^{\mathrm{op}}$ is isomorphic to $(\mathcal M_*^*(n,m)/\langle z=1\rangle)^{\mathrm{op}}$, which in turn is isomorphic to $\mathcal M_*^*(n,m)^{\mathrm{op}}/\langle z=1\rangle$. Thus, by Lemma \ref{OpClassical} a presentation for $\mathcal M_*(n,m)^{\mathrm{op}}$ is 
\begin{equation}\label{Classical/z3}
\begin{aligned}
    &(1) \,\, \mathrm{Generators}\!:\,  \{x_1,\dots,x_n,y\};\\
&(2) \,\, \mathrm{Relations}\!: \,\\
  & x_i\cdots x_n\delta^{q-1}yx_1\cdots x_{i+r-1}=x_j\cdots x_n\delta^{q-1}yx_1\cdots x_{j+r-1},\,\forall \,  1\leq i<j\leq n-r+1,   \\
  &  x_i\cdots x_n\delta^{q}yx_1\cdots x_{i+r-n-1}=x_j\cdots x_n\delta^{q}yx_1\cdots x_{j+r-n-1},(*),\\
  &(*)\,\forall \,  n-r+1\leq i<j\leq n+1
\end{aligned}
\end{equation}
where $\delta$ is $yx_1\dots x_n$. For all $\alpha\in \N$, we have $\delta^{\alpha}y=(yx_1\dots x_n)^{\alpha}y=y(x_1\cdots x_ny)^{\alpha}$ so that Presentation \eqref{Classical/z3} corresponds to Presentation \eqref{BraidPresy1} of $\mathcal M_*(n,m)$, which concludes the proof.
\end{proof}

\begin{proof}[Proof of Theorem \ref{ClassicalGarside} (ii)]
By Remark \ref{remarkproof2}, the result follows by combining Lemmas \ref{MagicPropertyy} and \ref{MagicPropertyyOp} with Proposition \ref{MagicGarsideOp}.
\end{proof}

%%%%%%%%%%%%%%%%%%%%%%%%%%%%%%%%%%%%%%%%%%%%%%%%%%%%%%%%%%%%%%%%%%%%%%%%%%%%%%%%%%%%%%%%%%%%%%%%%%%%%%%
%%%%%%%%%%%%%%%%%%%%%%%%%%%%%%%%%%%%%%%%%%%%%%%%%%%%%%%%%%%%%%%%%%%%%%%%%%%%%%%%%%%%%%%%%%%%%%%%%%%%%%%
%%%%%%%%%%%%%%%%%%%%%%%%%%%%%%%%%%%%%%%%%%%%%%%%%%%%%%%%%%%%%%%%%%%%%%%%%%%%%%%%%%%%%%%%%%%%%%%%%%%%%%%
%%%%%%%%%%%%%%%%%%%%%%%%%%%%%%%%%%%%%%%%%%%%%%%%%%%%%%%%%%%%%%%%%%%%%%%%%%%%%%%%%%%%%%%%%%%%%%%%%%%%%%%

\subsection{\texorpdfstring{The monoid $\M^*(n,m)$ is Garside}{}}
The goal of this subsection is prove Theorem \ref{ClassicalGarside} (iii).

Recall from Definition \ref{ClassicalMonoid} that given two coprime integers $2\leq n\leq m$, writing $m=qn+r$ with $q\geq 1$ and $0\leq r\leq n-1$, the \textbf{classical monoid} for $\B^*(n,m)$ is the monoid $\mathcal M^*(n,m)$ which admits the following presentation:
\begin{subequations}\label{ClassicalMonoidPresz1}
\begin{align}
    &(1) \,\, \mathrm{Generators}\!:\,  \{x_1,\dots,x_n,z\};\notag\\
&(2) \,\, \mathrm{Relations}\!: \,\notag\\
  &zx_1\cdots x_n=x_1\cdots x_nz,\label{ClassicalMonoidPresz1:1}\\
  & x_{i+1}\cdots x_{n}z\delta^{q-1}x_1\cdots x_{i+r}=x_i\cdots x_{n}z\delta^{q-1}x_1\cdots x_{i+r-1}, \, \forall \, 1\leq i\leq n-r , \label{ClassicalMonoidPresz1:2} \\
  & x_{i+1}\cdots x_{n}z\delta^{q}x_1\cdots x_{i+r-n}=x_i\cdots x_{n}z\delta^{q}x_1\cdots x_{i+r-n-1}, \,\forall \,  n-r+1\leq i\leq n,\label{ClassicalMonoidPresz1:3}
\end{align}
\end{subequations}
where $\delta$ denotes $x_1\cdots x_n$.
\begin{example}
For $n=2,m=3$, Presentation \eqref{ClassicalMonoidPresz1} reads
\begin{equation*}\left\langle\begin{array}{l|cl}
   		x_1,x_2\,	 & \, zx_1x_2=x_1x_2z   \\
  \,\,\,\,\,z\,  &\, x_1x_2zx_1=x_2zx_1x_2=zx_1x_2x_1
                          \end{array}\right\rangle,\end{equation*} 
thus we retrieve the Garside monoid associated to the BMR presentation of $G(3c,3,2)$ (see \cite[Theorem 2.27 and Tables 1-3]{BMR} for the presentation and \cite[Example 3.25]{GrosBouquinBleu} for the fact that the monoid is Garside).
\end{example}
\noindent

\begin{remark}\label{Remarkproof3}
By remark \ref{CommutativeBraidQuotients} the monoids $\mathcal M^*(n,m)$ and $\mathcal M_*^*(n,m)/\langle y=1\rangle$ are isomorphic. Thus, in order to prove Theorem \ref{ClassicalGarside} (ii), it is enough to show that the pair $(\eqref{ClassicalMonoidPres2},y)$ satisfies $(\hyperref[Champagne]{C2})$ and that $\mathcal M^*(n,m)$ is isomorphic to its opposite monoid to apply Proposition \ref{MagicGarsideOp}.
\end{remark}

\begin{lemma}\label{MagicPropertyz}
Writing $P$ for Presentation \eqref{ClassicalMonoidPres2}, the pair $(P,y)$ satisfies $(\hyperref[Champagne]{C2})$.
\end{lemma}

\begin{proof}
Write $P=\langle  S\, |\, R\rangle$.\\
\underline{The presentation $P$ is homogeneous and satisfies $(\hyperref[Poulet]{C1})$:} This is the content of Lemma \ref{CubeClassical}.\\
\underline{For all $(r_1,r_2)\in R$, we have $|r_1|_y=|r_2|_y$:} One easily verifies that the equality holds true for the lines of relations \eqref{ClassicalMonoidPres2:1}-\eqref{ClassicalMonoidPres2:6}.\\
\underline{The monoids with presentations $\langle S\,|\,R\rangle_y$ and $\langle S\, |\, R\rangle/y$ are isomorphic:} In this context, the monoid with presentation $\langle S\, | \, R\rangle/y$ is isomorphic to $\mathcal M_*^*(n,m)/\langle y=1\rangle=\mathcal M^*(n,m)$ and $\langle S\, | \, R\rangle_y$ is 
\begin{subequations}\label{Classical/y}
\begin{align}
    &(1) \,\, \mathrm{Generators}\!:\,  \{x_1,\dots,x_n,z\};\notag\\
&(2) \,\, \mathrm{Relations}\!: \,\notag\\
  &zx_1\cdots x_n=x_1\cdots x_nz,\label{Classical/y:1}\\
  &  x_{i}\cdots x_{n}z\delta^{q-1}x_1\cdots x_{i+r-1}=x_j\cdots x_{n}z\delta^{q-1}x_1\cdots x_{j+r-1}, \,  \forall \, 1\leq i<j\leq n-r+1  , \label{Classical/y:2}\\
  &x_{i}\cdots x_{n}z\delta^{q}x_1\cdots x_{i+r-n-1}=x_j\cdots x_{n}z\delta^{q}x_1\cdots x_{j+r-n-1}, \,\forall \,  n-r+2\leq i<j\leq n,\label{Classical/y:3}\\
  &x_i\cdots x_nz\delta^{q-1}x_1\cdots x_{i+r-1}=x_j\cdots x_{n}z\delta^{q}x_1\cdots x_{j+r-n-1},\, (*),\label{Classical/y:4}\\
  &x_{i}\cdots x_{n}z\delta^{q-1}x_1\cdots x_{i+r-1}=z\delta^qx_1\cdots x_r, \,\forall \,   2\leq i\leq n-r+1,\label{Classical/y:5}\\
  &x_{i}\cdots x_{n}z\delta^{q}x_1\cdots x_{i+r-n-1}=z\delta^qx_1\cdots x_r, \,\forall \, n-r+2\leq i\leq n,\label{Classical/y:6}\\
  & (*) \,\forall \,  i\in[n-r+1],\, j\in \llbracket n-r+2,n\rrbracket\notag
\end{align}
\end{subequations}
where $\delta$ denotes $x_1\cdots x_n$. First, the set of relations \eqref{ClassicalMonoidPresz1:1} and \eqref{Classical/y:1} are the same.
Moreover, the subset of relations of \eqref{Classical/y:4} obtained by setting $i=n-r+1$ 
$$x_{n-r+1}\cdots x_nz\delta^{q-1}x_1\cdots x_n=x_j\cdots x_nz\delta^qx_1\cdots x_{j+r-n-1}, \,\forall \,  j\in \llbracket n-r+2,n\rrbracket$$
combined with \eqref{Classical/y:3} and \eqref{Classical/y:6} is equivalent to the set of relations 
\begin{equation}\label{/y1}
x_i\cdots x_nz\delta^qx_1\cdots x_{i+r-n-1}=x_j\cdots x_nz\delta^qx_1\cdots x_{j+r-n-1}, \,\forall \,  n-r+1\leq i<j\leq n+1.
\end{equation}
We now show that the sets of relations \eqref{Classical/y:1} and \eqref{Classical/y:2} along with the set of relations \eqref{/y1} imply \eqref{Classical/y:4} and \eqref{Classical/y:5}:\\
\underline{\eqref{Classical/y:4}:} For all $k\in [n-r+1], l\in \llbracket n-r+2,n\rrbracket$, we have
\begin{equation*}
\begin{aligned}
x_k\cdots x_nz\delta^{q-1}x_1\cdots x_{k+r-1}&\underset{\eqref{Classical/y:2}}{=}x_{n-r+1}\cdots x_nz\delta^{q-1}x_1\cdots x_n \\&\underset{\eqref{/y1}}{=}x_l\cdots x_n\delta^qx_1\cdots x_{l+r-n-1}.
\end{aligned}
\end{equation*}
\underline{\eqref{Classical/y:5}:}
For all $k\in \llbracket 2,n-r+1\rrbracket$, we have
\begin{equation*}
    \begin{aligned}
        x_k\dots x_nz\delta^{q-1}x_1\cdots x_{k+r-1}&\underset{\eqref{Classical/y:2}}{=}x_1\cdots x_nz\delta^{q-1}x_1\cdots x_r
        \underset{\eqref{Classical/y:1}}{=}z\delta^qx_1\cdots x_r.
    \end{aligned}
\end{equation*}
In particular, an alternative presentation for the monoid with Presentation $\langle S\,|\, R\rangle_y$ is 
\begin{subequations}\label{Classical/y2}
\begin{align}
    &(1) \,\, \mathrm{Generators}\!:\,  \{x_1,\dots,x_n,z\};\notag\\
&(2) \,\, \mathrm{Relations}\!: \,\notag\\
  &zx_1\cdots x_n=x_1\cdots x_nz,\label{Classical/y2:1}\\
  & x_{i}\cdots x_{n}z\delta^{q-1}x_1\cdots x_{i+r-1}=x_j\cdots x_{n}z\delta^{q-1}x_1\cdots x_{j+r-1}, \,\forall \,  1\leq i<j\leq n-r+1 ,\label{Classical/y2:2}\\
  &x_{i}\cdots x_{n}z\delta^{q}x_1\cdots x_{i+r-n-1}=x_j\cdots x_{n}z\delta^{q}x_1\cdots x_{j+r-n-1}, \,(*),\label{Classical/y2:3}\\
  &(*)\, \forall \,  n-r+1\leq i<j\leq n+1\notag
\end{align}
\end{subequations}
where $\delta$ denotes $x_1\dots x_n$. Since the set of relations \eqref{ClassicalMonoidPresz1:1} (respectively \eqref{ClassicalMonoidPresz1:2}, \eqref{ClassicalMonoidPresz1:3}) and \eqref{Classical/y2:1} (respectively \eqref{Classical/y2:2}, \eqref{Classical/y2:3}) are equivalent, Presentation \eqref{Classical/y2} is a presentation for both $\langle S\, | \, R\rangle_y$ and $\mathcal M^*(n,m)=\langle S\, | \, R \rangle/y$. This concludes the proof.
\end{proof}

\begin{lemma}\label{MagicPropertyzOp}
Let $n,m\in \N_{\geq 2}$ and assume that $n\leq m$ and $n\wedge m=1$. The monoid $\mathcal M^*(n,m)$ is isomorphic to $\mathcal M^*(n,m)^{\mathrm{op}}$.
\end{lemma}

\begin{proof}
To see this, observe that $\mathcal M^*(n,m)^{\mathrm{op}}$ is isomorphic to $(\mathcal M_*^*(n,m)/\langle y=1\rangle)^{\mathrm{op}}$, which in turn is isomorphic to $\mathcal M_*^*(n,m)^{\mathrm{op}}/\langle y=1\rangle$. Thus, by Lemma \ref{OpClassical} a presentation for $\mathcal M^*(n,m)^{\mathrm{op}}$ is
\begin{equation}\label{Classical/y3}
\begin{aligned}
    &(1) \,\, \mathrm{Generators}\!:\,  \{x_1,\dots,x_n,z\};\\
&(2) \,\, \mathrm{Relations}\!: \,\\
  &zx_1\cdots x_n=x_1\cdots x_nz,\\
  & x_i\cdots x_nz\delta^{q-1}x_1\cdots x_{i+r-1}=x_j\cdots x_nz\delta^{q-1}x_1\cdots x_{j+r-1},\, \forall \, 1\leq i<j\leq n-r+1,   \\
  & x_i\cdots x_n\delta z\delta^{q-1}x_1\cdots x_{i+r-n-1}=x_j\cdots x_n\delta z\delta^{q-1}x_1\cdots x_{j+r-n-1} \, (*),\\
  &(*)\,\forall \,  n-r+1\leq i<j\leq n+1,
\end{aligned}
\end{equation}
where $\delta$ denotes $x_1\cdots x_n$. Since $z\delta=\delta z$, we have $\delta z \delta^{q-1}=z\delta^q$ so that Presentation \eqref{Classical/y3} corresponds to Presentation \eqref{BraidPresz1} of $\mathcal M^*(n,m)$, which concludes the proof.
\end{proof}

\begin{proof}[Proof of Theorem \ref{ClassicalGarside} (iii)]
By Remark \ref{Remarkproof3}, the result follows by combining Lemmas \ref{MagicPropertyz} and \ref{MagicPropertyzOp} with Proposition \ref{MagicGarsideOp}.
\end{proof}

%%%%%%%%%%%%%%%%%%%%%%%%%%%%%%%%%%%%%%%%%%%%%%%%%%%%%%%%%%%%%%%%%%%%%%%%%%%%%%%%%%%%%%%%%%%%%%%%%%%%%%%
%%%%%%%%%%%%%%%%%%%%%%%%%%%%%%%%%%%%%%%%%%%%%%%%%%%%%%%%%%%%%%%%%%%%%%%%%%%%%%%%%%%%%%%%%%%%%%%%%%%%%%%
%%%%%%%%%%%%%%%%%%%%%%%%%%%%%%%%%%%%%%%%%%%%%%%%%%%%%%%%%%%%%%%%%%%%%%%%%%%%%%%%%%%%%%%%%%%%%%%%%%%%%%%
%%%%%%%%%%%%%%%%%%%%%%%%%%%%%%%%%%%%%%%%%%%%%%%%%%%%%%%%%%%%%%%%%%%%%%%%%%%%%%%%%%%%%%%%%%%%%%%%%%%%%%%%%%%%%%%%%%%%%%%%%%%%%%%%%%%%%%%%%%%%%%%%%%%%%%%%%%%%%%%%%%%%%%%%%%%%%%%%%%%%%%%%%%%%%%%%%%%%%%%%%%%%%
%%%%%%%%%%%%%%%%%%%%%%%%%%%%%%%%%%%%%%%%%%%%%%%%%%%%%%%%%%%%%%%%%%%%%%%%%%%%%%%%%%%%%%%%%%%%%%%%%%%%%%%
%%%%%%%%%%%%%%%%%%%%%%%%%%%%%%%%%%%%%%%%%%%%%%%%%%%%%%%%%%%%%%%%%%%%%%%%%%%%%%%%%%%%%%%%%%%%%%%%%%%%%%%
%%%%%%%%%%%%%%%%%%%%%%%%%%%%%%%%%%%%%%%%%%%%%%%%%%%%%%%%%%%%%%%%%%%%%%%%%%%%%%%%%%%%%%%%%%%%%%%%%%%%%%%%%%%%%%%%%%%%%%%%%%%%%%%%%%%%%%%%%%%%%%%%%%%%%%%%%%%%%%%%%%%%%%%%%%%%%%%%%%%%%%%%%%%%%%%%%%%%%%%%%%%%%%
%%%%%%%%%%%%%%%%%%%%%%%%%%%%%%%%%%%%%%%%%%%%%%%%%%%%%%%%%%%%%%%%%%%%%%%%%%%%%%%%%%%%%%%%%%%%%%%%%%%%%%%
%%%%%%%%%%%%%%%%%%%%%%%%%%%%%%%%%%%%%%%%%%%%%%%%%%%%%%%%%%%%%%%%%%%%%%%%%%%%%%%%%%%%%%%%%%%%%%%%%%%%%%%
%%%%%%%%%%%%%%%%%%%%%%%%%%%%%%%%%%%%%%%%%%%%%%%%%%%%%%%%%%%%%%%%%%%%%%%%%%%%%%%%%%%%%%%%%%%%%%%%%%%%%%%

\section{\texorpdfstring{The Dual braid monoids are Garside}{}}
In this section, when we say that two sets of relations are equivalent, we mean that they are equivalent as \textit{monoid relations} instead of merely group relations.\\
The goal of this section is to prove the following theorem:
\begin{theorem}\label{GarsideDual}
Let $n,m\in\N^*$ and assume that $n\wedge m=1$ and $n\leq m$. Then the following hold:\\
(i) The monoid $\mathcal D_*^*(n,m)$ is Garside.\\
(ii) If $m\geq 2$, the monoid $\D_*(n,m)$ is Garside.\\
(iii) If $n\geq 2$, the monoid $\D^*(n,m)$ is Garside.
\end{theorem}

The proof strategy goes as follows: we show that $\mathcal D_*^*(n,m)$ is cancellative and admits conditional left- and right-lcms (Propositions \ref{LeftCancDual} and \ref{LeftCancDualOp}), then we exhibit a Garside element (Lemma \ref{GarsideDualLemma4}) and conclude the proof using results of Section \ref{SecGarside}.

\subsection{\texorpdfstring{The monoid $\D_*^*(n,m)$ is left-cancellative}{}}
\begin{proposition}\label{LeftCancDual}
Let $n,m\in\N^*$ with $n\wedge m=1$ and $n\leq m$. The monoid $\D_*^*(n,m)$ is left cancellative and admits conditional right-lcms.
\end{proposition}
We will separate the proof of Proposition \ref{LeftCancDual} in two cases: $n=1$ and $n\geq 2$. In order to be able to use results from Section \ref{SecGarside}, we enlarge the monoid presentation \eqref{dual3} to get a right-full homogeneous presentation of $\D_*^*(n,m)$.\\

\noindent
\fbox{The case $n\geq 2$:}\\
The fact that $n\geq 2$ gives us $m-n+1<m$ so that for all $i\in [m-n+1]$, the equality $x_iz_{i+1}=z_ix_i$ holds. Moreover, the set of relations \eqref{dual3:3} is equivalent to 
\begin{equation}\label{L2Dual}
z_ix_i\cdots x_{i+n-1}=z_jx_j\cdots x_{j+n-1}, \, \forall 1\leq i<j\leq m-n+1,
\end{equation}
and the set of relations \eqref{dual3:4} is equivalent to 
\begin{equation}\label{L3Dual}
z_ix_i\cdots x_myx_1\cdots x_{i+n-m-1}=z_jx_j\cdots x_myx_1\cdots x_{j+n-m-1},\, \forall m-n+1\leq i<j\leq m.
\end{equation}
In addition to \eqref{dual3:1},\eqref{dual3:2} \eqref{dual3:5}, \eqref{L2Dual} and \eqref{L3Dual}, observe the following:\\
\noindent
$\bullet$ For all $i,j\in [m-n+1]$ with $i\neq j$, we have
\begin{equation}\label{EnlargeDuaL1}
\begin{aligned}
x_iz_{i+1}x_{i+1}\cdots x_{i+n-1}&\underset{\eqref{dual3:1}}{=}z_ix_i\cdots x_{i+n-1}\underset{\eqref{L2Dual}}{=}z_{j}x_{j}\cdots x_{j+n-1}.\\
\end{aligned}
\end{equation}

\noindent
$\bullet$ For all $1\leq i <j\leq m-n+1$, we have
\begin{equation}\label{EnlargeDual2}
\begin{aligned}
x_iz_{i+1}x_{i+1}\cdots x_{i+n-1}\underset{\eqref{EnlargeDuaL1}}{=}z_{j}x_{j}\cdots x_{j+n-1}\underset{\eqref{dual3:1}}{=}x_{j}z_{j+1}x_{j+1}\cdots x_{j+n-1}.
\end{aligned}
\end{equation}

\noindent
$\bullet$ For all $i\in [m-n+1]$, $j\in \llbracket m-n+2,m\rrbracket$, we have
\begin{equation}\label{EnlargeDual3}
\begin{aligned}
x_iz_{i+1}x_{i+1}\cdots x_{i+n-1}y&\underset{\eqref{dual3:1}}{=}z_ix_i\cdots x_{i+n-1}y \\
\underset{\eqref{L2Dual}}{=}z_{m-n+1}x_{m-n+1}&\cdots x_my
\underset{\eqref{L3Dual}}{=}z_{j}x_{j}\cdots x_{m}yx_1\cdots x_{j+n-m-1}.
\end{aligned}
\end{equation}

\noindent
$\bullet$ For all $i\in [m-n+1]$, $j\in \llbracket m-n+2,m-1\rrbracket$, we have
\begin{equation}\label{EnlargeDual4}
\begin{aligned}
x_iz_{i+1}x_{i+1}\cdots x_{i+n-1}y&\underset{\eqref{dual3:1}}{=}z_ix_i\cdots x_{i+n-1}y \underset{\eqref{L2Dual}}{=}z_{m-n+1}x_{m-n+1}\cdots x_my\\
\underset{\eqref{L3Dual}}{=}z_{j}x_{j}\cdots x_{m}yx_1\cdots &x_{j+n-m-1}
\underset{\eqref{dual3:1}}{=}x_{j}z_{j+1}x_{j+1}\cdots x_{m}yx_1\cdots x_{j+n-m-1}.
\end{aligned}
\end{equation}

\noindent
$\bullet$ For all $i\in [m-n+1]$, we have
\begin{equation}\label{EnlargeDual5}
\begin{aligned}
x_iz_{i+1}x_{i+1}\cdots x_{i+n-1}y&\underset{\eqref{dual3:1}}{=}z_ix_i\cdots x_{i+n-1}y \underset{\eqref{L2Dual}}{=}z_{m-n+1}x_{m-n+1}\cdots x_my\, \,\\
&\underset{\eqref{L3Dual}}{=}z_{m}x_{m}yx_1\cdots x_{n-1} \underset{\eqref{dual3:2}}{=}x_myz_1x_1\cdots x_{n-1}.
\end{aligned}
\end{equation}

\noindent
$\bullet$ For all $i\in [m-n+1]$, we have
\begin{equation}\label{EnlargeDual6}
\begin{aligned}
x_iz_{i+1}x_{i+1}\cdots x_{i+n-1}y&\underset{\eqref{EnlargeDual5}}{=}z_{m}x_{m}yx_1\cdots x_{n-1} \underset{\eqref{dual3:5}}{=}yz_1x_1\cdots x_n.
\end{aligned}
\end{equation}

\noindent
$\bullet$ For all $i\in \llbracket m-n+2,m-1\rrbracket$, $j\in [m-n+1]$, we have
\begin{equation}\label{EnlargeDual7}
\begin{aligned}
x_iz_{i+1}x_{i+1}\cdots x_{m}yx_1\cdots x_{i+n-m-1}&\underset{\eqref{dual3:1}}{=}z_ix_i\cdots x_{m}yx_1\cdots x_{i+n-m-1} \\
\underset{\eqref{L3Dual}}{=}z_{m-n+1}x_{m-n+1}&\cdots x_{m}y 
\underset{\eqref{L2Dual}}{=}z_jx_j\cdots x_{j+n-1}y.
\end{aligned}
\end{equation}

\noindent
$\bullet$ For all $i\in \llbracket m-n+2,m-1\rrbracket$, $j\in \llbracket m-n+2,m\rrbracket$ with $i\neq j$, we have
\begin{equation}\label{EnlargeDual8}
\begin{aligned}
x_iz_{i+1}x_{i+1}\cdots x_{m}yx_1\cdots x_{i+n-m-1}&\underset{\eqref{dual3:1}}{=}z_ix_i\cdots x_{m}yx_1\cdots x_{i+n-m-1}\\
&\underset{\eqref{L3Dual}}{=}z_jx_j\cdots x_{m}yx_1\cdots x_{j+n-m-1}.
\end{aligned}
\end{equation}

\noindent
$\bullet$ For all $m-n+2\leq i<j\leq m-1$, we have
\begin{equation}\label{EnlargeDual9}
\begin{aligned}
x_iz_{i+1}x_{i+1}\cdots x_{m}yx_1\cdots x_{i+n-m-1}&\underset{\eqref{EnlargeDual8}}{=}z_jx_j\cdots x_{m}yx_1\cdots x_{j+n-m-1}\\
&\underset{\eqref{dual3:1}}{=}x_jz_{j+1}x_{j+1}\cdots x_myx_1\cdots x_{j+n-m-1}.
\end{aligned}
\end{equation}

\noindent
$\bullet$ For all $i\in \llbracket m-n+2,m-1\rrbracket$, we have
\begin{equation}\label{EnlargeDuaL10}
\begin{aligned}
x_iz_{i+1}x_{i+1}\cdots x_{m}yx_1\cdots x_{i+n-m-1}&\underset{\eqref{dual3:1}}{=}z_ix_i\cdots x_{m}yx_1\cdots x_{i+n-m-1} \\
&\underset{\eqref{L3Dual}}{=}z_mx_{m}yx_1\cdots x_{n-1} \underset{\eqref{dual3:2}}{=}x_myz_1x_1\cdots x_{n-1}.
\end{aligned}
\end{equation}

\noindent
$\bullet$ For all $i\in \llbracket m-n+2,m-1\rrbracket$, we have
\begin{equation}\label{EnlargeDuaL11}
\begin{aligned}
x_iz_{i+1}x_{i+1}\cdots x_{m}yx_1\cdots x_{i+n-m-1}&\underset{\eqref{EnlargeDuaL10}}{=}z_mx_{m}yx_1\cdots x_{n-1} \\
&\underset{\eqref{dual3:5}}{=}yz_1x_1\cdots x_{n}.
\end{aligned}
\end{equation}

\noindent
$\bullet$ For all $i\in [m-n+1]$, $j\in \llbracket m-n+2,m\rrbracket$, we have
\begin{equation}\label{EnlargeDuaL12}
\begin{aligned}
z_ix_i\cdots x_{i+n-1}y&\underset{\eqref{L2Dual}}{=}z_{m-n+1}x_{m-n+1}\cdots x_{m}y\underset{\eqref{L3Dual}}{=}z_{j}x_{j}\cdots x_{m}yx_1\cdots x_{j+n-m-1}.
\end{aligned}
\end{equation}

\noindent
$\bullet$ For all $i\in [m-n+1]$, we have
\begin{equation}\label{EnlargeDuaL13}
\begin{aligned}
z_ix_i\cdots x_{i+n-1}y&\underset{\eqref{EnlargeDuaL12}}{=}z_{m}x_{m}yx_1\cdots x_{n-1}
\underset{\eqref{dual3:2}}{=}x_myz_1x_1\cdots x_{n-1}.
\end{aligned}
\end{equation}

\noindent
$\bullet$ For all $i\in [m-n+1]$, we have
\begin{equation}\label{EnlargeDuaL14}
\begin{aligned}
z_ix_i\cdots x_{i+n-1}y&\underset{\eqref{EnlargeDuaL12}}{=}z_{m}x_{m}yx_1\cdots x_{n-1}
\underset{\eqref{dual3:5}}{=}yz_1x_1\cdots x_{n}.
\end{aligned}
\end{equation}

\noindent
$\bullet$ For all $i\in \llbracket m-n+2,m-1\rrbracket$, we have
\begin{equation}\label{EnlargeDuaL15}
\begin{aligned}
z_ix_i\cdots x_myx_1\cdots x_{i+n-m-1}&\underset{\eqref{L3Dual}}{=}z_mx_myx_1\cdots x_{n-1}\underset{\eqref{dual3:2}}{=}x_myz_1x_1\cdots x_{n-1}.
\end{aligned}
\end{equation}

\noindent
$\bullet$ For all $i\in \llbracket m-n+2,m-1\rrbracket$, we have
\begin{equation}\label{EnlargeDuaL16}
\begin{aligned}
z_ix_i\cdots x_myx_1\cdots x_{i+n-m-1}&\underset{\eqref{L3Dual}}{=}z_mx_myx_1\cdots x_{n-1}\underset{\eqref{dual3:5}}{=}yz_1x_1\cdots x_n.
\end{aligned}
\end{equation}

\noindent
$\bullet$ We have
\begin{equation}\label{EnlargeDuaL17}
\begin{aligned}
yz_1x_1\cdots x_n&\underset{\eqref{dual3:5}}{=}z_mx_myx_1\cdots x_{n-1}\underset{\eqref{dual3:2}}{=}x_myz_1x_1\cdots x_{n-1}.
\end{aligned}
\end{equation}

\noindent
Finally, the set of relations given by \eqref{L3Dual} with $i=m-n+1$, $j=m-n+2,\dots,m$ is the same set of relations as the one given by \eqref{EnlargeDuaL12} with $i=m-n+1$, $j=m-n+2,\dots,m$. Therefore, if we add relations \eqref{EnlargeDuaL12}, we can replace \eqref{L3Dual} by the set of relations
\begin{equation}\label{EnlargeDuaL18}
z_ix_i\cdots x_myx_1\cdots x_{i+n-m-1}=z_jx_j\cdots x_myx_1\cdots x_{j+n-m-1}, \, \, \forall m-n+2\leq i<j\leq m.
\end{equation}
\noindent
Combining equations \eqref{dual3:1},\eqref{dual3:2},\eqref{dual3:5} and \eqref{L2Dual} with equations \eqref{EnlargeDuaL1}-\eqref{EnlargeDuaL18}, we get the following corollary:

\begin{corollary}\label{DualMonoid2}
Let $n,m\in \N_{\geq 2}$ and assume that $n\leq m$ and $n\wedge m=1$. The monoid $\mathcal D_*^*(n,m)$ admits the following right-full presentation:

\begin{subequations}\label{DualMonoidPres2}
\begin{align}
    &(1) \,\, \mathrm{Generators}\!:\,  \{x_1,\dots,x_m,y,z_1\dots,z_m\};\notag\\
&(2) \,\, \mathrm{Relations}\!: \,\notag\\
  &x_iz_{i+1}=z_ix_i \, \forall i=1,\dots,m-1,\, \,\label{DualMonoidPres2:1}\\
  &x_myz_1=z_mx_my,\label{DualMonoidPres2:2}\\
  &z_ix_{i}\cdots x_{i+n-1}=z_jx_j\cdots x_{j+n-1} \,\forall \,  1\leq i<j\leq m-n+1,  \label{DualMonoidPres2:3} \\
  & z_ix_{i}\cdots x_{m}yx_1\cdots x_{i+n-m-1}=z_jx_j\cdots x_{m}yx_1\cdots x_{j+n-m-1}, \,(*_1),\label{DualMonoidPres2:4}\\
                    &yz_1x_1\cdots x_{n}=z_mx_myx_1\cdots x_{n-1},\label{DualMonoidPres2:5}\\
                     & x_iz_{i+1}x_{i+1}\cdots x_{i+n-1}=z_jx_j\cdots x_{j+n-1},\,\forall \,  1\leq i\neq j\leq m-n+1,\label{DualMonoidPres2:6}\\
                     &  x_iz_{i+1}x_{i+1}\cdots x_{i+n-1}=x_jz_{j+1}x_{j+1}\cdots x_{j+n-1}, \,\forall \,   1\leq i<j\leq m-n+1,\label{DualMonoidPres2:7}\\
                     &x_iz_{i+1}x_{i+1}\cdots x_{i+n-1}y=z_jx_j\cdots x_myx_1\cdots x_{j+n-m-1},\, (*_2),\label{DualMonoidPres2:8}\\
                     & x_iz_{i+1}x_{i+1}\cdots x_{i+n-1}y=x_jz_{j+1}x_{j+1}\cdots x_{m}yx_1\cdots x_{j+n-m-1}, \,(*_3) ,\label{DualMonoidPres2:9}\\
            & x_iz_{i+1}x_{i+1}\cdots x_{i+n-1}y=x_myz_1x_1\cdots x_{n-1} ,\,\forall \,  1\leq i\leq m-n+1,\label{DualMonoidPres2:10}\\
&x_iz_{i+1}x_{i+1}\cdots x_{i+n-1}y=yz_1x_1\cdots x_{n}, \,\forall \,   1\leq i\leq m-n+1,\label{DualMonoidPres2:11}\\
&x_iz_{i+1}x_{i+1}\cdots x_myx_1\cdots x_{i+n-m-1}=z_jx_j\cdots x_{j+n-1}y,\, (*_4),\label{DualMonoidPres2:12}\\
   & x_iz_{i+1}x_{i+1}\cdots x_myx_1\cdots x_{i+n-m-1}=z_jx_j\cdots x_myx_1\cdots x_{j+n-m-1}, \, (*_5),\label{DualMonoidPres2:13}\\
                & x_iz_{i+1}x_{i+1}\cdots x_{m}yx_1\cdots x_{i+n-m-1}=x_jz_{j+1}x_{j+1}\cdots x_myx_1\cdots x_{j+n-m-1},\, (*_6),\label{DualMonoidPres2:14}\\
                     &x_iz_{i+1}x_{i+1}\cdots x_{m}yx_1\cdots x_{i+n-m-1}=x_myz_1x_1\cdots x_{n-1}, \,\forall \,  m-n+2\leq i\leq m-1,\label{DualMonoidPres2:15}\\
                    &x_iz_{i+1}x_{i+1}\cdots x_{m}yx_1\cdots x_{i+n-m-1}=yz_1x_1\cdots x_n, \, \forall \,  m-n+2\leq i\leq m-1,\label{DualMonoidPres2:16}\\
                    & z_ix_i\cdots x_{i+n-1}y=z_jx_j\cdots x_myx_1\cdots x_{j+n-m-1},\,(*_7),\label{DualMonoidPres2:17}\\
                    &z_ix_i\cdots x_{i+n-1}y=x_myz_1x_1\cdots x_{n-1},\,\forall \,  1\leq i\leq m-n+1,\label{DualMonoidPres2:18}\\
                    &z_ix_i\cdots x_{i+n-1}y=yz_1x_1\cdots x_n, \,\forall \,  1\leq i\leq m-n+1,\label{DualMonoidPres2:19}\\
                    & z_ix_i\cdots x_myx_1\cdots x_{i+n-m-1}=x_myz_1x_1\cdots x_{n-1},\, \forall \, m-n+2\leq i\leq m-1,\label{DualMonoidPres2:20}\\
                    &z_ix_i\cdots x_myx_1\cdots x_{i+n-m-1}=yz_1x_1\cdots x_{n},\,\forall \,  m-n+2\leq i\leq m-1\label{DualMonoidPres2:21} \\
                    &yz_1x_1\cdots x_{n}=x_myz_1x_1\cdots x_{n-1}.\label{DualMonoidPres2:22}\\
                    &(*_1)\,\forall \,  m-n+2\leq i<j\leq m\notag\\
                    &(*_2)\,\forall \, i\in[m-n+1], j\in\llbracket m-n+2,m\rrbracket\notag\\
                    &(*_3)\,\forall \,  i\in[m-n+1], j\in \llbracket m-n+2,m-1\rrbracket\notag\\
                    &(*_4)\,\forall \,  i\in \llbracket {m-n+2},{m-1}\rrbracket,j\in [ m-n+1]\notag\\
                    &(*_5)\,\forall \, i\in \llbracket{m-n+2},{m-1}\rrbracket, j\in \llbracket {m-n+2},m\rrbracket,\, i\neq j \notag\\
                    &(*_6)\,\forall \, m-n+2\leq i<j\leq m-1 \notag\\
                    &(*_7)\,\forall \,   i\in [ m-n+1],j\in \llbracket {m-n+2},m\rrbracket\notag
				\end{align}\end{subequations}
    
\end{corollary}

\begin{remark}
In Presentation \eqref{DualMonoidPres2}, it is possible that some conditions on the indices are empty. In this case, there is simply no corresponding relations. For instance, if $n=2$ and $m=3$, the set of relations \eqref{DualMonoidPres2:21} is empty.
\end{remark}

\begin{lemma}\label{CubeDual}
Presentation \eqref{DualMonoidPres2} satisfies $(\hyperref[Poulet]{C1})$.
\end{lemma}
\begin{proof}
The proof is done in a similar way as the proof of Lemma \ref{CubeClassical}. The complete details of this proof will appear in the author's thesis \cite{These}.
\end{proof}

\noindent
\fbox{The case $n=1$:}\\
In this case, the monoid $\mathcal D_*^*(n,m)$ admits the presentation 
\begin{equation}\label{dualN=1}\left\langle
       \begin{array}{l|cl}
           x_1,\dots,x_{m},y     \,        & x_iz_{i+1}=z_ix_i=z_{i+1}x_{i+1} \, \, \forall i=1,\dots,m-1,\\
      z_1,\dots,z_{m}     \,   &\,    x_myz_1=yz_1x_1=z_{m}x_{m}y

                          \end{array}
     \right\rangle.\end{equation}
\noindent 
Now, we will enlarge Presentation \eqref{dualN=1} to get a right-full homogeneous presentation satisfying $(\hyperref[Poulet]{C1})$.\\
First, observe that the first line of Presentation \eqref{dualN=1} is equivalent to the sets of relations 
\begin{equation}\label{EnlargeDualN1}
x_iz_{i+1}=z_jx_j, \, \forall i\in [m-1], j\in [m]
\end{equation}
and 
\begin{equation}\label{EnlargeDualN2}
z_ix_i=z_jx_j,\, \forall 1\leq i<j\leq m.
\end{equation}
We also refer to the second line of Presentation \eqref{dualN=1} as
\begin{equation}\label{L3N1}
    x_myz_1=yz_1x_1=z_mx_my.
\end{equation}
Additionally to the relations \eqref{EnlargeDualN1}-\eqref{L3N1}, observe the following:\\
\noindent
$\bullet$ For all $1\leq i<j\leq m-1$, we have
\begin{equation}\label{EnlargeDualN3}
\begin{aligned}
    x_iz_{i+1}&\underset{\eqref{EnlargeDualN1}}{=}z_1x_1\underset{\eqref{EnlargeDualN1}}{=}x_jz_{j+1}.
    \end{aligned}
\end{equation}

\noindent
$\bullet$ For all $i\in [m]$, we have
\begin{equation}\label{EnlargeDualN4}
    \begin{aligned}
        z_ix_iy&\underset{\eqref{EnlargeDualN2}}{=}z_mx_my\underset{\eqref{L3N1}}{=}yz_1x_1.
    \end{aligned}
\end{equation}

\noindent
$\bullet$ For all $i\in [m]$, we have
\begin{equation}\label{EnlargeDualN5}
    \begin{aligned}
        z_ix_iy&\underset{\eqref{EnlargeDualN2}}{=}z_mx_my\underset{\eqref{L3N1}}{=}x_myz_1.
    \end{aligned}
\end{equation}

\noindent
$\bullet$ For all $i\in [m-1]$, we have
\begin{equation}\label{EnlargeDualN6}
    \begin{aligned}
        x_iz_{i+1}y&\underset{\eqref{EnlargeDualN1}}{=}z_mx_my\underset{\eqref{L3N1}}{=}yz_1x_1.
    \end{aligned}
\end{equation}

\noindent
$\bullet$ For all $i\in [m-1]$, we have
\begin{equation}\label{EnlargeDualN7}
    \begin{aligned}
        x_iz_{i+1}y&\underset{\eqref{EnlargeDualN1}}{=}z_mx_my\underset{\eqref{L3N1}}{=}x_myz_1.
    \end{aligned}
\end{equation}

\noindent
Finally, observe that the relation $x_myz_1=z_mx_my$ is covered by both sets of relations \eqref{L3N1} and \eqref{EnlargeDualN5}. We thus remove it from \eqref{L3N1}. Using equations \eqref{EnlargeDualN1}-\eqref{EnlargeDualN7}, we get the following corollary:

\begin{corollary}\label{DualMonoidN}
Let $m\in \N^*$. The monoid $\mathcal D_*^*(1,m)$ admits the following presentation:

\begin{equation}\label{dualN=12}\left\langle
       \begin{array}{l|cl}
        \,        & z_ix_i=z_jx_j \, \forall 1\leq i<j\leq m\\
    \,   &\, x_iz_{i+1}=z_jx_j\, \forall i\in [m-1],j\in [m] \\ 
      x_1,\dots,x_{m},y        \, & \, x_{i}z_{i+1}=x_jz_{j+1}\, \forall 1\leq i<j\leq m-1 \\
       z_1,\dots,z_{m}      \, & \, z_ix_iy=yz_1x_1,\, z_ix_iy=x_myz_1 \, \forall i\in [m]\\
      \, & \, x_{i}z_{i+1}y=yz_1x_1, \, x_iz_{i+1}y=x_myz_1 \, \forall i\in  [m-1]\\
      \,& \,x_myz_1=yz_1x_1, \,

                          \end{array}
     \right\rangle.\end{equation}

\end{corollary}

\begin{remark}
In Presentation \eqref{dualN=12}, it is possible that some conditions on the indices are empty. In this case, there is simply no corresponding relations.
\end{remark}

\begin{lemma}\label{CubedualN=1}
Presentation \eqref{dualN=12} satisfies $(\hyperref[Poulet]{C1})$.
\end{lemma}
\begin{proof}
The proof is a straightforward check, done in a similar way as in the proof of Lemma \ref{ClassicalLeftCanc}. The complete details of this check will appear in the author's thesis.
\end{proof}

\begin{proof}[Proof of Proposition \ref{LeftCancDual}]
It follows by combining Lemmas \ref{CubeDual} and \ref{CubedualN=1} with Lemma \ref{MagicCube}. 
\end{proof}

\subsection{\texorpdfstring{The monoid $\D_*^*(n,m)$ is right-cancellative}{}}
\begin{proposition}\label{LeftCancDualOp}
Let $n,m\in\N^*$ with $n\wedge m=1$ and $n\leq m$. The monoid $\D_*^*(n,m)$ is right-cancellative and admits conditional left-lcms.
\end{proposition}
To prove Proposition \ref{LeftCancDualOp}, we show the equivalent statement that $\D_*^*(n,m)^{\mathrm{op}}$ is left-cancellative and admits conditional right-lcms.  In order to do so, we use the following convenient presentation of $\D_*^*(n,m)^{\mathrm{op}}$:

\begin{lemma}\label{OpDual}
Let $n,m\in \N^*$ and assume that $n\leq m$ and $n\wedge m=1$. The monoid $\mathcal D_*^*(n,m)^{\mathrm{op}}$ admits the following presentation:
\begin{subequations}\label{OpDualPres1}
\begin{align}
    &(1) \,\, \mathrm{Generators}\!:\,  \{x_1,\dots,x_m,y,z_1,\dots,z_m\};\notag\\
&(2) \,\, \mathrm{Relations}\!: \,\notag\\
  & z_ix_{i+1}=x_{i+1}z_{i+1}, \, \forall \, i=1,\dots,m-1, \, \label{OpDualPres1:1}\\ &z_{m}yx_1=yx_{1}z_{1},\label{OpDualPres1:2}\\
  &  x_i\cdots x_{n+i-1}z_{n+i-1}=x_{j}\cdots x_{n+j-1}z_{n+j-1} ,\, \forall \,1\leq i<j\leq m-n+1,\label{OpDualPres1:3}\\
  &x_i\cdots x_myx_1\cdots x_{i+n-m-1}z_{i+n-m-1}=x_{j}\cdots x_myx_1\cdots x_{j+n-m-1}z_{j+n-m-1},\, (*),\label{OpDualPres1:4}\\
 & \, x_{m-n+1}\cdots x_mz_my=x_{m-n+2}\cdots x_myx_1z_1.\label{OpDualPres1:5}\\
        &(*)\,\forall \, m-n+2\leq i<j\leq m+1\notag
\end{align}
\end{subequations}
\end{lemma}

\begin{proof}
Using Presentation \eqref{dual3}, we see that $\mathcal D_*^*(n,m)^{\mathrm{op}}$ admits the presentation 
\begin{equation}\label{OpDualPres2}
\begin{aligned}
    &(1) \,\, \mathrm{Generators}\!:\,  \{x_1,\dots,x_m,y,z_1,\dots,z_m\};\\
&(2) \,\, \mathrm{Relations}\!: \\
  & z_{i+1}x_i=x_iz_i ,\,  \forall\, i=1,\dots,m-1, \, \, z_1yx_m=yx_mz_m,\\
  &  x_{n+i}\cdots x_{i+1}z_{i+1}=x_{i+n-1}\cdots x_iz_i,\, \forall \,1\leq i\leq m-n,\\
  &x_{i+n-m}\cdots x_1yx_m\cdots x_{i+1}z_{i+1}=x_{i+n-m-1}\cdots x_1yx_m \cdots x_iz_i ,\, (*),  \\
 &   x_n\cdots x_1 z_1 y=x_{n-1}\cdots x_1yx_mz_m,\\
 & (*)\,\forall \, m-n+1\leq i\leq m-1
\end{aligned}
\end{equation}
with the convention that $x_0\dots x_1yx_m\cdots x_{m-n+1}z_{m-n+1}=yx_m\cdots x_{m-n+1}z_{m-n+1}$.\\
Relabelling $x_i$ by $f_{m-i+1}$ and $z_i$ by $g_{m-i+1}$ for all $i=1,\dots,m$ in Presentation \eqref{OpDualPres2} yields the presentation
\begin{equation}\label{OpDualPres3}
\begin{aligned}
    &(1) \,\, \mathrm{Generators}\!:\,  \{f_1,\dots,f_m,y,g_1,\dots,g_m\};\\
&(2) \,\, \mathrm{Relations}\!: \,\\
  &g_{m-i}f_{m-i+1}=f_{m-i+1}g_{m-i+1}, \, \, \forall i=1,\dots,m-1, \, \, g_myf_1=yf_1g_1,\\
  &   f_{m-n-i+1}\cdots f_{m-i}g_{m-i}=f_{m-n-i+2}\cdots f_{m-i+1}g_{m-i+1},\,\forall \,1\leq i\leq m-n,\\
  & f_{2m-n-i+1}\cdots f_myf_1\cdots f_{m-i}g_{m-i}=f_{2m-i-n+2}\cdots f_myf_1 \cdots f_{m-i+1}g_{m-i+1},\,(*),\\
 & \,  f_{m-n+1}\cdots f_m g_m y=f_{m-n+2}\cdots f_myf_1g_1.\\
 & (*)\,\forall \, m-n+1\leq  i\leq m-1
\end{aligned}
\end{equation}
Now, writing $u_k=g_{m-k}f_{m-k+1}$ and $v_k=f_{m-k+1}g_{m-k+1}$ for all $k=1,\dots,m-1$, the set of relations
$$g_{m-i}f_{m-i+1}=f_{m-i+1}g_{m-i+1}, \, \forall i=1,\dots,m-1$$ 
is equivalent to saying that $u_k=v_k$ for all $k=1,\dots,m-1$.\\
We have
\begin{equation*}
    \begin{aligned}
        u_1=g_{m-1}f_m&, \, v_1=f_mg_m\\
        u_2=g_{m-2}f_{m-1}&,\, v_2=f_{m-1}g_{m-1}\\
        & \vdots\\
        u_{m-1}=g_1f_2&,\, v_{m-1}=f_2g_2.
    \end{aligned}
\end{equation*}
Thus, the set of relations $\{u_i=v_i\}_{i=1,\dots,m-1}$ is the set of relations 
$$g_if_{i+1}=f_{i+1}g_{i+1},\, \forall i=1,\dots,m-1.$$
With similar computations for the second and third lines of Presentation \eqref{OpDualPres3}, we obtain that $\mathcal D_*^*(n,m)^{\mathrm{op}}$ admits the following presentation: 
\begin{equation*}
\begin{aligned}
    &(1) \,\, \mathrm{Generators}\!:\,  \{f_1,\dots,f_m,y,g_1,\dots,g_m\};\\
&(2) \,\, \mathrm{Relations}\!: \,\\
  &g_if_{i+1}=f_{i+1}g_{i+1}, \, \, \forall i=1,\dots,m-1, \, \, g_myf_1=yf_1g_1,\\
  &   f_i\cdots f_{i+n-1}g_{i+n-1}=f_i\cdots f_{j+n-1}g_{j+n-1}, \,\forall\, 1\leq i<j\leq m-n+1, \\
  & f_{i}\cdots f_myf_1\cdots f_{i+n-m-1}g_{i+n-m-1}=f_{j}\cdots f_myf_1\cdots f_{j+n-m-1}g_{j+n-m-1},\,(*), \\
 & \, f_{m-n+1}\cdots f_m g_m y=f_{m-n+2}\cdots f_myf_1g_1.\\
 &(*)\,\forall \,  m-n+2\leq i<j\leq m+1
\end{aligned}
\end{equation*}
Relabeling $f_k$ by $x_k$ and $g_k$ by $z_k$ for each $k\in [m]$ concludes the proof.
\end{proof}
\noindent
Now, we will enlarge Presentation \eqref{OpDualPres1} to get a right-full homogeneous presentation satisfying $(\hyperref[Poulet]{C1})$. In order to facilitate the computations, we prove the following lemma:
\begin{lemma}\label{OpDualzSlide}
In $\mathcal D_*^*(n,m)^{\mathrm{op}}$, the following relations hold:\\
$(i)$ For all $i\in [m-n+1]$ we have
\begin{equation}\label{DualOpModify1}
x_i\cdots x_{i+n-1}z_{i+n-1}=x_iz_ix_{i+1}\cdots x_{i+n-1}. 
\end{equation}
$(ii)$ For all $i\in \llbracket m-n+1,m\rrbracket$ we have
\begin{equation}\label{DualOpModify2}
x_i\cdots x_myx_1\cdots x_{i+n-m-1}z_{i+n-m-1}=x_iz_ix_{i+1}\cdots x_myx_1\cdots x_{i+n-m-1},
\end{equation}
where $x_mz_mx_{m+1}\cdots x_myx_1\cdots x_{n-1}=x_mz_myx_1\cdots x_{n-1}$.\\
$(iii)$ We have
\begin{equation}\label{DualOpModify3}
yx_1\cdots x_nz_n=yx_1z_1x_2\cdots x_n.
\end{equation}
\end{lemma}

\begin{proof}
For $i\in [m-n+1]$, we have
\begin{equation*}
    \begin{aligned}
        x_{i}\cdots x_{i+n-1} z_{i+n-1}&\underset{\eqref{OpDualPres1:1}}{=}x_{i}\cdots x_{i+n-2}z_{i+n-1}x_{i+n-1} \\
        &\underset{\eqref{OpDualPres1:1}}{=}\cdots
        \underset{\eqref{OpDualPres1:1}}{=}x_{i}z_{i+1}x_{i+1}\cdots x_{i+n-1}.
    \end{aligned}
\end{equation*}
This establishes \eqref{DualOpModify1}.\\
For $i\in \llbracket m-n+1,m\rrbracket$, we have

\begin{equation*}
    \begin{aligned}
        x_{i}\cdots x_myx_1\cdots x_{i+n-m-1}z_{i+n-m-1}&\underset{\eqref{OpDualPres1:1}}{=}x_{i}\cdots x_myx_1z_1x_2\cdots x_{i+n-m-1}\\
        \underset{\eqref{OpDualPres1:2}}{=}x_{i}\cdots x_{m} z_myx_1\cdots x_{i+n-m-1}&\underset{\eqref{OpDualPres1:1}}{=} x_iz_{i+1}x_{i+1}\cdots x_myx_1\cdots x_{i+n-m}.
    \end{aligned}
\end{equation*}
This establishes \eqref{DualOpModify2}. Finally, we have 
\begin{equation*}
\begin{aligned}
    yx_1\cdots x_nz_n&\underset{\eqref{OpDualPres1:1}}{=}yx_1z_1x_2\cdots x_{n}.
\end{aligned}    
\end{equation*} This establishes \eqref{DualOpModify3} and concludes the proof.
\end{proof}

\begin{corollary}\label{OpDualInter}
The monoid $\mathcal D_*^*(n,m)^\mathrm{op}$ admits the following presentation:

\begin{subequations}\label{OpDualInter1}
\begin{align}
    &(1) \,\, \mathrm{Generators}\!:\,  \{x_1,\dots,x_m,y,z_1,\dots,z_m\};\notag\\
&(2) \,\, \mathrm{Relations}\!: \,\notag\\
  & z_ix_{i+1}=x_{i+1}z_{i+1}, \, \forall \, i=1,\dots,m-1, \, \label{OpDualInter1:1}\\ &z_{m}yx_1=yx_{1}z_{1},\label{OpDualInter1:2}\\
  &  x_iz_ix_{i+1}\cdots x_{i+n-1}=x_{j}z_jx_{j+1}\cdots x_{j+n-1} ,\, \forall \,1\leq i<j\leq m-n+1,\label{OpDualInter1:3}\\
  &x_iz_ix_{i+1}\cdots x_myx_1\cdots x_{i+n-m-1}=x_{j}z_jx_{j+1}\cdots x_myx_1\cdots x_{j+n-m-1},\, (*),\label{OpDualInter1:4}\\
 & \, x_iz_ix_{i+1}\cdots x_{i+n-1}y=x_{j}z_jx_{j+1}\cdots x_myx_1\cdots x_{j+n-m-1}, \, (**),\label{OpDualInter1:5}\\
 & yx_1z_1x_2\cdots x_n=x_{j}z_jx_{j+1}\cdots x_myx_1\cdots x_{j+n-m-1}, \, \forall m-n+2\leq j \leq, m\label{OpDualInter1:6}\\
&(*)\,\forall \, m-n+2\leq i<j\leq m\notag\\
&(**)\, \forall i\in [m-n+1], \, j\in \llbracket m-n+2,m\rrbracket \notag
\end{align}
\end{subequations}
where $x_mz_mx_{m+1}\cdots x_myx_1\cdots x_{n-1}=x_mz_myx_1\cdots x_{n-1}$.
\end{corollary}

\begin{proof}
The sets of relations \eqref{OpDualInter1:1} and \eqref{OpDualInter1:2} are the same as \eqref{OpDualPres1:1} and \eqref{OpDualPres1:2}. Moreover, by \eqref{DualOpModify1} the sets of relations \eqref{OpDualInter1:3} and \eqref{OpDualPres1:3} are equivalent. Similarly, by \eqref{DualOpModify2} the set of relations \eqref{OpDualInter1:4} is equivalent to the subset of relations of \eqref{OpDualPres1:4} with $i\in [m-n+1]$ and $j\in \llbracket m-n+2,m\rrbracket$, and by \eqref{DualOpModify3} the subset of \eqref{OpDualPres1:4} given by $i\in [m-n+1]$ and $j=m+1$ is equivalent to \eqref{OpDualInter1:6}.\\
Finally, we see that the sets of relations \eqref{OpDualPres1:3} and \eqref{OpDualPres1:4} combined with the relation \eqref{OpDualPres1:5} imply that for all $i\in [m-n+1]$ and $j\in \llbracket m-n+2,m\rrbracket$, we have
\begin{equation}\label{DualInterModify1}
    x_i\cdots x_{i+n-1}z_{i+n-1}y=x_j\cdots x_myx_1\cdots x_{j+n-m-1}z_{j+n-m-1}.
\end{equation}
Using \eqref{OpDualInter1:3} and \eqref{OpDualInter1:4}, the set of relations given by \eqref{DualInterModify1} is equivalent to \eqref{OpDualInter1:5}. This concludes the proof.
\end{proof}

In addition to the relations of Presentation \eqref{OpDualInter1}, observe the following:\\

\noindent
$\bullet$ For all $i\in [m-n+1]$, $j\in [m-n]$ with $i\neq j+1$, we have
\begin{equation}\label{EnlargeDualOp1}
\begin{aligned}
x_iz_ix_{i+1}\cdots x_{i+n-1}&\underset{\eqref{OpDualInter1:3}}{=}x_{j+1}z_{j+1}x_{j+2}\cdots x_{j+n}z_{j+n}\underset{\eqref{OpDualInter1:1}}{=}z_{j}x_{j+1}\cdots x_{j+n}.
\end{aligned}
\end{equation}

\noindent
$\bullet$ For all $i\in [m-n+1]$, $j\in \llbracket m-n+1,m-1\rrbracket$, we have
\begin{equation}\label{EnlargeDualOp2}
\begin{aligned}
x_iz_i x_{i+1}\cdots x_{i+n-1}y&
\underset{\eqref{OpDualInter1:4}}{=}x_{j+1}z_{j+1}x_{j+2}\cdots x_myx_1\cdots x_{j+n-m}\\& \underset{\eqref{OpDualInter1:1}}{=}z_{j}x_{j+1}\cdots x_myx_1\cdots x_{j+n-m}.
\end{aligned}
\end{equation}

\noindent
$\bullet$ For all $i\in [m-n+1]$, we have
\begin{equation}\label{EnlargeDualOp22}
\begin{aligned}
x_iz_i x_{i+1}\cdots x_{i+n-1}y&
\underset{\eqref{OpDualInter1:5}}{=}x_{m}z_{m}yx_1\cdots x_{n-1}\\ \underset{\eqref{OpDualInter1:6}}{=}yx_1z_1x_2&\cdots x_n\underset{\eqref{OpDualInter1:2}}{=}z_myx_1x_2\cdots x_n.
\end{aligned}
\end{equation}

\noindent
$\bullet$ For all $i\in \llbracket m-n+2,m\rrbracket$, $j\in [m-n]$, we have
\begin{equation}\label{EnlargeDualOp4}
\begin{aligned}
x_iz_ix_{i+1}\cdots x_{m}yx_1\cdots x_{i+n-m-1} &\underset{\eqref{OpDualInter1:5}}{=}x_1z_1x_2\cdots x_ny
\underset{\eqref{EnlargeDualOp1}}{=}z_{j}x_{j+1}\cdots x_{j+n}y.
\end{aligned}
\end{equation}

\noindent
$\bullet$ For all $i\in \llbracket m-n+2,m\rrbracket$, $j\in \llbracket m-n+1,m\rrbracket$ with $i\neq j+1$, we have
\begin{equation}\label{EnlargeDualOp5}
\begin{aligned}
x_iz_ix_{i+1}\cdots x_{m}yx_1\cdots x_{i+n-m-1}&\underset{\eqref{OpDualInter1:4}}{=}x_1z_1x_2\cdots x_ny\\
\underset{\eqref{EnlargeDualOp2}/\eqref{EnlargeDualOp22}}{=}&z_{j}x_{j+1}\cdots x_{m}yx_1\cdots x_{j+n-m},
\end{aligned}
\end{equation}
where $z_mx_{m+1}\cdots x_myx_1\cdots x_n=z_myx_1\cdots x_n$ and the use of \eqref{EnlargeDualOp2} or \eqref{EnlargeDualOp22} depends on whether or not $j=m$.

\noindent
$\bullet$ For all $1\leq i<j\leq m-n$, we have
\begin{equation}\label{EnlargeDualOp6}
\begin{aligned}
z_ix_{i+1}\cdots x_{i+n}&\underset{\eqref{EnlargeDualOp1}}{=}x_{1}z_{1}x_{2}\cdots x_{n}
\underset{\eqref{EnlargeDualOp1}}{=}z_jx_{j+1}\cdots x_{j+n}.
\end{aligned}
\end{equation}

\noindent
$\bullet$ For all $i\in [m-n]$, $j\in \llbracket m-n+1,m\rrbracket$, we have
\begin{equation}\label{EnlargeDualOp7}
\begin{aligned}
z_ix_{i+1}\cdots x_{i+n}y&\underset{\eqref{EnlargeDualOp4}}{=}x_{m}z_myx_1\cdots x_{n}\underset{\eqref{EnlargeDualOp5}}{=}z_jx_{j+1}\cdots x_myx_1\cdots x_{j+n-m}.
\end{aligned}
\end{equation}

\noindent
$\bullet$ For all $m-n+1\leq i<j\leq m$, we have
\begin{equation}\label{EnlargeDualOp8}
\begin{aligned}
z_ix_{i+1}\cdots x_myx_1\cdots x_{i+n-m}&\underset{\eqref{EnlargeDualOp7}}{=}z_1x_{2}\cdots x_{n+1}y\underset{\eqref{EnlargeDualOp7}}{=}z_jx_{j+1}\cdots x_myx_1\cdots x_{j+n-m}.
\end{aligned}
\end{equation}

\noindent
$\bullet$ For all $i\in[m-n+1]$, we have
\begin{equation}\label{EnlargeDualOp9}
\begin{aligned}
x_{i}z_{i+1}x_{i+1}\cdots x_{i+n)1}y&\underset{\eqref{OpDualInter1:5}}{=}x_mz_myx_1\cdots x_{n-1}\underset{\eqref{OpDualInter1:6}}{=}yx_1z_1x_2\cdots x_{n}
\end{aligned}
\end{equation}

\noindent
$\bullet$ For all $i\in [m-n]$, we have
\begin{equation}\label{EnlargeDualOp10}
\begin{aligned}
z_ix_{i+1}\cdots x_{i+n}y&\underset{\eqref{EnlargeDualOp4}}{=}x_mz_myx_1\cdots x_{n-1}\underset{\eqref{OpDualInter1:6}}{=}yx_1z_1x_2\cdots x_{n}
\end{aligned}
\end{equation}

\noindent
$\bullet$ For all $m-n+1\leq i\leq m$, we have
\begin{equation}\label{EnlargeDualOp11}
\begin{aligned}
z_ix_{i+1}\cdots x_myx_1\cdots x_{i+n-m}&\underset{\eqref{EnlargeDualOp7}}{=}z_1x_{2}\cdots x_{n+1}y\underset{\eqref{EnlargeDualOp10}}{=}yx_1z_1x_2\cdots x_n
\end{aligned}
\end{equation}

Using equations \eqref{EnlargeDualOp1}-\eqref{EnlargeDualOp11}, we get the following corollary:

\begin{corollary}\label{OpDual2}
Let $n,m\in \N^*$ and assume $n\leq m$, $n\wedge m=1$. The monoid $\mathcal D_*^*(n,m)^{\mathrm{op}}$ admits the following presentation:
\begin{subequations}\label{OpDualPres4}
\begin{align}
    &(1) \,\, \mathrm{Generators}\!:\,  \{x_1,\dots,x_m,y,z_1,\dots,z_m\};\notag\\
&(2) \,\, \mathrm{Relations}\!: \,\notag\\
  &z_ix_{i+1}=x_{i+1}z_{i+1} \, \, \forall i=1,\dots,m-1,\label{OpDualPres4:1} \\
  &z_{m}yx_1=yx_{1}z_{1},\label{OpDualPres4:2}\\
  &  x_iz_ix_{i+1}\cdots x_{n+i-1}=x_{j}z_jx_{j+1}\cdots x_{n+j-1}, \, \forall \, 1\leq i<j\leq m-n+1, \label{OpDualPres4:3}\\
  &x_iz_ix_{i+1}\cdots x_myx_1\cdots x_{i+m-n-1}=x_{j}z_jx_{j+1}\cdots x_myx_1\cdots x_{j+m-n-1},\, (*_1),\label{OpDualPres4:4}\\
 & \, x_{i}z_ix_{i+1}\cdots x_{i+n-1}=z_jx_{j+1}\cdots x_{j+n} ,\, (*_2),\label{OpDualPres4:5}\\
 & x_{i}z_ix_{i+1}\cdots x_{i+n-1}y=z_jx_{j+1}\cdots x_myx_1\cdots x_{j+n-m}, \, (*_3),\label{OpDualPres4:6}\\
 & x_{i}z_ix_{i+1}\cdots x_{i+n-1}y=x_jz_jx_{j+1}\cdots x_myx_1\cdots x_{j+n-m-1},\,(*_4),\label{OpDualPres4:7}\\
 &\, x_iz_ix_{i+1}\cdots x_myx_1\cdots x_{i+n-m-1}=z_jx_{j+1}\cdots x_{j+n}y,\,(*_5),\label{OpDualPres4:8}\\
 & x_iz_ix_{i+1}\cdots x_myx_1\cdots x_{i+n-m-1}=z_jx_{j+1}\cdots x_myx_1\cdots x_{j+n-m}, \,(*_6),\label{OpDualPres4:9}\\
 & z_ix_{i+1}\cdots x_{i+n}=z_jx_{j+1}\cdots x_{j+n},\, \forall 1\leq i<j\leq m-n,\label{OpDualPres4:10}\\
 &z_ix_{i+1}\cdots x_{i+n}y=z_jx_{j+1}\cdots x_myx_1\cdots x_{j+n-m},\,(*_7),\label{OpDualPres4:11}\\
 & z_ix_{i+1}\cdots x_myx_1\cdots x_{i+n-m}=z_jx_{j+1}\cdots x_myx_1\cdots x_{j+n-m},\, (*_8),\label{OpDualPres4:12}\\
 & yx_1z_1x_2\cdots x_n=x_iz_ix_{i+1}\cdots x_{n+i-1}y ,\,\forall \, 1\leq i\leq m-n+1,\label{OpDualPres4:13}\\
 & yx_1z_1x_2\cdots x_n=x_iz_ix_{i+1}\cdots x_myx_1\cdots x_{i+n-m-1},\,\forall \, m-n+2\leq i\leq m,\label{OpDualPres4:14}\\
 & yx_1z_1x_2\cdots x_n=z_ix_{i+1}\cdots x_{i+n}y, \,\forall \, 1\leq i\leq m-n,\label{OpDualPres4:15}\\
 & yx_1z_1x_2\cdots x_n=z_ix_{i+1}\cdots x_myx_1\cdots x_{i+n-m}, \,\forall \, m-n+1\leq i\leq m,\label{OpDualPres4:16}\\
 &(*_1) \,\forall \, m-n+2\leq i<j\leq m\notag\\
 &(*_2)\,\forall \, i\in [m-n+1] ,j\in [m-n] , i\neq j+1 \notag\\
 &(*_3)\,\forall \,i\in [m-n+1],j\in \llbracket m-n+1,m\rrbracket\notag\\
 &(*_4)\, \forall \, i\in [m-n+1],j\in \llbracket m-n+2,m\rrbracket\notag\\
 &(*_5)\,\forall \,i\in \llbracket m-n+2,m\rrbracket ,j\in [m-n] \notag\\
 &(*_6)\,\forall \, i\in \llbracket m-n+2,m\rrbracket, j\in \llbracket m-n+1,m\rrbracket , i\neq j+1 \notag\\
 &(*_7)\, \forall \, i\in [m-n],j\in \llbracket m-n+1,m\rrbracket\notag\\
 &(*_8)\, \forall \,m-n+1\leq i<j\leq m\notag
\end{align}
\end{subequations}
where $z_mx_{m+1}\cdots x_myx_1\cdots x_n=z_myx_1\cdots x_n$ and  $x_mz_mx_{m+1}\cdots x_myx_1\cdots x_{n-1}=x_mz_myx_1\cdots x_{n-1}$.
\end{corollary}

\begin{proof}
Observe that the sets of relations \eqref{OpDualPres4:1}-\eqref{OpDualPres4:4}, \eqref{OpDualPres4:7} and \eqref{OpDualPres4:14} constitute the relations of Presentation \eqref{OpDualInter}. The remaining relations are the redundant sets of relations \eqref{EnlargeDualOp1}-\eqref{EnlargeDualOp11}, where \eqref{EnlargeDualOp2} and \eqref{EnlargeDualOp22} are combined to get \eqref{OpDualPres4:6}.
\end{proof}
\begin{remark}
In Presentation \eqref{ClassicalMonoidPres2}, it is possible that some conditions on the indices are empty. In this case, there is simply no corresponding relations.
\end{remark}
\begin{lemma}\label{CubeOpDual}
Presentation \eqref{OpDualPres4} satisfies $(\hyperref[Poulet]{C1})$.
\end{lemma}

\begin{proof}
The proof is a straightforward check, done in a similar way as in the proof of Lemma \ref{ClassicalLeftCanc}. The complete details of this check will appear in the author's thesis.
\end{proof}

\begin{proof}[Proof of Proposition \ref{LeftCancDualOp}]
It follows by combining Lemma \ref{CubeOpDual} with Lemma \ref{MagicCube}.
\end{proof}

\subsection{\texorpdfstring{The monoid $\D_*^*(n,m)$ is Garside}{}}
\noindent
Lemmas \ref{CubeDual}, \ref{CubedualN=1} and \ref{CubeOpDual} show that $\mathcal D_*^*(n,m)$ is cancellative and admits conditional left- and right-lcms. In order to conclude that $\mathcal D_*^*(n,m)$ is a Garside monoid, we exhibit a Garside element.

\begin{definition}
Define $w\in \mathcal D_*^*(n,m)$ to be the element represented by the word $z_1x_1\cdots x_n$ in Presentation \eqref{dual3}. Moreover, define $W\in \mathcal D_*^*(n,m)$ to be the element represented by the word $z_{m-n+1}x_{m-n+1}\cdots x_my$ in Presentation \eqref{dual3}. Finally, define $\Delta\in \mathcal D_*^*(n,m)$ to be $w^{m-n}W^n$.
\end{definition}

\begin{remark}\label{wWwords}
By \eqref{dual3:3}, for every $i\in [m-n+1]$ the word $z_ix_i\cdots x_{i+n-1}$ represents $w$. Moreover, for every $i\in [m-n]$ we have $x_iz_{i+1}=z_ix_i$, hence the word $x_iz_{i+1}x_{i+1}\cdots x_{i+n-1}$ represents $w$ as well.\\
Similarly, by \eqref{dual3:4}, for every $i\in \llbracket m-n+1,m\rrbracket$ the word $z_ix_i\cdots x_myx_1\cdots x_{i+n-m-1}$ represents $W$. Moreover, for every $i\in \llbracket m-n+1,m-1\rrbracket$ we have $x_iz_{i+1}=z_ix_i$, hence the word $x_iz_{i+1}x_{i+1}\cdots x_myx_1\cdots x_{i+n-m-1}$ represents $W$. In addition, since $x_myz_1=z_mx_my$ the word $x_myz_1x_1\cdots x_{n-1}$ represents $W$ as well. Finally, by \eqref{dual3:5} the word $yz_1x_1\cdots x_n$ represents $W$.
\end{remark}

\begin{lemma}\label{GarsideDualLemma1}
For all $i\in [m-n]$, we have $x_iw=wx_{i+n}$ and $z_iw=wz_{i+n}$. Moreover, we have $yw=wy$.
\end{lemma}

\begin{proof}
Let $i\in [m-n]$. Then we have 
\begin{eqnarray*}
x_iw&=&x_iz_{i+1}x_{i+1}\cdots x_{i+n} \, \, \text{by Remark \ref{wWwords}}\\
&\underset{\eqref{dual3:1}}{=}&(z_ix_i\cdots x_{i+n-1})x_{i+n}
=wx_{i+n} \, \, \text{by Remark \ref{wWwords}.}
\end{eqnarray*}
Similarly, we have
\begin{eqnarray*}
z_iw&=&z_iz_ix_i\cdots x_{i+n-1} \, \, \text{by Remark \ref{wWwords}}\\
&\underset{\eqref{dual3:1}}{=}&z_ix_iz_{i+1}x_{i+1}\cdots x_{i+n-1}
=\cdots
\underset{\eqref{dual3:1}}{=}(z_ix_i\dots x_{i+n-1})z_{i+n}\\
&=&wz_{i+n} \, \, \text{by Remark \ref{wWwords}.}
\end{eqnarray*}
Finally, we have
\begin{eqnarray*}
yw&=&yz_1x_1\cdots x_n 
\underset{\eqref{dual3:5}}{=}z_mx_myx_1\cdots x_{n-1}\\
&\underset{\eqref{dual3:4}}{=}& z_{m-n+1}x_{m-n+1}\cdots x_my
= wy, \, \text{by Remark \ref{wWwords}.}
\end{eqnarray*}
This concludes the proof.
\end{proof}

\begin{lemma}\label{GarsideDualLemma2}
For all $i\in \llbracket m-n+1,m\rrbracket$, we have $x_iW=Wx_{i+n}$ and $z_iW=Wz_{i+n}$, where indices are taken modulo $m$.
\end{lemma}

\begin{proof}
Let $i\in \llbracket m-n+1,m-1\rrbracket$. Then we have
\begin{eqnarray*}
x_iW&=&x_iz_{i+1}x_{i+1}\cdots x_myx_1\cdots x_{i+n-m}\, \, \text{by Remark \ref{wWwords}}\\
&\underset{\eqref{dual3:1}}{=}&(z_ix_ix_{i+1}\cdots x_myx_1\cdots x_{i+n-m-1})x_{i+n-m}
=Wx_{i+n-m} \, \, \text{by Remark \ref{wWwords}}\\
&=&Wx_{i+n} \, \, \text{since indices are taken modulo} \, m.
\end{eqnarray*}
Similarly, we have
\begin{eqnarray*}
x_mW&=&x_mz_mx_myx_1\cdots x_{n-1} \, \, \text{by Remark \ref{wWwords}}\\
&\underset{\eqref{dual3:5}}{=}&x_myz_1x_1\cdots x_n 
\underset{\eqref{dual3:2}}{=}(z_mx_myx_1\cdots x_{n-1})x_n
=Wx_n \, \, \text{by Remark \ref{wWwords}}\\
&=&Wx_{m+n} \, \, \text{since indices are taken modulo} \, m.
\end{eqnarray*}
Finally, for all $i\in \llbracket m-n+1,m\rrbracket$, we have
\begin{equation*}
\begin{aligned}
 &z_iW=z_iz_ix_i\cdots x_myx_1\cdots x_{i+n-m-1} \, \, \text{by Remark \ref{wWwords}}\\
 &\underset{\eqref{dual3:1}}{=}z_ix_i\cdots x_{m-1}z_mx_myx_1\cdots x_{i+n-m-1}
 \underset{\eqref{dual3:2}}{=}z_ix_i\cdots x_{m-1}x_myz_1x_1\cdots x_{i+n-m-1}\\
 &\underset{\eqref{dual3:1}}{=}(z_ix_i\cdots x_m yx_1\cdots x_{i+n-m-1})z_{i+n-m}
=Wz_{i+n-m} \, \, \text{by Remark \ref{wWwords}}\\
 &=Wz_{i+n} \, \, \text{since indices are taken modulo} \, m.
 \end{aligned}
\end{equation*}
This concludes the proof.
\end{proof}
\begin{remark}\label{CommuteDual}
Since $W=yw$ and $yw=wy$ by Lemma \ref{GarsideDualLemma1}, we have $wW=Ww$.
\end{remark}

\begin{lemma}\label{GarsideDualLemma3}
The element $\Delta\in\mathcal D_*^*(n,m)$ is central. 
\end{lemma}

\begin{proof}
Using Lemmas \ref{GarsideDualLemma1} and \ref{GarsideDualLemma2} and Remark \ref{CommuteDual}, the same argument as in the proof of Proposition \ref{DeltaCentral} shows that for all $i\in [m]$ the elements $x_i$ and $z_i$ commute with $\Delta$. Moreover, using Remark \ref{wWwords} we see that $W=yw$. Since $y$ and $w$ commute by Lemma \ref{GarsideDualLemma1}, we get that $\Delta=w^{m-n}(yw)^n$ commutes with $y$, which concludes the proof.
\end{proof}

\begin{lemma}\label{GarsideDualLemma4}
The element $\Delta\in \mathcal D_*^*(n,m)$ is a Garside element.
\end{lemma}

\begin{proof}
\underline{The sets $\text{Div}_L(\Delta)$and $\text{Div}_R(\Delta)$ are equal:} By Lemma \ref{GarsideDualLemma3}, the element $\Delta$ is central. Now, Proposition \ref{LeftCancDual} implies that $\mathcal D_*^*(n,m)$ is left-cancellative, hence $\text{Div}_L(\Delta)=\text{Div}_R(\Delta)$ by Lemma \ref{CenterDiv}.\\
\underline{The set $\text{Div}(\Delta)$ is finite:} The divisibility of $\mathcal D_*^*(n,m)$ is Noetherian since it admits a homogeneous presentation. Since $\mathcal D_*^*(n,m)$ is finitely generated, every element of $\mathcal D_*^*(n,m)$ (in particular, $\Delta$) admits finitely many left- and right-divisors.\\
\underline{The set $\text{Div}(\Delta)$ generates $\mathcal D_*^*(n,m):$} By definition of $w$, $W$ and $\Delta$ we have $\Delta=w^{m-n}W^n=W^nw^{m-n}$, thus by Remark \ref{wWwords} we have the inclusions $$\{x_1,\dots,x_{m-n+1},z_1,\dots,z_{m-n+1}\}\subset \text{Div}(w)\subset \text{Div}(\Delta)$$ and $$\{x_{m-n+2},\dots,x_m,z_{m-n+2},\dots,z_m,y\}\subset\text{Div}(W)\subset\text{Div}(\Delta).$$
This shows that the divisors of $\Delta$ generate $\mathcal D_*^*(n,m)$, which concludes the proof.
\end{proof}
\noindent
We can now prove Theorem \ref{GarsideDual} (i). 

\begin{proof}[Proof of Theorem \ref{GarsideDual} (i)]
First, divisibility of $\mathcal D_*^*(n,m)$ is Noetherian since it admits a homogeneous presentation. Moreover, combining Propositions \ref{LeftCancDual} and \ref{LeftCancDualOp}, the monoid $\mathcal D_*^*(n,m)$ is cancellative and admits conditional left- and right-lcms. Now, Lemma \ref{GarsideDualLemma4} exhibits a Garside element, which concludes the proof by combining Lemmas \ref{Deltamultiples} and \ref{LcmGcd}.
\end{proof}

\subsection{\texorpdfstring{The monoid $\D_*(n,m)$ is Garside}{}}
Recall from Definition \ref{DefDualMonoid} that given two coprime integers $n\leq m$ with $m\geq 2$, writing $m=qn+r$ with $q\geq 1$ and $0\leq r\leq n-1$, the \textbf{dual monoid} associated to $\B_*(n,m)$ is the monoid $\mathcal D_*(n,m)$ which admits the following presentation:
\begin{equation}\label{DualMonoidz2}\left\langle
       \begin{array}{l|cl}
           x_1,\dots,x_{m}     \,        & \,     x_{i}z_{i+1}=z_ix_i \, \, \forall i=1,\dots,m, \\
     z_1,\dots,z_{m} \,  &\, z_{i+1}x_{i+1}\cdots x_{i+n}=z_ix_i\cdots x_{i+n-1}\, \forall i=1,\dots,m 
                          \end{array}
     \right\rangle,\end{equation}
where indices are taken modulo $m$.\\
The goal of this subsection is to prove Theorem \ref{GarsideDual} (ii).\\
By Remark \ref{dualexists}, setting $n=1$ in Presentation \eqref{DualMonoidz2} yields the dual monoid structure of $\mathcal B(I_2(2m))$ defined in \cite{DualBessis}, which is Garside thanks to \cite{Origines}. Hence, we only need to prove Theorem \ref{GarsideDual} (ii) for $n\geq 2$. For the rest of the section, we assume that $n\geq 2$.\\
Using Presentations \eqref{dual3} and \eqref{DualMonoidz2}, one can see that $\mathcal D_*(n,m)=\mathcal D_*^*(n,m)/\langle y=1 \rangle$, thus it is enough to show that the pair $(\eqref{DualMonoidPres2},y)$ satisfies $(\hyperref[Champagne]{C2})$ and that $\mathcal D_*(n,m)$ is isomorphic to its opposite monoid to apply Proposition \ref{MagicGarsideOp} and deduce Theorem \ref{GarsideDual} (ii).

\begin{lemma}\label{MagicPropertyDualz}
Writing $P$ for Presentation \eqref{DualMonoidPres2}, the pair $(P,y)$ satisfies $(\hyperref[Champagne]{C2})$.
\end{lemma}

\begin{proof}
Write $P=\langle S\,|\, R\rangle$.\\
\underline{The presentation $P$ is homogeneous and satisfies $(\hyperref[Poulet]{C1})$:} This is the content of Lemma \ref{CubeDual}.\\
\underline{For all $(r_1,r_2)\in R$, we have $|r_1|_y=|r_2|_y$ :} In Lines \eqref{DualMonoidPres2:1}, \eqref{DualMonoidPres2:3}, \eqref{DualMonoidPres2:6} and \eqref{DualMonoidPres2:7}, both quantities are equal to $0$. In \eqref{DualMonoidPres2:2} and in Lines \eqref{DualMonoidPres2:4},\eqref{DualMonoidPres2:5} and \eqref{DualMonoidPres2:8}-\eqref{DualMonoidPres2:22}, both quantities are equal to 1.\\
\underline{The monoids with Presentations $\langle S\, | \, R\rangle_y$ and $\langle S\, | \, R\rangle/y$ are isomorphic:} In this context, the monoid with presentation $\langle S\, | \, R\rangle/y$ is isomorphic to $\mathcal D_*^*(n,m)/\langle y=1\rangle=\mathcal D_*(n,m)$ and $\langle S\, | \, R\rangle_y$ is
\begin{subequations}\label{BigDualz}
\begin{align}
    &(1) \,\, \mathrm{Generators}\!:\,  \{x_1,\dots,x_m,z_1,\dots,z_m\};\notag\\
&(2) \,\, \mathrm{Relations}\!: \,\notag\\
  & x_iz_{i+1}=z_ix_i \, \forall i=1,\dots,m-1,\label{BigDualz:1} \\& x_mz_1=z_mx_m,\label{BigDualz:2}\\
  & z_ix_{i}\cdots x_{i+n-1}=z_jx_j\cdots x_{j+n-1}, \, \forall \,1\leq i<j\leq m-n+1, \label{BigDualz:3}  \\
  &z_ix_{i}\cdots x_{m}x_1\cdots x_{i+n-m-1}=z_jx_j\cdots x_{m}x_1\cdots x_{j+n-m-1}, \,\forall \, m-n+2\leq i<j\leq m,\label{BigDualz:4}\\
 &  x_iz_{i+1}x_{i+1}\cdots x_{i+n-1}=z_jx_j\cdots x_{j+n-1},\,\forall \, 1\leq i\neq j\leq m-n+1,\label{BigDualz:5}\\
 & x_iz_{i+1}x_{i+1}\cdots x_{i+n-1}=x_jz_{j+1}x_{j+1}\cdots x_{j+n-1}, \, \forall \, 1\leq i<j\leq m-n+1 ,\label{BigDualz:6}\\
 &x_iz_{i+1}x_{i+1}\cdots x_{i+n-1}=z_jx_j\cdots x_mx_1\cdots x_{j+n-m-1},\, (*_1) ,\label{BigDualz:7}\\
 &x_iz_{i+1}x_{i+1}\cdots x_{i+n-1}=x_jz_{j+1}x_{j+1}\cdots x_{m}x_1\cdots x_{j+n-m-1}, \, (*_2),\label{BigDualz:8}\\
 &x_iz_{i+1}x_{i+1}\cdots x_{i+n-1}=x_mz_1x_1\cdots x_{n-1}, \,\forall \, 1\leq i\leq m-n+1,\label{BigDualz:9}\\
 &x_iz_{i+1}x_{i+1}\cdots x_mx_1\cdots x_{i+n-m-1}=z_jx_j\cdots x_{j+n-1},\, (*_3),\label{BigDualz:10}\\
 &x_iz_{i+1}x_{i+1}\cdots x_mx_1\cdots x_{i+n-m-1}=z_jx_j\cdots x_mx_1\cdots x_{j+n-m-1}, \, (*_4),\label{BigDualz:11}\\
 & x_iz_{i+1}x_{i+1}\cdots x_{m}x_1\cdots x_{i+n-m-1}=x_jz_{j+1}x_{j+1}\cdots x_mx_1\cdots x_{j+n-m-1},\, (*_5),\label{BigDualz:12}\\
 &x_iz_{i+1}x_{i+1}\cdots x_{m}x_1\cdots x_{i+n-m-1}=x_mz_1x_1\cdots x_{n-1}, \, \forall \, m-n+1\leq i\leq m-1,\label{BigDualz:13}\\
  &z_ix_i\cdots x_{i+n-1}=z_jx_j\cdots x_mx_1\cdots x_{j+n-m-1},\,(*_6),\label{BigDualz:14}\\
  &z_ix_i\cdots x_{i+n-1}=x_mz_1x_1\cdots x_{n-1},\,\forall \,  1\leq i\leq m-n+1,\label{BigDualz:15}\\
  &z_ix_i\cdots x_mx_1\cdots x_{i+n-m-1}=x_mz_1x_1\cdots x_{n-1},\,\forall \, m-n+2\leq i\leq m-1.\label{BigDualz:16}\\
   &(*_1)\,\forall \, i\in[1, m-n+1], j\in\llbracket m-n+2,m\rrbracket\notag\\
                    &(*_2)\,\forall \,  i\in[1, m-n+1], j\in \llbracket m-n+2,m-1\rrbracket\notag\\
                    &(*_3)\,\forall \,  i\in \llbracket {m-n+2},{m-1}\rrbracket,j\in [1, m-n+1]\notag\\
                    &(*_4)\,\forall \, i\in \llbracket{m-n+2},{m-1}\rrbracket, j\in \llbracket {m-n+2},m\rrbracket,\, i\neq j \notag\\
                    &(*_5)\,\forall \, m-n+2\leq i<j\leq m-1\notag \\
                    &(*_6)\,\forall \,   i\in [1, m-n+1],j\in \llbracket {m-n+2},m\rrbracket\notag
\end{align}
\end{subequations}
First, observe that the first line of Presentation  \eqref{DualMonoidz2} is equivalent to \eqref{BigDualz:1}+\eqref{BigDualz:2}. Moreover, with redundancies, set of relations \eqref{BigDualz:3}+\eqref{BigDualz:4}+\eqref{BigDualz:14} is equivalent to the set of relations
\begin{equation}\label{D}
    z_ix_i\cdots x_{i+n-1}=z_jx_j\cdots x_{j+n-1}, \,\forall \, 1\leq i<j\leq m,
\end{equation}
where indices are taken modulo $m$.\\
Finally, since $x_iz_{i+1}=z_ix_i$ for all $i=1,\dots,m$, we get that the set of relations $(\eqref{BigDualz:5}-\eqref{BigDualz:13})+\eqref{BigDualz:15}+ \eqref{BigDualz:16}$ is a consequence of Lines \eqref{BigDualz:1} and \eqref{BigDualz:2} combined with the set of relations \eqref{D}.\\
In particular, an alternative presentation for the monoid with Presentation $\langle S\, | \, R \rangle_y$ is
\begin{equation}\label{Dual/z}\left\langle
       \begin{array}{l|cl}
           x_1,\dots,x_{m}     \,        & \,     x_{i}z_{i+1}=z_ix_i \, \, \forall i=1,\dots,m\\
     z_1,\dots,z_{m} \,  &\, z_{i}x_{i}\cdots x_{i+n-1}=z_jx_j\cdots x_{j+n-1}\, \forall i=1,\dots,m 
                          \end{array}
     \right\rangle,\end{equation}
where indices are taken modulo $m$. Since the first lines of Presentations \eqref{DualMonoidz2} and \eqref{Dual/z} are the same and the second lines of Presentations \eqref{DualMonoidz2} and \eqref{Dual/z} are equivalent, Presentation \eqref{Dual/z} is a presentation for both $\langle S\, | \, R \rangle_y$ and $\mathcal D_*(n,m)=\langle S\, | \, R \rangle/y$. This concludes the proof.
\end{proof}

\begin{lemma}\label{MagicPropertyDualzOp}
Let $n,m\in \N_{\geq 2}$ and assume that $n\leq m$ and $n\wedge m=1$. The monoid $\mathcal D_*(n,m)$ is isomorphic to $\mathcal D_*(n,m)^{\mathrm{op}}$.    
\end{lemma}

\begin{proof}
To see this, observe that $\mathcal D_*(n,m)^{\mathrm{op}}$ is isomorphic to $(\mathcal D_*^*(n,m)/\langle y=1\rangle)^{\mathrm{op}}$, which in turn is isomorphic to $\mathcal D_*^*(n,m)^{\mathrm{op}}/\langle y=1\rangle$. Thus, by Lemma \ref{OpDual} a presentation for $\mathcal D_*(n,m)^{\mathrm{op}}$ is
\begin{subequations}\label{BigDualz2}
\begin{align}
    &(1) \,\, \mathrm{Generators}\!:\,  \{x_1,\dots,x_m,z_1,\dots,z_m\};\notag\\
&(2) \,\, \mathrm{Relations}\!: \notag\\
  & z_ix_{i+1}=x_{i+1}z_{i+1} \, \, \forall i=1,\dots,m-1,\label{BigDualz2:1}\\
  &z_{m}x_1=x_{1}z_{1},\label{BigDualz2:2}\\
  &   x_i\cdots x_{i+n-1}z_{i+n-1}=x_{j}\cdots x_{j+n-1}z_{j+n-1}, \,\forall \, 1\leq i<j\leq m-n+1,\label{BigDualz2:3}\\
  & x_i\cdots x_mx_1\cdots x_{i+n-m-1}z_{i+n-m-1}=x_{j}\cdots x_mx_1\cdots x_{j+n-m-1}z_{j+n-m-1}\,,\label{BigDualz2:4}\\
 &  x_{m-n+1}\cdots x_mz_m=x_{m-n+2}\cdots x_mx_1z_1.\label{BigDualz2:5}\\
 & (*)\,\forall \, m-n+2\leq i<j\leq m+1\notag
\end{align}
\end{subequations}
Taking indices modulo $m$, the set of relations \eqref{BigDualz2:1}+\eqref{BigDualz2:2} is equivalent to
\begin{equation}\label{BigDualz3}
    z_ix_{i+1}=x_{i+1}z_{i+1}\, \forall i\in [m].
\end{equation}
Using \eqref{BigDualz3} repeatedly, we have:
\begin{equation}\label{BigDualz2EQ}
\begin{aligned}
x_i\cdots x_{i+n-1}z_{i+n-1}&=z_{i-1}x_i\cdots x_{i+n-1}, \, \forall i=1,\dots,m-n+1,\\
x_i\cdots x_mx_1\cdots x_{i+n-m-1}z_{i+n-m-1}&=z_{i-1}x_i\cdots x_{i+n-1} ,\, \forall i=m-n+2,\dots,m+1,
\end{aligned}
\end{equation}
where indices are taken modulo $m$.\\
Now, Equation \eqref{BigDualz2EQ} allows us to rewrite the sets of relations \eqref{BigDualz2:3}-\eqref{BigDualz2:5} as the set of relations
\begin{equation}\label{BigDualz4}
    z_{i-1}x_i\cdots x_{i+n-1}=z_{j-1}x_j\cdots x_{j+n-1} \, \forall\,  1\leq i<j\leq m,
\end{equation}
where indices are taken modulo $m$.
We get the following presentation for $\mathcal D_*(n,m)^{\mathrm{op}}:$
\begin{equation}\left\langle
       \begin{array}{l|cl}
           x_1,\dots,x_{m}     \,        & \,     z_{i}x_{i+1}=x_{i+1}z_{i+1} \, \, \forall i=1,\dots,m\\
     z_1,\dots,z_{m} \,  &\, z_{i-1}x_{i}\cdots x_{i+n-1}=z_{j-1}x_j\cdots x_{j+n-1}\, \forall 1\leq i<j\leq m 
                          \end{array}
     \right\rangle,\end{equation}
where indices are taken modulo $m$. Relabeling $z_i$ by $z_{i+1}$ for each $i=1,\dots,m-1$ and $z_m$ by $z_1$ yields the presentation
\begin{equation}\left\langle
       \begin{array}{l|cl}
           x_1,\dots,x_{m}     \,        & \,     z_{i}x_{i}=x_{i}z_{i+1} \, \, \forall i=1,\dots,m\\
     z_1,\dots,z_{m} \,  &\, z_{i}x_{i}\cdots x_{i+n-1}=z_{j}x_j\cdots x_{j+n-1}\, \forall 1\leq i<j\leq m 
                          \end{array}
     \right\rangle\end{equation}
for $\mathcal D_*(n,m)^{\mathrm{op}}$ where indices are taken modulo $m$, which is also a presentation of $\mathcal D_*(n,m)$. This concludes the proof.
\end{proof}

\begin{proof}[Proof of Theorem \ref{GarsideDual} (ii)]
This follows by combining Lemmas \ref{MagicPropertyDualz} and \ref{MagicPropertyDualzOp} with Proposition \ref{MagicGarsideOp}.
\end{proof}

%%%%%%%%%%%%%%%%%%%%%%%%%%%%%%%%%%%%%%%%%%%%%%%%%%%%%%%%%%%%%%%%%%%%%%%%%%%%%%%%%%%%%%%%%%%%%%%%%%%%%%%
%%%%%%%%%%%%%%%%%%%%%%%%%%%%%%%%%%%%%%%%%%%%%%%%%%%%%%%%%%%%%%%%%%%%%%%%%%%%%%%%%%%%%%%%%%%%%%%%%%%%%%%
%%%%%%%%%%%%%%%%%%%%%%%%%%%%%%%%%%%%%%%%%%%%%%%%%%%%%%%%%%%%%%%%%%%%%%%%%%%%%%%%%%%%%%%%%%%%%%%%%%%%%%%
%%%%%%%%%%%%%%%%%%%%%%%%%%%%%%%%%%%%%%%%%%%%%%%%%%%%%%%%%%%%%%%%%%%%%%%%%%%%%%%%%%%%%%%%%%%%%%%%%%%%%%%

\subsection{\texorpdfstring{The monoid $\D^*(n,m)$ is Garside}{}}
For this subsection, denote by $Z$ the subset $\{z_1,\dots,z_m\}\subset\mathcal D_*^*(n,m)$.\\
Recall from Definition \ref{DefDualMonoid} that given two coprime integers $2\leq n\leq m$, writing $m=qn+r$ with $q\geq 1$ and $0\leq r\leq n-1$, the \textbf{dual monoid} associated to $\B^*(n,m)$ is the monoid $\mathcal D^*(n,m)$ which admits the following presentation:
\begin{subequations}\label{DualMonoidy2}
\begin{align}
    &(1) \,\, \mathrm{Generators}\!:\,  \{x_1,\dots,x_m,y\};\notag\\
&(2) \,\, \mathrm{Relations}\!:\notag\\
  & x_{i+1}\cdots x_{i+n}=x_i\cdots x_{i+n-1}, \,\forall \,  1\leq i\leq m-n,\label{DualMonoidy2:1}\\
  &   x_{i+1}\cdots x_{m}yx_1\cdots x_{i+n-m}=x_i\cdots x_{m}yx_1\cdots x_{i+n-m-1}, \, \forall \, m-n+1\leq i\leq m.\label{DualMonoidy2:2}
  \end{align}
\end{subequations}
The goal of this subsection is to prove Theorem \ref{GarsideDual} (iii).\\
Using Presentations \eqref{dual3} and \eqref{DualMonoidy2}, one can see that $\mathcal D^*(n,m)=\mathcal D_*^*(n,m)/\langle z_i=1\, \forall i\in [m]\rangle$, thus it is enough to show that the pair $(\eqref{DualMonoidPres2},Z)$ satisfies $(\hyperref[Champagne]{C2})$ and that $\mathcal D^*(n,m)$ is isomorphic to its opposite monoid to apply Proposition \ref{MagicGarsideOp} and prove Theorem \ref{GarsideDual} (iii).

\begin{lemma}\label{MagicDualy}
Writing $P$ for Presentation \eqref{DualMonoidPres2}, the pair $(P,Z)$ satisfies $(\hyperref[Champagne]{C2})$. 
\end{lemma}

\begin{proof}
Write $P=\langle S\,|\, R\rangle$.\\
\underline{The presentation $P$ is homogeneous and satisfies $(\hyperref[Poulet]{C1})$:} This is the content of Lemma \ref{CubeDual}.\\
\underline{For all $(r_1,r_2)\in R$, we have $\sum_{z\in Z}\lambda(z)|r_1|_z=\sum_{z\in Z}\lambda(z)|r_2|_z$:} In all cases, both sides of the equality are equal to 1.\\
\underline{The monoids with presentations $\langle S\, |\, R\rangle_Z$ and $\langle S\, | \, R\rangle/Z$ are isomorphic:} In this context, the monoid with presentation $\langle S\, | \, R\rangle/Z$ is isomorphic to $\mathcal D_*^*(n,m)/\langle z_i=1, \forall i\in [m]\rangle=\mathcal D^*(n,m)$ and $\langle S\, |\, R\rangle_Z$ is 
\begin{subequations}\label{BigDualy}
\begin{align}
    &(1) \,\, \mathrm{Generators}\!:\,  \{x_1,\dots,x_m,y\};\notag\\
&(2) \,\, \mathrm{Relations}\!:\notag\\
  & x_i\cdots x_{i+n-1}=x_j\cdots x_{j+n-1}, \, \forall \, 1\leq i<j\leq m-n+1, \label{BigDualy:1}\\
  &   x_i\cdots x_{i+n-1}y=x_j\cdots x_{m}yx_1\cdots x_{j+n-m-1}, \, (*),\label{BigDualy:2}\\
  &x_i\cdots x_{i+n-1}y=x_myx_1\cdots x_{n-1}, \,\forall \, 1\leq i\leq m-n+1, \label{BigDualy:3}\\
  &x_i\cdots x_{i+n-1}y=yx_1\cdots x_{n}, \,\forall \, 1\leq i\leq m-n+1,\label{BigDualy:4}\\
  &x_i\cdots x_{m}yx_1\cdots x_{i+n-m-1}=x_j\cdots x_myx_1\cdots x_{j+n-m-1},\,(**),\label{BigDualy:5}\\
  &x_i\cdots x_{m}yx_1\cdots x_{i+n-m-1}=x_myx_1\cdots x_{n-1},\,\forall \,m-n+2\leq i\leq m-1,\label{BigDualy:6}\\
  &x_i\cdots x_{m}yx_1\cdots x_{i+n-m-1}=yx_1\cdots x_n ,\,\forall \, m-n+2\leq i\leq m-1,\label{BigDualy:7}\\
  &yx_1\cdots x_{n}=x_myx_1\cdots x_{n-1}.\label{BigDualy:8}\\
  &(*)\,\forall \, i\in[m-n+1] , j\in\llbracket m-n+2,m-1\rrbracket\notag\\
  &(**) \, \forall \,m-n+2\leq i<j\leq m-1\notag
\end{align}
\end{subequations}
First, the set of relations \eqref{BigDualy:2}-\eqref{BigDualy:4} is equivalent the set of relations 
\begin{equation}\label{BigDualyReduce1}
x_i\cdots x_{i+n-1}y=x_j\cdots x_myx_1\cdots x_{j+n-m-1}, \, \forall 1\leq i\leq m-n+1,\,  m-n+2\leq j\leq m+1.
\end{equation}
Moreover, the set of relations \eqref{BigDualy:6}+\eqref{BigDualy:7} imply Line \eqref{BigDualy:8} and the set of relations \eqref{BigDualy:5}-\eqref{BigDualy:7} is equivalent to the set of relations 
\begin{equation}\label{BigDualyReduce2}
x_i\cdots x_myx_1\cdots x_{i+n-m-1}=x_j\cdots x_myx_1\cdots x_{j+n-m-1}, \, \forall m-n+1\leq i<j \leq m+1.
\end{equation}
This shows that $\langle S\, | \, R \rangle_Z$ admits the following presentation:\\
\begin{subequations}\label{BigDualy2}
\begin{align}
    &(1) \,\, \mathrm{Generators}\!:\,  \{x_1,\dots,x_m,y\};\notag\\
&(2) \,\, \mathrm{Relations}\!:\notag \\
  &x_{i}\cdots x_{i+n-1}=x_j\cdots x_{j+n-1}, \, \forall \,1\leq i<j\leq m-n+1,\label{BigDualy2:1}\\
  &   x_i\cdots x_{i+n-1}y=x_j\cdots x_myx_1\cdots x_{j+n-m-1}, \, (*),\label{BigDualy2:2}\\
  & x_{i}\cdots x_{m}yx_1\cdots x_{i+n-m-1}=x_j\cdots x_{m}yx_1\cdots x_{j+n-m-1},  \,\forall \, m-n+1\leq i<j\leq m+1.\label{BigDualy2:3}\\
  &(*) \, \forall \,i\in[m-n+1],j\in\llbracket m-n+2,m+1\rrbracket\notag
\end{align}
\end{subequations}
Finally, the set of relations \eqref{BigDualy2:1} and \eqref{BigDualy2:3} implies \eqref{BigDualy2:2}:
For all $i\in [m-n+1]$ and $j\in \llbracket m-n+2,m+1\rrbracket$, we have
\begin{equation*}
\begin{aligned}
x_{i}\cdots x_{i+n-1}y&\underset{\eqref{BigDualy2:1}}{=} x_{m-n+1}\cdots x_{m}y\underset{\eqref{BigDualy2:3}}{=}x_j\cdots x_myx_1\cdots x_{j+n-m-1}.
\end{aligned}
\end{equation*}
We thus get that $\langle S\, | \, R \rangle_Z$ is isomorphic to $\mathcal D^*(n,m)=\langle S \, | \, R \rangle /Z$, which concludes the proof.
\end{proof}

\begin{lemma}\label{MagicPropertyDualyOp}
Let $n,m\in \N_{\geq 2}$ and assume that $n\leq m$ and $n\wedge m=1$. The monoid $\mathcal D^*(n,m)$ is isomorphic to $\mathcal D^*(n,m)^{\mathrm{op}}$.
\end{lemma}

\begin{proof}
Observe that $\mathcal D^*(n,m)^{\mathrm{op}}$ is isomorphic to $(\mathcal D^*(n,m)/\langle z=1\, \forall z\in Z \rangle)^{\mathrm{op}}$, which in turn is isomorphic to $\mathcal D_*^*(n,m)^{\mathrm{op}}/\langle z=1\,  \forall z\in Z\rangle$. Thus, by Lemma \ref{OpDual} a presentation for $\mathcal D^*(n,m)^{\mathrm{op}}$ is
\begin{subequations}\label{BigDualy3}
\begin{align}
    &(1) \,\, \mathrm{Generators}\!:\,  \{x_1,\dots,x_m,y\};\notag\\
&(2) \,\, \mathrm{Relations}\!: \notag\\
  & x_i\cdots x_{n+i-1}=x_{j}\cdots x_{n+j-1},\,\forall \, 1\leq i<j\leq m-n+1, \label{BigDualy3:1}\\ 
  &   x_i\cdots x_myx_1\cdots x_{i+n-m-1}=x_{j}\cdots x_myx_1\cdots x_{j+n-m-1},\,(*),\label{BigDualy3:2}\\
  &x_{m-n+1}\cdots x_my=x_{m-n+2}\cdots x_myx_1.\label{BigDualy3:3}\\
  &(*)\, \forall \, m-n+2\leq i<j\leq m+1\notag
\end{align}
\end{subequations}
From the set of relations \eqref{BigDualy3:2}+\eqref{BigDualy3:3}, for all $j\in \llbracket m-n+2,m+1\rrbracket$ we get
\begin{equation}\label{lastdualy}
\begin{aligned}
x_{m-n+1}\cdots x_my&\underset{\eqref{BigDualy3:3}}{=}x_{m-n+2}\cdots x_myx_1\underset{\eqref{BigDualy3:2}}{=}x_j\cdots x_myx_1\cdots x_{j+n-m-1}.
\end{aligned}
\end{equation}
Now, the set of relations given by \eqref{BigDualy3:2}+\eqref{BigDualy3:3}+\eqref{lastdualy} is equivalent to the set of relations 
$$x_i\cdots x_myx_1\cdots x_{i+n-m-1}=x_j\cdots x_myx_1\dots x_{j+n-m-1} ,\, \forall m-n+1 \leq i<j\leq m+1,$$
thus $\mathcal D^*(n,m)^{\mathrm{op}}$ is isomorphic to $\mathcal D^*(n,m)$.
\end{proof}

\begin{proof}[Proof of Theorem \ref{GarsideDual} (iii)]
This follows by combining Lemmas \ref{MagicDualy}, \ref{MagicPropertyDualyOp} and Proposition \ref{MagicGarsideOp}.
\end{proof}

\end{document}